\documentclass[11pt,reqno]{amsart}
\usepackage{amsmath,amsthm,amsfonts,amssymb,mathrsfs,bm,graphicx,stmaryrd}

\usepackage{mathtools}
\usepackage{dsfont}
\usepackage{multicol}
\usepackage[colorlinks=true,linkcolor=blue]{hyperref}
\hypersetup{bookmarksdepth=1}
\usepackage{enumitem}
\usepackage{bbm}
\usepackage{mathrsfs}
\usepackage{versions}
\usepackage{subcaption}
\usepackage{tikz}
\usepackage[letterpaper,hmargin=1.0in,vmargin=1.0in]{geometry}
\parskip 	\smallskipamount

\newcommand{\E}{\mathbb{E}}
\renewcommand{\P}{\mathbb{P}}

\newcommand{\cf}{\mathcal{F}}
\newcommand{\cd}{\mathcal{D}}

\newcommand{\ce}{\mathcal{E}}
\newcommand{\cg}{\mathcal{G}}
\newcommand{\cj}{\mathcal{J}}
\newcommand{\cm}{\mathcal{M}}
\newcommand{\ct}{\mathcal{T}}
\newcommand{\cs}{\mathcal{S}}
\newcommand{\Z}{\mathbb{Z}}
\newcommand{\N}{\mathbb{N}}
\newcommand{\R}{\mathbb{R}}
\newcommand*{\wih}{\widehat}
\newcommand{\cl}{\mathcal{L}}
\newcommand*{\bo}{\boldsymbol}
\newcommand*{\wt}{\widetilde}
\newcommand*{\un}{\underline}
\newcommand*{\ov}{\overline}
\newcommand{\ca}{\mathcal{A}}
\newcommand{\cb}{\mathcal{B}}
\newcommand{\cc}{\mathcal{C}}
\newcommand{\ci}{\mathcal{I}}
\newcommand{\ck}{\mathcal{K}}
\newcommand{\cn}{\mathcal{N}}
\newcommand{\red}{\textcolor{red}}
\newcommand{\vep}{\varepsilon}
\newcommand*{\msf}{\mathsf}
\newcommand{\green}{\textcolor{green}}
\newcommand*{\m}{\mathbbm}
\def\l{\left}
\def\r{\right}
\def\<{\langle}
\def\>{\rangle}
\DeclareMathOperator{\sech}{sech}

\newtheorem{theorem}{Theorem}[section]
\newtheorem{corollary}[theorem]{Corollary}
\newtheorem{lemma}[theorem]{Lemma}
\newtheorem{proposition}[theorem]{Proposition}

\newtheorem{claim}[theorem]{Claim}

\theoremstyle{remark}
\newtheorem{remark}{\bf Remark}

\numberwithin{equation}{section}

\newcommand{\ptoline}[1]{
    \begin{tikzpicture}[baseline=-0.6ex,thick,scale=0.2]
        \draw (0,0) -- (1,0) -- (0,1) -- cycle; % Draw triangle
            \end{tikzpicture}^{#1} % Include the exponent
}
\newcommand{\linetop}[1]{
    \begin{tikzpicture}[baseline=-0.6ex,thick,scale=0.2]
        \draw (1,1) -- (1,0) -- (0,1) -- cycle; % Draw triangle
            \end{tikzpicture}^{#1} % Include the exponent
}

\newcommand{\ptohalfspace}[1]{
    \begin{tikzpicture}[baseline=-0.6ex,thick,scale=0.2]
        \draw (1,1) -- (0,1) -- (0,0) -- cycle; % Draw triangle
            \end{tikzpicture}^{#1} % Include the exponent
}

\newcommand{\ptop}[1]{
    \begin{tikzpicture}[baseline=-0.6ex,thick,scale=0.2]
        \draw (1,1) -- (0,1) -- (0,0)-- (1,0) -- cycle; % Draw triangle
            \end{tikzpicture}^{#1} % Include the exponent
}
\usepackage{palatino}

\title[Tail estimates in $\beta$ ensembles and last passage percolation]{Optimal tail estimates in $\beta$-ensembles and applications to last passage percolation}
\author[Baslingker, Basu, Bhattacharjee, Krishnapur]{Jnaneshwar Baslingker, Riddhipratim Basu, \\ Sudeshna Bhattacharjee, Manjunath Krishnapur}

\begin{document}
\begin{abstract}
    Hermite and Laguerre $\beta$-ensembles are important and well studied models in random matrix theory with special cases $\beta=1,2,4$ corresponding to eigenvalues of classical random matrix ensembles. It is well known that the largest eigenvalues in these, under appropriate scaling, converge weakly  to the Tracy-Widom $\beta$ distribution whose distribution function $F_{\beta}$ has asymptotics given by $1-F_{\beta}(x)=\exp\l(-\frac{2\beta}{3}(1+o(1))x^{3/2}\r)$ as $x\to \infty$ and $F_{\beta}(x)=\exp\l(-\frac{\beta}{24}(1+o(1))|x|^3\r)$ as $x\to -\infty$. Although tail estimates for the largest eigenvalues with correct exponents have been proved \cite{LR10, BGHK21} for the pre-limiting models, estimates with matching constants had not so far been established for general $\beta$; even in the exactly solvable cases, some of the bounds were missing. In this paper, we prove upper and lower moderate deviation estimates for both tails with matching constants.
    
    We illustrate the usefulness of these estimates by considering certain questions in planar exponential last passage percolation (LPP), a well-studied model in the KPZ universality class in which certain statistics have same distributions as largest eigenvalues in Laguerre $\beta$-ensembles (for $\beta=1,2,4$). Using our estimates in conjunction with a combination of old and new results on the LPP geometry, we obtain three laws of iterated logarithm including one which settles a conjecture from \cite{L18}. We expect that the sharp moderate deviation estimates will find many further applications in LPP problems and beyond. 
    
\end{abstract}
\address{Jnaneshwar Baslingker,  Department of Mathematics, Indian Institute of Science, Bangalore, India}
\email{jnaneshwarb@iisc.ac.in}
\address{Riddhipratim Basu, International Centre for Theoretical Sciences, Tata Institute of Fundamental Research, Bangalore, India} 
\email{rbasu@icts.res.in}
\address{Sudeshna Bhattacharjee, Department of Mathematics, Indian Institute of Science, Bangalore, India}
\email{sudeshnab@iisc.ac.in }
\address{Manjunath Krishnapur, Department of Mathematics, Indian Institute of Science, Bangalore, India}
\email{manju@iisc.ac.in}
\maketitle

\tableofcontents

\section{Introduction and statements of main results} 
The study of random matrices with i.i.d.\ (real or complex) Gaussian entries go back to almost 100 years. The point processes formed by the eigenvalues of such matrices (GUE/GOE) and their counterparts for the corresponding Wishart matrices (LUE/LOE) are important and extensively studied. It is known that these eigenvalue ensembles are special cases for a more general point process called the $\beta$-ensembles (for a real parameter $\beta>0$), which, starting from the work of Dyson in the 60s, have become an important topic of study over the last decades \cite{RRV, ES07, DE, KJ, BF97, KN04} (see also \cite{AGZ09} for an overview of $\beta$-ensembles). Hermite, Laguerre and Jacobi $\beta$-ensembles are the most well-studied; in this paper we shall only consider the Hermite and Laguerre $\beta$-ensembles. The special cases $\beta=1,2,4$ are exactly solvable and correspond to the eigenvalues of some matrix models (G(O/U/S)E respectively for the Hermite case and L(O/U/S)E respectively for the Laguerre case) and in general are more well studied (see \cite{Dy63, TW94}). Following the influential work of Dumitriu and Edelman \cite{DE} where tridiagonal matrix models for Hermite and Laguerre $\beta$-ensembles were introduced, these sparse matrix models have proved to be extremely useful for the study of these ensembles. 

% \marginpar{\tiny Some of these references  are not particularly for beta ensembles? }

In this paper, our topic of interest is the largest eigenvalue of Hermite and Laguerre $\beta$-ensembles. It was shown in a seminal paper of Ram\'irez, Rider and Vir\'ag \cite{RRV} (following heuristic approach outlined in \cite{ES07}) that the largest eigenvalue of both the Hermite and Laguerre $\beta$-ensembles (under appropriate normalisation) converge weakly to a distribution known as the Tracy-Widom $\beta$ distribution, which contains the classical (GUE) Tracy-Widom distribution as the $\beta=2$ special case. The tail behaviour of the Tracy-Widom distribution was also determined in \cite{RRV} where they showed that its distribution function $F_{\beta}$ satisfies $1-F_{\beta}(x)=\exp(-\frac{2\beta}{3}(1+o(1))x^{3/2})$ as $x\to \infty$ and $F_{\beta}(x)=\exp(-\frac{\beta}{24}(1+o(1))|x|^3)$ as $x\to -\infty$. It is natural to expect, and important in many applications, that the pre-limiting (centred and scaled) largest eigenvalues also show similar tail behaviour. Although such estimates have been shown with the correct exponents $x^{3/2}$ for the upper tail and $x^{3}$ for the lower tail in \cite{LR10} (see also \cite{BGHK21} for the lower bound for lower tail in the Laguerre case), tail estimates with the optimal constants have been missing so far, although there have been some upper tail estimates for the cases $\beta=1,2$ (see e.g.\ \cite{PZ17,L07,BX23, BBF22} and the references therein). In this paper (Theorems \ref{t: right tail upper bound for beta geq 2}, \ref{t: rtub beta geq 1}, \ref{t: rtlb}, \ref{t: ltub}, \ref{t: ltlb}) we establish the optimal moderate deviation estimates for general Hermite and Laguerre $\beta$-ensembles for different parameter ranges.

Tracy-Widom distributions are extremely important even beyond random matrix theory; in the theory of the Kardar-Parisi-Zhang (KPZ) universality class, 
the classical Tracy-Widom laws play a role akin to what the Gaussian distribution plays in theory of sums of independent random variables. They arise as universal scaling limit in (1+1)-dimensional KPZ models. In fact, there are also pre-limiting correspondences leading to distributional equalities between finite size statistics in certain KPZ models and largest eigenvalues in $\beta$-ensembles. It is therefore expected that our optimal tail estimates for the latter will find numerous applications in the study of the KPZ class. We illustrate this with an application: using the correspondence between the largest eigenvalues of the Laguerre $\beta$-ensembles (for $\beta=1,2,4$) and various passage times in exponential last passage percolation (LPP) on $\Z^2$, a canonical exactly solvable model in the KPZ class, together with the optimal tail estimates we prove for the former, we prove (see Theorem \ref{t: LILs}) three laws of iterated logarithm (LIL) for the passage times in exponential LPP. This includes a law of iterated logarithm for point-to-point passage times in LPP which was conjectured in \cite{L18, BGHK21} where, using the weaker estimates from \cite{LR10,BGHK21} a law of iterated logarithm was established without explicitly evaluating the limiting constants. Our result provides a complete resolution to the problem. We point out here that these results (especially the $\liminf$ bounds in the laws of iterated logarithm) are not immediate corollaries of the sharp moderate deviation estimates (as was already noted in \cite{L18,BGHK21}) and require geometric inputs from the last passage percolation. Our proofs use a combination of a few known results about the geometry of geodesics in exponential LPP models as well as a couple of new results which might be of independent interest. We shall also mention a number of other applications of our moderate deviation works which will be covered in upcoming works. 

We now move towards formulating precise statements of our results.

%\textcolor{green}{Should we sell the moderate deviations as the result and give LIL as application? Or LIL as the result and moderate deviations as the tool?}

\subsection{Moderate deviations in Hermite and Laguerre $\beta$-ensembles}
%\textcolor{red}{joint density of the $\beta$-ensembles, precise result about convergence to TW$_\beta$, Ledoux-Rider moderate deviation result. statement of the main result (one theorem for Hermite and one for Laguerre), limitations, outline of the proof (i.e., Tridiagonal matrices and relationship with the Ledoux-Rider and RRV arguments), also a discussion of cases where the exact constant was known before, like upper bound of upper tails in $\beta=1,2$, Ganguly-Hegde paper, related Meixner and Plancherel cases, Riemann-Hilbert methods.}

For any real $\beta>0$, consider the probability density function of $\lambda_1\geq \lambda_2\geq\dots\geq\lambda_n\in \mathbb{R}$ given by 
\begin{align}\label{Hermite density}
\mathbb{P}_{n}^{\beta}\l(\lambda_1,\lambda_2,\dots,\lambda_n\r)=\frac{1}{Z_n^{\beta}}\prod_{k=1}^{n}e^{-\frac{\beta}{4}\lambda_k^2}\times\prod\limits_{j<k}\left|\lambda_j-\lambda_k\right| ^{\beta}. 
\end{align}

For $\beta=1,2$ or $4$, this distribution corresponds to the joint density of eigenvalues of $n\times n$ Gaussian orthogonal, unitary or symplectic matrices, G(O/U/S)E of (for general $\beta$ this density is also known as Hermite $\beta$-ensemble) random matrix theory.

For $m\in\mathbb{R}$ with $m>n-1$, consider the following joint density of $\lambda_1\geq\lambda_2\geq \dots\geq\lambda_n\geq 0 $,
\begin{align}\label{Laguerre density}
\mathbb{P}_{n,m}^{\beta}\l(\lambda_1,\lambda_2,\dots,\lambda_n\r)=\frac{1}{Z_{n,m}^{\beta}}\prod_{k=1}^{n}\lambda_k^{\frac{\beta}{2}\l(m-n+1\r)-1}e^{-\frac{\beta}{2}\lambda_k}\times \prod\limits_{j<k}\left|\lambda_j-\lambda_k\right| ^{\beta}.
\end{align}

When $m$ is an integer and $\beta=1,2$ or $4$, this distribution corresponds to the joint density of eigenvalues of  Laguerre orthogonal, unitary or symplectic matrices, L(O/U/S)E (for general $\beta$ the density is known as Laguerre $\beta$-ensemble) of random matrix theory. 

Despite being the focus of research for many years, there were no matrix models for general Hermite $\beta$ and Laguerre $\beta$-ensembles, until the work of 
Dumitriu and Edelman \cite{DE} who gave tridiagonal matrix models for \eqref{Hermite density} and \eqref{Laguerre density}. Define 
\begin{align}\label{eq: tridiag matrix}
 H_{n,\beta}:=\frac{1}{\sqrt{\beta}}\begin{bmatrix}
X_1 & Y_1 & 0 &\cdots 0\\
Y_1 & X_2 & Y_2 & \cdots 0\\
0 & Y_2 & X_3 &\cdots 0\\
\vdots & \vdots & \vdots &\vdots
\end{bmatrix}_{n\times n}
    \hspace{2cm}
B_{n,m,\beta}:=\frac{1}{\sqrt{\beta}}\begin{bmatrix}
D_1 & 0 & 0 &\cdots 0\\
C_1 & D_2 & 0 & \cdots 0\\
0 & C_2 & D_3 &\cdots 0\\
\vdots & \vdots & \vdots &\vdots
\end{bmatrix}_{n\times n}.
\end{align}
$H_{n,\beta}$ %(known as Hermite-$\beta$ ensemble)
is the symmetric tridiagonal matrix with $X_i\sim N\l(0,2\r)$ and $Y_i\sim \chi_{\beta\l(n-i\r)}$ (here $Y\sim \chi_{\nu}$ means $Y>0$ and $Y^2\sim \chi_{\nu}^2$) and the entries are independent up to symmetry.
 $B_{n,m, \beta}$ is the bidiagonal matrix  with $D_i\sim \chi_{\beta\l(m-i+1\r)}$ and $C_i\sim \chi_{\beta\l(n-i\r)}$, where $m\in \mathbb{R}$ and $m+1>n$ and all the entries are independent.
Dumitriu and Edelman showed that the $n$ eigenvalues of $H_{n,\beta}$ have the joint law given by \eqref{Hermite density}
 and the eigenvalues of $L_{n,m,\beta}=B^t_{n,m, \beta}B_{n,m, \beta}$ %(Laguerre-$\beta$ ensemble) 
 have the joint density given by \eqref{Laguerre density}. In particular, if  $\lambda_1^{\l(n,H,\beta\r)}$ and $\lambda_1^{\l(n,m,L,\beta\r)}$ denote the largest point $\lambda_1$ in \eqref{Hermite density} and \eqref{Laguerre density} respectively, then
 \begin{align}\label{eq: Dimtriu Edelman result}
     \lambda_{\max}\l(H_{n,\beta}\r)\overset{d}{=}\lambda_1^{\l(n,H,\beta\r)}
     \quad   \text{and}\quad\lambda_{\text{max}}\l(L_{n,m, \beta}\r)\overset{d}{=}\lambda_1^{\l(n,m,L,\beta\r)}
 \end{align} 
where $\lambda_{\max}(A)$ denotes the largest eigenvalue of a real symmetric matrix $A$.

Ram\'irez-Rider-Vir\'ag \cite{RRV} showed that the largest eigenvalues of the above random tridiagonal matrices, after suitable shifting and scaling, converge to the point process of eigenvalues of a second order stochastic differential operator called the stochastic Airy operator ($\mbox{SAO}_{\beta}$). In particular, using \eqref{eq: Dimtriu Edelman result}, their results show that
\begin{align}\label{eq: Hermite TW convergence}
    \l(\frac{\lambda_1^{\l(n,H,\beta\r)}}{\sqrt{ n}}-2\r)n^{2/3}&\implies \mbox{TW}_{\beta},
    \\ \label{eq: Laguerre TW convergence}
    {\l(\sqrt{m n}\r)^{1/3}}{\l(\sqrt{m}+\sqrt{n}\r)^{2/3}}&\l(\frac{\lambda_1^{\l(n, m, L,\beta\r)}}{\l(\sqrt{m}+\sqrt{n}\r)^{2}}-1\r)\implies \mbox{TW}_{\beta}.
\end{align}
They also found various properties of  $\mbox{TW}_{\beta}$ (read as Tracy-Widom~$\beta$). Of particular interest to us is the tail behaviour which they found. As $a\rightarrow\infty$,
\begin{align}\label{eq: TW right tail bounds}
    \P\l(TW_{\beta}>a\r)&=\exp\l(-\frac{2}{3}\beta a^{3/2}\l(1+o(1)\r)\r),\\\label{eq: TW left tail bounds}
        \P\l(TW_{\beta}<-a\r)&=\exp\l(-\frac{1}{24}\beta a^{3}\l(1+o(1)\r)\r).
\end{align}
Prior to their work, Baik-Buckingham-DiFranco~\cite{BBD08} had already found much more  precise asymptotics for $\beta=1,2,4$. Although our prime application to last passage percolation only requires these special values of $\beta$, our methods to prove tail bounds for the largest point in Gaussian and Laguerre ensembles are  very close to the methods of \cite{RRV}. They used the  variational characterisation of the smallest eigenvalue of $\mbox{SAO}_{\beta}$, as well as an associated diffusion. In \cite{DV13}, they sharpened the  right tail estimates for general $\beta$ by finding lower order corrections.

The $TW_{\beta}$ distributions (and their pre-limiting variants) also appear in many contexts outside random matrices. For instance, these laws describe the fluctuations of last passage times in percolation \cite{KJ}, longest increasing subsequence of a random permutation \cite{BDJ98} and the current in simple exclusion process \cite{KJ, TW09}. It is therefore natural, and important in many applications, to investigate the tail bounds for $ \lambda_1^{\l(n,H,\beta\r)}$ and $ \lambda_1^{\l(n,m,L,\beta\r)}$. 

%\green{I have added these references and known results. If there are any inaccuracies or if I have missed other references, please let me know. (J.B.)}

We summarise here some of the known results about the finite size tail estimates for statistics of interest that converge to the Tracy-Widom distribution. It was considered in \cite{L04,L07} that the right tail inequality for GUE (also LUE) may be shown to follow from results of Johansson \cite{KJ} which uses sub-additivity arguments and large deviation asymptotics. The optimal right tail bounds in the case of GUE, up to the constants in the exponential tails, were obtained in \cite{PZ17}. For GOE, \cite{BX23} gives the right tail upper bound with optimal constants in a certain parameter regime. The right tail upper bounds, with the optimal constant for last passage model that corresponds to LOE was obtained in \cite{BBF22}. By the superposition-decimation procedure \cite{FR01}, similar bounds for GOE can be also be found (see \cite{Ledoux09}). Precise upper tail estimates are also available for some related positive temperature models in the KPZ universality class, see e.g. \cite{DG23, GH22}.

Even for the exactly solvable cases, the left tail bounds are generally less accessible, often requiring more difficult Riemann-Hilbert analysis. The left tail inequalities for last passage times in planar last passage percolation with geometric weights, which corresponds to largest element in the Meixner ensemble, are proven in \cite{BDMMZ01} using Riemann-Hilbert methods. One might expect a similar analysis to establish these estimates for GUE and LUE as well. In case of passage times in Poissonian LPP, which is obtained by a suitable limit of largest element in Meixner ensemble, optimal estimate for upper tail and lower tail was shown in \cite{LM01,LMS02} using Riemann-Hilbert methods. Coming to the case of Hermite and Laguerre ensembles with general $\beta$, Ledoux and Rider\cite{LR10} provide the tail bounds with the correct exponents $3/2$ in the upper tail and $3$ in the lower tail for general $\beta$ (but not the optimal constant $2\beta/3$ and $\beta/24$). Upper bounds for the right tail and left tails for Hermite and Laguerre ensembles were obtained in \cite[Theorem 1,2]{LR10}. But it was remarked that the lower bounds were obtained only in the Hermite ensemble. The left tail lower bound for Laguerre ensemble (general $\beta$) was obtained in \cite[Theorem 2]{BGHK21}. In \cite{GH23} the 3/2 and 3 exponents (without the optimal constants) for the right tail and left tails was conditionally obtained in last passage percolation models via bootstrapping and geodesic geometry.   

We state the main theorems now. Theorems~\ref{t: right tail upper bound for beta geq 2}-\ref{t: ltlb} essentially assert that the pre-limiting random variables in \eqref{eq: Hermite TW convergence} and \eqref{eq: Laguerre TW convergence} have roughly the same tail probabilities as the  Tracy-Widom distributions. The range of the tail for which we are able to prove this are different for the left and right tails, for the upper and lower bounds etc. All the theorems are valid in the range $1\ll t\ll n^{\delta}$ for some $\delta>0$. The optimal $\delta$ one can expect is $\frac23$, beyond which it passes into the large deviation regime.

\begin{comment}Using the large deviation asymptotics, Brownian LPP connection (for GUE), sub-additivity arguments (for LUE), we get upper bounds for right tails for $\beta=2$ case. Then using a stochastic domination result \cite{JB23}, we obtain similar bounds for all $\beta\geq 2$. For the upper bounds, for right tail of Hermite-$\beta$ and Laguerre-$\beta$ ensembles we have the following theorems.
\end{comment}

\begin{theorem}
\label{t: right tail upper bound for beta geq 2}
For the right tail upper bound we have
    \begin{enumerate}[label=(\roman*), font=\normalfont] 
        \item For any $\beta\geq 2$ and $\varepsilon>0$, we can choose $n_{\varepsilon,\beta} \in \N$ and $t_{\vep,\beta},\gamma\l(\varepsilon\r)>0$ such that for all $n\geq n_{\varepsilon,\beta}$ and for all $t$ with $t_{\vep,\beta} \leq t\leq \gamma(\varepsilon)n^{2/3}$  the following holds 
\begin{align*}
    \P \l( \lambda_1^{\l(n,H,\beta\r)} \geq 2 \sqrt{n}+tn^{-1/6}\r) \leq   \exp\l({-\frac{2\beta}{3}t^{3/2}\l(1-\varepsilon\r)}\r).
\end{align*}
\item For any $\beta \geq 2$ and $\varepsilon>0$ and $M>0$ there exist $n_{\varepsilon,\beta,M} \in \N$ and $t_{\vep,\beta,M},\gamma(\varepsilon,M)>0$ such that for any pair of natural numbers $\l(m,n\r)$ with $m,n \geq 0$ and $ m+1>n$ such that $ \frac{m}{n} \leq M$ and for any $t_{\vep,\beta}\leq t \leq \gamma(\varepsilon,M)n^{2/3}$ and for all $n \geq n_{\varepsilon,\beta,M}$ we have 
        \[
        \P \l(\lambda_{1}^{(n,m,L, \beta)} \geq \l(\sqrt{m}+\sqrt{n} \r)^2+t \l(\sqrt{m n} \r)^{-1/3}\l(\sqrt{m}+\sqrt{n} \r)^{4/3} \r) \leq \exp\l({-\frac{2\beta}{3}t^{3/2}\l(1-\varepsilon\r)}\r).
        \]

% For all $\delta>0$ there exists $N_\delta \in \mathbb{N}, C_\beta>0$ such that for all $N \geq N_{\delta}$ and all $\beta \geq 1$ with $\beta N \in \mathbb{N}$ and for any  $t$ such that $\l(\log N/\delta\r)^{2/3}\leq t\leq N^{\frac{2}{3}-\delta}$, we have
%         \[
%         \P \l( \lambda_1^{\l(N,H,\beta\r)} \geq 2 \sqrt{N}+t N^{-1/6}\r) \leq C_\beta e^{-\l(\frac{2\beta}{3}-\delta\r)t^{3/2}}.
%         \]
\begin{comment}
        \item For all $\varepsilon>0$ there exists $n_\varepsilon \in \mathbb{N}, C_\beta>0, \gamma(\varepsilon)>0$ such that for all $n \geq n_\varepsilon$ and for all $\beta \geq 1$ with 
        %$N \beta \in 2 \mathbb{N}$ 
        and all $t \leq \gamma(\varepsilon)n^{\frac{2}{3}}$ with $\frac{\log(n)}{t^{3/2}} \leq \gamma(\varepsilon)$ we have 
        \[
        \P \l( \lambda_1^{\l(n,H,\beta\r)} \geq 2 \sqrt{n}+t n^{-1/6}\r) \leq C_\beta e^{\l(-\frac{2\beta}{3}+\varepsilon \r)t^{3/2}}.
        \]
        \end{comment}
    \end{enumerate}
\end{theorem}
For $1 \leq \beta <2$ also we have the following theorem which is less satisfactory.

\begin{theorem}
    \label{t: rtub beta geq 1}
    For the right tail upper bound in $1 \leq \beta <2$ range we have
    \begin{enumerate}[label=(\roman*), font=\normalfont]
        \item For any $1 \leq \beta <2$ and $\vep>0,$ we can choose $n_{\vep,\beta} \in \N$ and $t_{\vep,\beta}>0$ such that for all $n \geq n_{\vep,\beta}$ and for all $t$ with $t_{\vep,\beta}\leq t \leq n^{1/4}$
        \[
        \P \l( \lambda_1^{\l(n,H,\beta\r)} \geq 2 \sqrt{n}+tn^{-1/6}\r) \leq  \exp\l({-\frac{2\beta}{3}t^{3/2}\l(1-\varepsilon\r)}\r).
        \]
        \item For any $1 \leq \beta <2$ and $\vep>0$ there exists $n_{\vep,\beta}\in \N $ and $t_{\vep,\beta}, \gamma(\vep)>0$ such that for all $n \geq n_{\vep,\beta}$ and for all $t_{\vep,\beta} \leq t \leq \gamma(\vep)n^{2/3}$ 
        \[
        \P \left(\lambda_1^{(n,n+1,L,\beta)} \geq 4n+t2^{4/3}n^{1/3} \right) \leq \exp\l({-\frac{2\beta}{3}t^{3/2}\l(1-\varepsilon\r)}\r).
        \]
    \end{enumerate}
\end{theorem}
\begin{remark}
    In the above theorem, $(i)$ is true for the range of $t$ up to $n^{2/3}$. Using the moment recursion in \cite{Ledoux09}, we can upgrade the right tail upper bound result of \cite{BX23} for $\beta=1$ case up to the range $t\leq \gamma(\varepsilon) n^{2/3}$. Now applying our proof methods, we can extend $(i)$ of the above theorem up to $t \leq \gamma(\varepsilon)n^{2/3}$.    For the Laguerre case we were unable to obtain the right tail upper bound for general $(m,n)$ in the range $1 \leq \beta <2.$ But our methods could be used for $(m,n)$ of the form $(n+c,n)$ for any constant $c \geq 0.$ See Remark \ref{remark: Laguerre n+c}.
    \end{remark}
\begin{theorem}
\label{t: rtlb}
For the right tail lower bound we have
    \begin{enumerate}[label=(\roman*), font=\normalfont]
        \item For any $\beta>0$ and any small $\varepsilon>0$ there exists $n_{\vep,\beta} \in \N$ and $t_{\vep,\beta},\gamma(\varepsilon)>0$ such that for all $n \geq n_{\vep,\beta}$ and $ t_{\vep,\beta} \leq t \leq \gamma(\vep)n^{2/3}$ we have that 
        \[
        \P \l(\lambda_1^{(n,H,\beta)} \geq 2\sqrt{n}+tn^{-1/6} \r) \geq \exp\l({-\frac{2\beta}{3}t^{3/2}\l(1+\varepsilon\r)}\r).
        \]
        \item For any $\beta>0$, any small $\vep>0$ with any $M>0$ there exists $n_{\vep,\beta,M} \in \N$ and $t_{\vep,\beta,M},\gamma(\vep,M)>0$ such that for any $(m,n)$ with $m> n-1$ and $ \frac m n \leq M$, for any $n \geq n_{\vep,\beta,M}$ and $t_{\vep,\beta,M} \leq t \leq \gamma(\vep,M)n^{{2/3}}$ we have 
        \[
        \P \l(\lambda_1^{(n,m,L,\beta)} \geq \l(\sqrt{m}+\sqrt{n} \r)^2+t\l(\sqrt{m n}\r)^{-1/3}\l(\sqrt{m}+\sqrt{n}\r)^{4/3} \r) \geq \exp \l({-\frac{2\beta}{3}t^{3/2}\l(1+\varepsilon\r)}\r).
        \]
    \end{enumerate}
\end{theorem} 
% \red{For the left tail bounds, should we opt for less optimal range of $t$ such as $n^{1/10}$, so that qualifier $\delta$ is removed and Theorem $1.4$ and $1.5$ are combined in to one? (JB)}
\begin{theorem}
    \label{t: ltub} 
    For the left tail upper bound we have
    \begin{enumerate}[label=(\roman*), font=\normalfont]
        \item For any $\beta>0$ and any $\vep,\delta>0$ small enough there exist $n_{\vep,\beta,\delta} \in \N$ and $t_{\vep,\beta,\delta}>0$ such that for any $n \geq n_{\vep,\beta,\delta}$ and $t_{\vep,\beta,\delta} \leq t \leq n^{1/6-\delta}$ we have
        \[
        \P \l(\lambda_{1}^{(n,H,\beta)} \leq 2\sqrt{n} -tn^{-1/6}\r) \leq \exp \l({-\frac{\beta}{24}t^{3}\l(1-\varepsilon\r)}\r).
        \]
        \item For any $\beta>0$, any small $\vep,\delta>0$ and any $M>0$ there exist $n_{\vep,\beta,\delta,M} \in \N$ and $t_{\vep,\beta,\delta,M}>0$ such that for any $n \geq n_{\vep,\beta,\delta,M}$ and $t_{\vep,\beta,\delta,M} \leq t \leq n^{1/6-\delta}$ and for any $(m,n)$ with $m > n-1$ with $ \frac m n \leq M$ we have 
        \[
        \P \l(\lambda_{1}^{(n,m,L,\beta)} \leq \l(\sqrt{m} +\sqrt{n}\r)^2-t\l(\sqrt{m n} \r)^{-1/3} \l(\sqrt{m} +\sqrt{n}\r)^{4/3}\r) \leq \exp \l({-\frac{\beta}{24}t^{3}\l(1-\varepsilon\r)}\r).
        \]
    \end{enumerate}
\end{theorem}

\begin{comment}For the left tail lower bound, we use the well known three-term recursion satisfied by vectors in tridiagonal matrices and compute lower bounds for the event that there are no sign changes in the vector, along with spiked random matrix results \cite{BV12} to prove, 
\end{comment}

\begin{theorem}
    \label{t: ltlb}
    For the left tail lower bound we have
    \begin{enumerate}[label=(\roman*), font=\normalfont]
        \item For any $\beta>0$ and any $\vep,\delta>0$ small enough there exist $n_{\vep,\beta,\delta} \in \N$ and $t_{\vep,\beta,\delta}>0$ such that for any $n \geq n_{\vep,\beta,\delta}$ and $t_{\vep,\beta,\delta} \leq t \leq n^{1/9-\delta}$ we have
        \[
        \P \l(\lambda_{1}^{(n,H,\beta)} \leq 2\sqrt{n} -tn^{-1/6}\r) \geq \exp \l({-\frac{\beta}{24}t^{3}\l(1+\varepsilon\r)} \r).
        \]
        \item For any $\beta>0$, any small $\vep,\delta>0$ and any $M>0$  there exist $n_{\vep,\beta,\delta,M} \in \N$ and $t_{\vep,\beta,\delta,M}$ such that for any $n \geq n_{\vep,\beta,\delta,M}$ and $t_{\vep,\beta,\delta,M} \leq t \leq n^{1/9-\delta}$ and for any $(m,n)$ with $m>n-1$ and $\frac mn \leq M$ we have 
        \[
        \P \l(\lambda_{1}^{(n,m,L,\beta)} \leq \l(\sqrt{m} +\sqrt{n}\r)^2-t\l(\sqrt{m n} \r)^{-1/3} \l(\sqrt{m} +\sqrt{n}\r)^{4/3}\r) \geq \exp \l({-\frac{\beta}{24}t^{3}\l(1+\varepsilon\r)}\r).
        \]
    \end{enumerate}
\end{theorem}

\subsection{Application: Laws of iterated logarithm in exponential LPP}
\label{subs: application to LPP}
As mentioned already, we now move towards describing some applications of the optimal tail estimates to problems in exponential LPP. We begin by defining the model.

\begin{comment}\textcolor{red}{
Definition of LPP, point-to-point, point-to-line/line-to-point and half space.} 
\end{comment}

\begin{comment}\textcolor{blue}{Later we have used $\tau$ as stopping time. So, here should we use $\zeta_v$ to denote the exponential random variables? (SB)}
\end{comment}First we fix two important notations. Throughout this article for $r \in \Z $ we denote the vertex $(r,r) \in \Z^2$ by $ \bo r$ and $\cl_{r}$ is the line consisting of all vertices $(v_1,v_2) \in \Z^2$ such that $v_1+v_2=r$. We assign i.i.d.\ random variables $\{\zeta_{v}\}_{v \in \mathbb{Z}^2}$ to each vertex of $\mathbb{Z}^2$, where $\zeta_{v}$'s are distributed as Exp(1). Let $u,v \in \mathbb{Z}^2$ be such that $u \leq v$ (i.e., if $u=(u_1,u_2),v=(v_1,v_2)$ then $u_1 \leq v_1$ and $u_2 \leq v_2$). For an up-right path $\gamma$ between $u$ and $v$ we define $\ell(\gamma)$, the \textit{passage time of $\gamma$} and $T_{u,v}$, the \textit{last passage time between $u$ and $v$} by 
\begin{displaymath}
    \ell (\gamma) := \sum_{w \in \gamma} \zeta_{w}.
    \begin{comment}\footnote{Our definition of passage time of a path excludes both t-he initial and final vertex. This is slightly different from standard definitions where we include both the vertices. We will work with this definition because of some technical reasons which will be clear from the proofs. Also note that excluding initial and final vertices does not change the geodesics. All the tail estimate results for last passage times that we will use in this paper hold for this definition as well.}
    \end{comment}
\end{displaymath} 
and
\begin{displaymath}
T_{u,v}:=\max\{ \ell(\gamma): \gamma \text{ is an up-right path from } u \text{ to } v\}
\end{displaymath}
respectively.

\begin{comment}Note that in the definition above we have excluded the initial vertex that is common to all paths. We can give a definition including the initial vertex and excluding the final vertex as well. This does not change the estimates for the last passage times. We will use all these definitions throughout this article. For $\gamma$ and $u,v$ as above
\[
  \un{\ell(\gamma)}:=\ell(\gamma)+\tau_u-\tau_v.
  %\ov{\ell(\gamma)}:=\ell(\gamma)+\tau_u-\tau_v.
\]
and 
\[
\un{T_{u,v}}:=T_{u,v}+\tau_u-\tau_v. 
%\ov{T_{u,v}}:=T_{u,v}+\tau_u-\tau_v.
\]
\end{comment}

\begin{comment}\textcolor{red}{I suggest at this point define passage times with both endpoints included.}
\end{comment}

Clearly, as the number of up-right paths between $u$ and $v$ is finite, the maximum is always attained. Between any two (ordered) points $u,v \in \mathbb{Z}^2$, maximum attaining paths are called \textit{point-to-point geodesics}. As $\zeta_{v}$ has a continuous distribution, almost surely, between any two points $u,v \in \mathbb{Z}^2$, there exists a unique geodesic denoted by $\Gamma_{u,v}$. Further, $T_n^{\ptop{}}$ denotes the last passage time from $\boldsymbol{0}$ to $\boldsymbol{n}$ and $\Gamma_n^{\ptop{}}$ will denote the a.s. unique geodesic between $\bo 0$ to $\bo n$.\\
We can also define \textit{point-to-line} and \textit{line-to-point passage times}: for $v=(v_1,v_2) \in \Z^2$ and $r \in \Z,$ with $r \leq v_1+v_2,$ the line-to-point last passage time is defined by
\[
T_{\cl_r,v}=\max_{u \in \cl_r}T_{u,v}.
\]
By same argument as before unique \textit{line-to-point} geodesics exist a.s. and are denoted by $\Gamma_{\cl_r,u}$. In particular, $T^{{\linetop{}}}_n$ will denote the last passage time from $\cl_{0}$ to $\bo n$ and $\Gamma_n^{\linetop{}}$ will denote the a.s. unique geodesic between $\cl_{0}$ and $\bo n$.
\begin{comment}
\[
T^{{\linetop{}}}_n=\max_{u \in \cl_0}T_{u, \boldsymbol{n}}.
\]
\end{comment}
%Similarly, we can define another variant
%\[
%\ov{\wih{T}_n}=\max_{u \in \cl_0}\ov{T_{u, \boldsymbol{n}}}.
%\]
Point-to-line last passage times are defined as follows. Let $u=(u_1,u_2) \in \Z^2$ and $r \in \Z$ be such that $u_1+u_2 \leq r$ we define 
\[
T_{u,\cl_r}:=\max_{v \in \cl_r}T_{u,v}.
\]
The a.s. unique \textit{point-to-line geodesic} will be denoted by $\Gamma_{u, \cl_r}$. Last passage time from $\bo 0$ to $\cl_{2n}$ will be denoted by $T_n^{{\ptoline{}}}$ and 
\begin{comment}
\[
   T_n^{\ptoline{}}=\max_{u \in \cl_{2n}}T_{\bo 0, u}.
\]
\end{comment}
the a.s. unique point-to-line geodesic from $\bo 0$ to $\cl_{2n}$ is denoted by $\Gamma_n^{\ptoline{}}.$ Observe that by symmetry, $T^{{\ptoline{}}}_n$ and $T^{{\linetop{}}}_n$ have the same distribution for any $n$, but the sequence $\left(T^{{\ptoline{}}}_n\right)_{n\ge 1}$ is not identically distributed to the sequence $\left(T^{{\linetop{}}}_n\right)_{n\ge 1}$. Therefore, the results for the two sequences point-to-line and line-to-point (in Theorem \ref{t: LILs}) cannot be deduced from one another, although the end results are identical.

\begin{comment}\textcolor{red}{I still think it makes sense to keep the weights non-zero on the diagonal}
\end{comment}

We consider another last passage model on $\mathbb{Z}^2$, having i.i.d $\zeta_v\sim$ Exp$(1)$ random weights on vertices, with symmetry across the $x=y$ line that is $\zeta_{(i,j)}=\zeta_{(j,i)}$ for $i \neq j$. Define
\begin{align*}
T^{\ptohalfspace{}}_{n}:=\max\limits_{\gamma}\ell(\gamma),
\end{align*}
where $\gamma$ are the up-right oriented paths in $\mathbb{Z}^2$ from $\bo{0}\in\mathbb{Z}^2$ to $\bo{n}\in\mathbb{Z}^2$ and $\ell(\gamma)=\sum\limits_{v\in \gamma}\zeta_v$. Note that, by symmetry $T^{\ptohalfspace{}}_{n}$ remains same if we consider the paths $\gamma$ from $\bo{0}, \bo{n}$ and $\gamma$ stays above the diagonal (i.e., if $(v_1,v_2) \in \gamma$, then $v_1 \leq v_2$). Hence, the above model is equivalent to the model where we have collection of i.i.d. Exp(1) random variables $\zeta_v$ for all $v \in \Z^2$ such that $v_1,v_2 \geq 0$ and $v_1 \leq v_2$ and all the other vertices have zero weights. This model is known as the \textit{half space model}. We can similarly define \textit{point-to-point passage time in the half space model} by considering maximum over the up-right paths. $T_{u,v}^{\ptohalfspace{}}$ and $\Gamma_{u,v}^{\ptohalfspace{}}$ will denote the passage time and the geodesics in this model respectively. As before, $T_n^{\ptohalfspace{}}$ and $\Gamma_n^{\ptohalfspace{}}$ will denote the passage time and geodesic between $\bo 0$ and $\bo n$ respectively in this model.
 
 \begin{comment}\textcolor{red}{Do we need these here?}
 \end{comment}
 \begin{comment}From almost sure uniqueness of geodesics, together with planarity, it follows that geodesics are ordered (i.e., two geodesics with pairs of starting and ending points having the same order in the spatial co-ordinate, remain ordered throughout their journeys and cannot cross).\\
 \end{comment}

 It follows immediately from the definitions that one can couple $T^{\linetop{}}_{n},T^{\ptoline{}}_{n},T^{\ptop{}}_{n}, T^{\ptohalfspace{}}_{n}$ such that
\begin{equation}
\label{eq: ordering}
T^{\linetop{}}_{n}\left(T^{\ptoline{}}_{n}\right) \geq T^{\ptop{}}_{n} \geq T^{\ptohalfspace{}}_{n}. 
\end{equation}

%\textcolor{red}{Correspondence with the LUE/LOE/LSE. Tail estimates for the passage times. Motivate the scaling for the LIL, Ledoux's and BGHK result and conjectures.}

We have the following distributional equalities. 
\begin{equation}
    \begin{aligned}\label{eq: distributional equalities}
T^{\ptop{}}_{n-1}\overset{d}{=}&\lambda_{\text{max}}\left( B^t_{n,n, 2}B_{n,n, 2}\right) \quad\mbox{\cite[Proposition 1.4]{KJ}}.\\
T^{\linetop{}}_{n-1}\overset{d}{=}T^{\ptoline{}}_{n-1}\overset{d}{=}&\frac{1}{2}\lambda_{\text{max}}\left( B^t_{2n-1,2n, 1}B_{2n-1,2n, 1}\right) \quad \mbox{\cite[Proposition 1.3]{BGHK21}}.\\
T^{\ptohalfspace{}}_{{2n-2}}\overset{d}{=}&2\lambda_{\text{max}}\left( B^t_{n,n-\frac{1}{2}, 4}B_{n,n-\frac{1}{2}, 4}\right)\quad\mbox{\cite[Proposition 7]{JB23}}.
\end{aligned}
\end{equation}

By the distributional equalities above and \eqref{eq: Laguerre TW convergence} we have,
\begin{equation}
    \begin{aligned}\label{eq: LPP TW convergence}
    \frac{T^{\ptop{}}_{{n}}-4n}{2^{4/3}n^{1/3}}\implies TW_2.\\
    \frac{T^{\ptoline{}}_{{n}}-4n}{2^{4/3}n^{1/3}}\  \mbox{  and   }\  \frac{T^{\linetop{}}_{{n}}-4n}{2^{4/3}n^{1/3}}\implies \frac{1}{2^{2/3}}TW_1.\\
\frac{T^{\ptohalfspace{}}_{{2n}}-8n}{2^{4/3}(2n)^{1/3}}\implies 2^{2/3}TW_4.
\end{aligned}
\end{equation}
%\textcolor{red}{next, write a theorem with the tail estimates for the passage times as well as the Tracy-Widom convergence}
From the theorems in Section $1.1$ and above distributional equalities we immediately get the following tail estimates for the passage times. The following theorem holds for all the 4 passage times $T_n^{\ptop{}},T_n^{\ptoline{}},T_n^{\linetop{}},T_n^{\ptohalfspace{}}$. We will use a common notation $T_n^*$ to denote all of them.
\begin{theorem}
\label{t: passage time tail estimates}
Let  $T_n^*$ be any of $T_n^{\linetop{}}$ or $T_n^{\ptoline{}}$ or $T_n^{\ptop{}}$ or  $T_n^{\ptohalfspace{}}$ and correspondingly let $\beta$ be $1$ or $1$ or $2$ or $4$ respectively.
    \begin{enumerate}[label=(\roman*), font=\normalfont]
        \item For any $\varepsilon>0$ there exist $n_{\varepsilon} \in \N$ and $\gamma(\varepsilon),t_{\varepsilon}>0$ such that for all $t_{\varepsilon} \leq t \leq \gamma(\varepsilon)n^{2/3}$ and for all $n \geq n_{\varepsilon}$ we have
        \[
      \exp\l({-\frac{4}{3}t^{3/2}\l(1+\varepsilon\r)}\r)\leq\P \l(T_{n}^* \geq 4n+2^{4/3}tn^{1/3} \r) \leq \exp\l({-\frac{4}{3}t^{3/2}\l(1-\varepsilon\r)}\r).\]
      % \item For any $\varepsilon>0$ there exist $n_{\varepsilon} \in \N$ and $\gamma(\vep),t_{\varepsilon}>0$ such that for all $ t_{\varepsilon}\leq t \leq \gamma(\vep)n^{\frac{2}{3}}$ and for all $n \geq n_{\varepsilon}$ we have 
      %   \[
      % \P \l(T_{n}^* \geq 4n+2^{4/3}tn^{1/3} \r) \geq \exp\l({-\frac{4}{3}t^{3/2}\l(1+\varepsilon\r)}\r).
      % \]
        \item For any any small $\vep>0$ there exist $n_{\vep}\in \N$ and $t_{\vep}$ such that for any $n \geq n_{\vep}$ and $t_{\vep} \leq t \leq n^{1/10}$, we have  
        \[
         \exp \l(-{\frac{1}{6\beta}t^{3}\l(1+\varepsilon\r)} \r)\leq\P \l(T_n^*\leq 4n-2^{4/3}tn^{1/3}\r) \leq \exp \l(-{\frac{1}{6\beta}t^{3}\l(1-\varepsilon\r)}\r). 
         \]
       % \item For any small $\vep,\delta>0$ there exist $n_{\vep,\delta} \in \N$ and $t_{\vep,\delta}$ such that for any $n \geq n_{\vep,\delta}$ and $t_{\vep,\delta} \leq t \leq n^{1/9-\delta}$, we have 
       %  \[
       %  \P \l(T_n^*\leq 4n-2^{4/3}tn^{1/3}\r)  \geq \exp \l(-{\frac{1}{6\beta}t^{3}\l(1+\varepsilon\r)} \r),
       %  \]
%        where $\beta=1,2,4$ respectively for $T_n^*=T_n^{\linetop{}}\l(T_n^{\ptoline{}}\r), T_n^{\ptop{}},T_n^{\ptohalfspace{}}$.
        \end{enumerate}
\end{theorem}

\begin{remark}
    In the proofs later we will use some variants of the last passage times, where we will sometimes remove the initial vertices or final vertices from the definition of passage times. It can be checked easily that the above tail estimates are true for these variants as well with the same constants.
\end{remark}

%\begin{theorem}
%    \label{t: point to point passage time tail estimate}
 %   \begin{enumerate}[label=(\roman*),font=\normalfont]
%        \item For any $\varepsilon>0$ there exist $n_{\varepsilon},\gamma(\varepsilon)>0$ such that for all $1 \leq t \leq \gamma(\varepsilon)n^{2/3}$ and for all $n \geq n_{\varepsilon}$ we have 
   %     \[
   %   \P \l(T_{n}^{\ptoline{}} \geq 4n+2^{4/3}tn^{1/3} \r) \leq \exp\l(-\l(\frac{4}{3}-\varepsilon\r)t^{3/2}\r).
        %\]
       % \item For any $\delta,\varepsilon>0$ there exist $n_{\delta,\varepsilon},t_{\delta,\varepsilon}>0$ such that for all $ t_{\delta,\varepsilon}\leq t \leq n^{\frac{2}{3}-\delta}$ and for all $n \geq n_{\delta,\varepsilon}$ we have 
      %  \[
     % \P \l(T_{n}^{\ptoline{}} \geq 4n+2^{4/3}tn^{1/3} \r) \geq \exp\l(-\l(\frac{4}{3}+\varepsilon\r)t^{3/2}\r).
       % \]
       % \item For any any small $\vep,\delta>0$ there exist $n_{\vep,\beta,\delta},t_{\vep,\beta,\delta}$ such that for any $n \geq n_{\vep,\beta,\delta}$ and $t_{\vep,\delta} \leq t \leq n^{1/6-\delta}$, we have  
       % \[
       % \P \l(T_n^{\ptop{}}\leq 4n-t2^{4/3}n^{1/3}\r) \leq \exp \l(-\l(\frac{1}{12}-\vep\r)t^3 \r).
       % \]
       % \item For any small $\vep,\delta>0$ there exist $n_{\vep,\delta},t_{\vep,\delta}$ such that for any $n \geq n_{\vep,\delta}$ and $t_{\vep,\delta} \leq t \leq n^{1/9-\delta}$, we have 
       % \[
       % \P \l(T_n^{\ptop{}}\leq 4n-t2^{4/3}n^{1/3}\r) \geq \exp \l(-\l(\frac{1}{12}+\vep\r)t^3 \r).
       % \]
      %  \end{enumerate}
%\end{theorem}

The following discussion applies equally to all four of the sequence of random variables $T^{\linetop{}}_{n}, T^{\ptoline{}}_{n}$, $ T^{\ptop{}}_{n}$ and $T^{\ptohalfspace{}}_{n}$, but for concreteness we shall stick to the case of $T^{\ptop{}}_{n}$ for now. Consider the coupled sequence of random variables $\left\{T^{\ptop{}}_{n}\right\}_{n\ge 0}$. Given the weak convergence result above \eqref{eq: LPP TW convergence}, one natural question is to consider the $\limsup$ and $\liminf$ of the sequence $\biggl\{\frac{T^{\ptop{}}_{n}-4n}{n^{1/3}}\biggr\}$. By a zero-one law, one might expect that there would be some functions $g_{+}(n)$ and $g_{-}(n)$ diverging as $n\to \infty$ such that 
$$\limsup_{n\to\infty} \frac{T^{\ptop{}}_{n}-4n}{g_{+}(n)} \quad \mbox{ and } \quad \liminf_{n\to\infty} \frac{T^{\ptop{}}_{n}-4n}{g_{-}(n)}$$
% \textcolor{red}{make it $g_{\pm}$ consistent with later}
 converge almost surely to non-zero finite constants. The question then is to find $g_{+}(n)$ and $g_{-}(n)$ and the limiting constants.

The question of studying limit laws for extrema of coupled sequence of random variables converging to the Tracy-Widom distribution goes back to 
Paquette and Zeitouni \cite{PZ17} who considered the largest eigenvalue $\lambda_{n}$ of the top $n\times n$ submatrix of an infinite GUE matrix which is centered and scaled to converge to the GUE Tracy-Widom distribution. They established a ``law of fractional logarithm" showing that there is an explicit constant $c\in (0,\infty)$ such that 
$$\limsup_{n\to \infty} \frac{\lambda_n}{(\log n)^{2/3}}=c$$
almost surely whereas 
$$\liminf_{n\to \infty} \frac{\lambda_n}{(\log n)^{1/3}}\in (c_1,c_2)$$
almost surely for some negative constants $c_1,c_2$ \footnote {One of the obstacles which made Paquette-Zeitouni $\liminf$ result weaker was the lack of optimal tail estimates in the lower tail which are now available in Theorems \ref{t: ltub}, \ref{t: ltlb}.}. Inspired by \cite{PZ17}, Ledoux \cite{L18} asked the same question in the LPP set-up as described above. 

Even though $T^{\ptop{}}_{n}$s are not sums of i.i.d. random variables, they possess a super-additive structure (i.e., $T^{\ptop{}}_{{m+n}}\ge T^{\ptop{}}_{{n}}+T_{\mathbf{n},\mathbf{m+n}}$ and the second term has the same law as $T^{\ptop{}}_{m}$ and is independent of $T^{\ptop{}}_{n}$\footnote{Notice that the inequality is not exactly true because of the fact that the weight of the vertex $\mathbf{n}$ is counted twice, but as we shall see this can be ignored.}). Therefore it is instructive to compare this problem with the classical law of iterated logarithms for the simple symmetric random walk on integers. Recall that for a simple symmetric random walk $S_{n}$ we have, almost surely, 
$$\limsup_{n\to \infty} \frac{S_{n}}{n^{1/2}\sqrt{2\log\log n}}=1,\quad \limsup_{n\to \infty} \frac{S_{n}}{n^{1/2}\sqrt{2\log\log n}}=-1.$$
The factor $\sqrt{2\log \log n}$ can be explained as follows. To have $\max\limits_{\frac{n}{2}\le m \le n} \frac{S_{m}}{\sqrt{m}}>x$, one roughly needs to have $\frac{S_{m}}{\sqrt{m}}>x$ at one of roughly $\log n$ many approximately independent scales, e.g., at $m\in \{n,n/2,n/4,\ldots \}$. By Gaussian approximation of Binomial, each of these have probability $\approx e^{-x^2/2}$ and we get $x=\sqrt{2\log \log n}$ by setting $e^{-x^2/2}=\frac{1}{\log n}$. Looking at the tail estimates for $T^{\ptop{}}_{n}$, the above heuristic would suggest that for $T^{\ptop{}}_{n}$ we should still have a law of iterated logarithms, with scaling for the $\limsup$ and $\liminf$ given by $(\log \log n)^{2/3}$ and $(\log \log n)^{1/3}$ respectively and the constants would be determined in the same fashion by the constants in front of the exponents in Theorem \ref{t: passage time tail estimates}. This is what Ledoux argued in \cite{L18} and using the weaker tail estimates from \cite{LR10}, he showed  that 
\begin{align}\label{eq: Ledoux LPP result}
  \limsup_{n\to\infty} \frac{T^{\ptop{}}_{n}-4n}{2^{4/3}n^{1/3}(\log \log n)^{2/3}}=C_1  
\end{align}
almost surely. In fact, assuming the results in Theorems \ref{t: right tail upper bound for beta geq 2}, \ref{t: rtlb}, \cite{L18} also showed that $C_1=\left(\frac{3}{4}\right)^{2/3}$. For the lower tail, \cite{L18} showed that 
$$\liminf_{n\to\infty} \frac{T^{\ptop{}}_{n}-4n}{2^{4/3}n^{1/3}(\log \log n)^{1/3}}>-\infty$$
which was later improved in \cite{BGHK21} to show 
\begin{align}\label{eq: BGHK LPP result}
  \liminf_{n\to\infty} \frac{T^{\ptop{}}_{n}-4n}{2^{4/3}n^{1/3}(\log \log n)^{1/3}}=-C_2  
\end{align}
almost surely for $C_2\in (0,\infty)$. In fact, it was conjectured in \cite{BGHK21,L18} that $C_2=(12)^{1/3}$. Our next theorem proves this conjecture and also establishes similar results for point-to-line and half space exponential LPP.

\begin{comment}\textcolor{red}{split it into the three theorem like we discussed. Also point out that even though point-to-line and line-to-point have same marginals, the couplings are different and therefore separate statements are warranted.}\textcolor{green}{The last remark is made after the definition of these two passage times.}
\end{comment}

We denote the two functions defining the upper and lower functions for the law of iterated logarithm by
$$
g_{+}(n)=2^{\frac43}n^{\frac13}(\log \log n)^{\frac23} \qquad\mbox{ and }\qquad  g_-(n)=2^{\frac43}n^{\frac13}(\log \log n)^{\frac13}.
$$
With this notation, we present the result.
\begin{comment}
\textcolor{green}{Lots of expressions will look much less ugly if we use $g_{\pm}$ instead of writing out the expressions. I have only changed it in the  theorem statements for now, in case we disagree and wish to go back. [M.K.]}
\end{comment}
\begin{theorem}
\label{t: LILs}
Let $T_n^*$ be any of $T_n^{\linetop{}}$ or $T_n^{\ptoline{}}$ or $T_n^{\ptop{}}$ or $T_n^{\ptohalfspace{}}$ and correspondingly let $\beta$ be $1$ or $1$ or $2$ or $4$ respectively. We have 
\begin{enumerate}[label=(\roman*), font=\normalfont]
    \item $\limsup \limits_{n \rightarrow \infty} \frac{T_n^*-4n}{g_+(n)}=\l(\frac 34 \r)^{2/3} a.s. $
    \item $\liminf \limits_{n \rightarrow \infty} \frac{T_n^*-4n}{g_-(n)}=-(6\beta)^{1/3} a.s.$
\end{enumerate}
    
\end{theorem}
\begin{comment}
\begin{theorem} 
    \label{thm: ptop lil} For $T_n^{\ptop{}}$ we have
    \begin{enumerate}[label=(\roman*), font=\normalfont]
    \item $\limsup\limits_{n \rightarrow \infty}\frac{T^{\ptop{}}_{n}-4n}{g_+(n)} = \left (\frac34 \right)^{2/3} \text{ a.s. }$
    \item $\liminf\limits_{n \rightarrow \infty}\frac{T^{\ptop{}}_{n}-4n}{g_-(n)}=-(12)^{1/3}\text{ a.s. }$
    \end{enumerate}
    \end{theorem}
    \begin{theorem}
    \label{t: ptol lil}
    For $T_n^{\ptoline{}}$ and $T_n^{\linetop{}}$ we have
    \begin{enumerate}[label=(\roman*), font=\normalfont]
    \item $\limsup\limits_{n \rightarrow \infty}\frac{T^{\ptoline{}}_{n}-4n}{g_+(n)}= \limsup\limits_{n \rightarrow \infty}\frac{T^{\linetop{}}_{n}-4n}{g_+(n)}= \left (\frac34 \right)^{2/3} \text{ a.s. }$
    \item $\liminf\limits_{n \rightarrow \infty}\frac{T^{\ptoline{}}_{n}-4n}{g_-(n)}= \liminf\limits_{n \rightarrow \infty}\frac{T^{\linetop{}}_{n}-4n}{g_-(n)}=-(6)^{1/3}\text{ a.s. }$
    \end{enumerate}
    \end{theorem}
    \end{comment}

    \begin{remark}
        As mentioned already, although $T^{\linetop{}}_{n}$ and $T^{\ptoline{}}_{n}$ have the same distribution \eqref{eq: distributional equalities}, the coupled sequences $\left\{T^{\linetop{}}_{n}\right\}_{n\geq 0}$ and $\l\{T^{\ptoline{}}_{n}\r\}_{n\geq 0}$ do not have the same distribution. Hence the equalities in the $\limsup$ and $\liminf$ statements for the point-to-line and line-to-point case in Theorem \ref{t: LILs} are non trivial.
        \end{remark}

We note that, in the theorems above, once we have the moderate deviation estimates, the result for the $\limsup$ is relatively easy and it mainly uses the super-additivity of the last passage time. This was already observed in \cite{L18} for the point-to-point case. However, even assuming the optimal estimate for the lower tail, the $\liminf$ results are non-trivial. We use the last passage percolation representation in a crucial way. Our arguments for the $\liminf$ results require a combination of old and new results especially about the geometry of geodesics in exponential LPP. 

Note also one may be able to prove several variants of Theorem \ref{t: LILs} by following our arguments. For example, instead of looking at the passage times along the diagonal direction, one might also consider the sequence of passage times along any non-axial direction. Since our estimates for LUE are available as long as $m/n$ bounded, one should be able to prove a version of Theorem \ref{t: LILs} in the set-up as well. One can also consider different exactly solvable models of planar last passage percolation. Sharp moderate deviation estimates for Poissonian LPP \cite{LM01,LMS02} and Geometric LPP \cite{BDMMZ01} were already available. Using the correspondence between the $\lambda_1^{(n,H,2)}$ and Brownian LPP (see e.g.\ \cite{BY01,OY02})
our results also provide the sharp tail estimates for Brownian LPP as well. So the appropriate variants of Theorem \ref{t: LILs} for all these models should follow from our arguments as well.

Following the works \cite{L18,BGHK21}, there has been some interest in investigating LIL for other KPZ models and law of iterated logarithms has been derived for certain limiting models (where one point estimates were already available). In particular, using Gibbsian line ensemble techniques, LIL for KPZ equation (both $\limsup$ and $\liminf$) was derived in \cite{DG23} \begin{comment}{\cite{DG23}}\end{comment}.
LIL for the upper tail was derived in the set-up of KPZ fixed point in \cite{DGL22}. However, as far as we are aware, our results are the first LIL results for a lattice model in the KPZ class. 

Finally, we point out that while we illustrate the usefulness of our sharp moderate deviation results by establishing the LIL for exponential LPP, this is by no means the only application of the optimal moderate deviation estimates in the LPP set-up. We expect these estimates to be useful in many other problems in exponential LPP and related areas, just as the results of \cite{LR10} has been extensively applied in a large number of LPP problems (see e.g.\ \cite{BBB23, BBF22, BGZ21} and many references therein). We describe a couple of such applications which are contained in forthcoming works. In \cite{BB24}, using the moderate deviation estimates proved here together with the fact that Airy$_1$ and Airy$_2$ processes arise as scaling limits for point-to-point and point-to-line exponential LPP, second and third author of the current paper obtain limit theorems for the extrema of the processes, extending and complementing a recent work of Pu \cite{P23}. In another upcoming work \cite{AB24}, the second author together with Agarwal use our moderate deviation results to obtain sharp estimates for the probability that the point-to-line geodesic in exponential LPP ends at a given point. We anticipate that these deviation estimates will be useful in questions beyond the LPP set-up as well.

%\textcolor{red}{related results: LIL for KPZ equation etc., possible extensions: Brownian LPP using Hermite, along different directions. Possible applications beyond LPP, Paquette-Zeitouni result.}

\subsection{Outline of the proofs.} We shall now describe briefly the ideas that go into the proofs of the main results. We will start with the tail estimates for largest eigenvalues in \eqref{eq: Hermite TW convergence}, \eqref{eq: Laguerre TW convergence}. Except for the right tail upper bounds (which are established using a stochastic comparison principle described below), our proofs, at a high level, take a route similar to the one taken in \cite{RRV}, with many additional technical difficulties due to working with a finite matrix. Two of these three estimates (lower bound for the right tail and upper bound for the left tail) can be established by constructing suitable test functions, while the left tail lower bound uses a discretized version of the diffusion employed in \cite{RRV}. In all these three cases, we modify $H_{n,\beta}$ and $L_{n,m,\beta}$ by replacing $\chi_{\beta(n-i)}$ random variables with $\sqrt{\beta\l(n-i\r)}+\frac{\zeta_i}{\sqrt{2}}$ for all $i\leq p$, where $\zeta_i$ are i.i.d.\ $N\l(0,1\r)$ random variables. We denote these modified matrices as $\wih{H}_{n,\beta,p}$ and $\wih{L}_{n,m,\beta,p}$. We use KMT theorem to couple $\chi_k$ random variables in tridiagonal matrices with Gaussians and then use Gershgorin circle theorem along with Weyl's inequality to control the difference in the spectrums of the modified and the original tridiagonal matrices. This is Proposition \ref{eq: estimate of off diagonals of the difference} proved in the beginning of Section \ref{s: rtlb}. 
%Other than the upper bound for right tails, our proof methods to study the largest eigenvalues of $\wih{H}_{n,\beta,p}$ and $\wih{L}_{n,m,\beta,p}$ are discretized versions of the methods used in \cite{RRV} which obtain \eqref{eq: TW right tail bounds} and \eqref{eq: TW left tail bounds}. 
Below we outline different pieces of the arguments for the Hermite ensemble in more detail. Similar arguments hold for Laguerre ensemble.

\subsubsection{\textbf{Upper bounds for right tail of $\beta$-ensembles.}} If one has sharp right tail upper bounds for any $\beta_1$, then one can obtain the same for any  $\beta_2>\beta_1$
 using the following stochastic domination result from \cite{JB23}:  
\begin{align*}
\lambda_{\text{max}}\l(\sqrt{\beta_1}H_{\left\lceil{\beta_2 n/\beta_1 }\right\rceil,\beta_1}\r)\succeq\lambda_{\text{max}}\l(\sqrt{\beta_2}H_{n,\beta_2}\r).
  \end{align*}
Sharp upper bounds for right tails were obtained for  $\beta=1$ in \cite[Lemma 1.1]{BX23}, although the range for $t$ is up-to $n^{1/4}$, instead of $n^{2/3}$. This allows for right tail upper bounds for the range $1\leq\beta<2$, for the above mentioned range of $t$. However, we can obtain right tail upper bounds for the range of $t$ up to $\delta n^{2/3}$ for $\beta=2$, and hence for all $\beta\ge 2$. The proof for $\beta=2$ uses the connection between GUE matrices and Brownian last passage percolation, together with the  super-additivity of passage times and a large deviations result for the largest point in $\beta$ ensembles. The proofs are in Section \ref{s: upper bound right tail}.

\subsubsection{\textbf{Lower bounds for right tails of modified $\beta$-ensembles.}} We modify the matrix $H_{n,\beta}$ to $\wih{H}_{n,\beta,p}$, by replacing  $p\asymp  n^{1/3}$ of the $\chi$ random variables with  appropriate Gaussian random variables. Then we  construct a vector $v$ and use Gaussian tail bounds to conclude that  
\begin{align*}
    \< \wih{H}_{n,\beta,p}v,v\>\geq\l(2\sqrt{n}+tn^{-1/6}\r) \<v,v\>
\end{align*}
occurs with probability at least as that of the lower bound of Theorem \ref{t: rtlb}. Using variational  characterisation of $\lambda_{\text{max}}$, we have the desired lower bound for right tail of $\lambda_{\text{max}}(\wih{H}_{n,\beta,p})$ as defined in the beginning of this subsection. The vector $v$ we construct is such that the non-zero part of the vector is simply the discretized form of the function used in \cite{RRV} to obtain lower bound for right tails of $TW_{\beta}$. We then use the K.M.T. theorem (in Proposition \ref{eq: estimate of off diagonals of the difference}) to argue that the lower bound carries over to the right tails of  $\lambda_{\text{max}}(H_{n,\beta})$. The proofs are  in Section \ref{s: rtlb}.

\subsubsection{\textbf{Upper bounds for left tail of modified $\beta$-ensembles.}} Similar to the right tail lower bound, we study the largest eigenvalue of $\wih{H}_{n,\beta,p}$ first. Again we exhibit a deterministic vector $v$ and use Gaussian bounds to show  
\begin{align*}
     \< \wih{H}_{n,\beta,p}v,v\>\leq \l(2\sqrt{n}-tn^{-1/6}\r)\< v,v\>
\end{align*}
occurs with probability at most as in the upper bound of Theorem \ref{t: ltub}. Using variational  characterisation of $\lambda_{\text{max}}$, we have the desired upper bound for left tail of $\lambda_{\text{max}}(\wih{H}_{n,\beta,p})$. The vector $v$ we construct is such that the non-zero part of the vector is simply the discretized form of the function used in \cite{RRV} to obtain upper bound for left tails of $TW_{\beta}$. We then use Proposition \ref{eq: estimate of off diagonals of the difference} to argue that this also implies the desired upper bound for left tails of $\lambda_{\text{max}}(H_{n,\beta})$. We prove these with details in Section \ref{s: ltub}.

\subsubsection{\textbf{Lower bounds for left tail of modified $\beta$-ensembles.}} This is the most technically challenging part of our arguments.   We use the fact that for a tridiagonal matrix $A$ with positive off-diagonal entries, if we define the vector $v\in \R^{n+1}$ recursively such that $Av=\lambda v$ with $v(1)>0$ and the vector does not change sign, then $\lambda_{\text{max}}\leq \lambda$. 

Working again with $\wih{H}_{n,\beta,p}$, we use this fact and show that the probability of roughly the top $tn^{1/3}$ entries of the vector defined using the above recursion do not change sign is at least the optimal lower bound of Theorem \ref{t: ltlb}. This shows that $\lambda_{\text{max}}(\wih{H}_{n,\beta,p}^{top})\leq \l(2\sqrt{n}-tn^{-1/6}\r)$ occurs with the desired lower bound. Here $\wih{H}_{n,\beta,p}^{top}$ is the top $tn^{1/3}$ sub-matrix of $H_{n,\beta,p}$. Such analysis is possible only for the top $tn^{1/3}$ sub-matrix as we have replaced the $\chi_k$ random variables by Gaussians in that range. This is the main difficulty in discretizing the diffusion technique used in \cite{RRV} to obtain lower bounds for \eqref{eq: TW left tail bounds}.
% \textcolor{red}{emphasize this a bit more, this is the main additional difficulty we have compare to RRV, also reference for the result for the spiked case?}
We obtain similar bounds for the largest eigenvalue of bottom sub-matrix. Although the top and bottom sub-matrices are independent (for Hermite ensemble), the non-zero off-diagonal entries excluded from both top and bottom sub-matrices have to be accounted for to combine both bounds for top and bottom sub-matrices (this issue is not present in the limiting diffusion studied in \cite{RRV}). This is done using Proposition \ref{prop: Tridiag resultwith1} along with a spiked random matrix result \cite{BV12}. This allows to prove the desired lower bound for left tail of $\lambda_{\text{max}}(\wih{H}_{n,\beta,p})$. We then use Proposition \ref{eq: estimate of off diagonals of the difference} to argue that this also implies the desired lower bound for left tails of $\lambda_{\text{max}}(H_{n,\beta})$. The details are given in Section \ref{s: ltlb}.

\subsubsection{\textbf{The LILs}}Here we briefly outline the ideas that go into the proofs of the LILs and the new results from LPP geometry that we use in the proofs. These are interesting in their own rights and we expect them to be useful in study of several last passage percolation problems.\\ 
\begin{comment}The LIL problem for the point-to-point case, was considered first by Ledoux \cite{L18} where, using the weaker moderate deviation estimates from \cite{LR10}, it was shown that, almost surely,
$$\limsup_{n\to \infty} \frac{T^{\ptop{}}_{n}-4n}{g_{+}(n)}\leq\l(\frac{3}{4}\r)^{2/3}$$
and 
$$\liminf_{n\to \infty} \frac{T^{\ptop{}}_{n}-4n}{g_{-}(n)}>-\infty.$$

Using the same heuristic described above and anticipating the optimal tail estimates for $T^{\ptop{}}_{n}$ Ledoux conjectured the $\limsup$ in the point-to-point case in Theorem \ref{t: LILs} in \cite{L18}.
\end{comment}
As pointed out before, Ledoux proved \eqref{eq: Ledoux LPP result}
% $\limsup$ statement for the point-to-point case in Theorem \ref{t: LILs}
assuming a version of Theorem \ref{t: passage time tail estimates}, (i)  for $T^{\ptop{}}_{n}$. Ledoux's argument relies only on the tail estimates and super-additivity and we shall use variants of the same argument to prove all our $\limsup$ results. 

Proofs of the $\liminf$ statements turn out to be more difficult. 
Combining the result of \cite{L18}, a zero-one law and using the point-to-line passage time to bound the point-to-point passage time, and a weak version of Theorem \ref{t: passage time tail estimates}, (ii) for $T^{\ptop{}}_{n}$, the result \eqref{eq: BGHK LPP result} was shown in \cite{BGHK21}.

\begin{comment}Following \cite{L18}, in \cite{BGHK21}, using the point-to-line passage time to bound the point-to-point passage time, and a weak version of Theorem \ref{t: passage time tail estimates}, (ii) for $T^{\ptop{}}_{n}$, \eqref{eq: BGHK LPP result} was shown. 
$$\liminf_{n\to \infty} \frac{T^{\ptop{}}_{n}-4n}{g_{-}(n)}<0$$ 
and combining with Ledoux's result and a zero-one law, this implied that almost surely
$$\liminf_{n\to \infty} \frac{T^{\ptop{}}_{n}-4n}{g_{-}(n)}=-C_2\in (-\infty,0).$$
\end{comment}
However, neither the argument in \cite{L18} nor the one in \cite{BGHK21} gives the correct value of $C_2$ even assuming the optimal tail estimate (in the latter case, this is due to the fact that using point-to-line passage times in stead of point-to-point leads to a loss in the constant, as was already pointed out in \cite{BGHK21}). 

Not surprisingly, it turns out to get the correct constant in the $\liminf$, one needs to utilize the LPP geometry. For all the cases we need certain known results about geodesic geometry from \cite{BGZ21, BBB23} together with some tail estimates for interval-to-line passage times from \cite{BBF22}. 
\begin{comment}\textcolor{red}{is this correct? if it is from somewhere it cannot be new. point out what exactly is new?}
\end{comment}
Similar estimates for interval-to-point passage times, point-to-line geodesics and half space geodesics, which were not in the literature before are proved to establish the $\liminf$ results in Theorem \ref{t: LILs} (see Lemma \ref{lemma: infimum over a line for half space}, \ref{typical interval to point passage time},\ref{local transversal fluctuation for point to line geodesic}, \ref{Interval to full line lemma}, \ref{lemma: interval to line in half space}, \ref{lemma: local transversal fluctuation for half space}).

\subsection*{Organisation of this paper} The rest of the paper is organised as follows. We prove the right tail upper bounds (Theorem \ref{t: right tail upper bound for beta geq 2} and Theorem \ref{t: rtub beta geq 1}) in Section \ref{s: upper bound right tail}, right tail lower bounds (Theorem \ref{t: rtlb}) in Section \ref{s: rtlb}, left tail upper bounds (Theorem \ref{t: ltub}) in Section \ref{s: ltub}, left tail lower bounds (Theorem \ref{t: ltlb}) in Section \ref{s: ltlb}. In each section, the Hermite case is dealt with first, and then the more complicated Laguerre case is tackled. Theorem \ref{t: passage time tail estimates} is a corollary of the tail bounds. We prove the LILs in Section \ref{sec: limsup} and Section \ref{sec: liminf}. The $\limsup$s in Theorem \ref{t: LILs}, are proved in Section \ref{sec: limsup}. The $\liminf$s in Theorem \ref{t: LILs}, are proved in \ref{sec: liminf}. In Section \ref{s: LPP estimates} we prove all the estimates in LPP that we will use to prove the LILs.

\subsection*{Acknowledgements} This work is partly supported by the DST FIST program-2021[TPN-700661].  JB is supported by scholarship from Ministry of Education (MoE). RB is partially supported by a MATRICS grant (MTR/2021/000093) from SERB, Govt.\ of India, DAE project no.\ RTI4001 via ICTS, and the Infosys Foundation via the Infosys-Chandrasekharan Virtual Centre for Random Geometry of TIFR. SB is supported by scholarship from National Board for Higher Mathematics (NBHM) (ref no: 0203/13(32)/2021-R\&D-II/13158). %JB is supported by scholarship from Ministry of Education (MoE). 

\section{Upper bounds for right tails}\label{s: upper bound right tail}
In this subsection we prove Theorem \ref{t: right tail upper bound for beta geq 2} and Theorem \ref{t: rtub beta geq 1}. We first consider the Hermite ensemble.
\subsection{Hermite Ensemble} We first prove the $\beta=2$ case. Note that by \cite[Lemma 1.1]{BX23} we get Theorem \ref{t: right tail upper bound for beta geq 2}, (i) for $\beta=2$ case the range $t \leq n^{1/4}.$ To get the upper bound for the range all the way up-to $\gamma(\vep) n^{2/3}$ for sufficiently small $\vep$, we need the correspondence between the $\lambda_1^{(n,H,2)}$ and the last passage time in Brownian LPP and a super-additivity argument.  The definition of Brownian LPP and the exact correspondence can be found in \cite[Theorem 1]{OY02}. Then, an  application of large deviations and super-additivity will give us Theorem \ref{t: right tail upper bound for beta geq 2}, (i). This approach was already considered in \cite{KJ} and by Ledoux in \cite{L07}. We don't write this in detail for the Hermite case, but we give full details in the Laguerre  $\beta=1$ case.

\begin{proof}[\textbf{Proof of Theorem \ref{t: right tail upper bound for beta geq 2}, (i)}] Once we have Theorem \ref{t: right tail upper bound for beta geq 2}, (i)  for $\beta=2$, the same follows   for any $\beta \geq 2$ because  $\lambda_{\text{max}}\l(H_{n,\beta}\r)\overset{d}{=}\lambda_1^{\l(n,H,\beta\r)}$ and by \cite[Theorem 1]{JB23}
\begin{align}\label{eq:stochasticdominationJB}
\lambda_{\text{max}}\l(\sqrt{2}H_{\left\lceil\frac{\beta n}{2}\right\rceil,2}\r)\succeq\lambda_{\text{max}}\l(\sqrt{\beta}H_{n,\beta}\r).
  \end{align}
Here is a sketch of the argument, see \cite{JB23} for more details.
 
 The matrices $\sqrt{2}H_{\left\lceil\frac{\beta n}{2}\right\rceil,2}$ and $\sqrt{\beta}H_{n,\beta}$ 
 can be coupled so that the off-diagonal entries in the top-left $n\times n$  principal sub-matrix of the former are strictly larger than the corresponding entries in the latter, and so that the diagonal entries are equal. As the off-diagonal entries are positive, the stochastic domination \eqref{eq:stochasticdominationJB} follows.
  
   Hence we can write,
  \begin{align}\label{eq: right tail comparison}
       \P \l( \lambda_1^{\l(n,H,\beta\r)} \geq 2 \sqrt{n}+tn^{-1/6}\r)&=     \P \l( \lambda_{\text{max}}\l(\sqrt{\beta}H_{n,\beta}\r) \geq 2 \sqrt{n\beta}+\sqrt{\beta}tn^{-1/6}\r)\nonumber\\
       &\leq \P \l( \lambda_{\text{max}}\l(\sqrt{2}H_{\left\lceil\frac{\beta n}{2}\right\rceil,2}\r) \geq 2 \sqrt{n\beta}+\sqrt{\beta}tn^{-1/6}\r)\nonumber\\
       & \leq \P \l( \lambda_{\text{max}}\l(H_{\left\lceil\frac{\beta n}{2}\right\rceil,2}\r) \geq 2 \sqrt{n\beta/2}+t\sqrt{\beta/2}n^{-1/6}\r).
  \end{align}
Hence using the result for $\beta=2$ case, 
%\eqref{eq: right tail comparison} and \eqref{eq: comparing right tail upper bounds} 
we get that for any $\beta\geq 2$ and $\vep>0$, we can choose $n_{\vep,\beta}$ such that for all $n\geq n_{\vep,\beta}$ and for all $t$ with $1\leq t\leq \gamma(\vep)n^{\frac{2}{3}}$, the following holds (note that the difference between $\sqrt{\beta n/2}$ and $\sqrt{\lceil \beta n/2\rceil}$ is $O(1/\sqrt{\beta n})$ and can be  absorbed into $\vep$):
\begin{align*}
    \P \l( \lambda_1^{\l(n,H,\beta\r)} \geq 2 \sqrt{n}+tn^{-1/6}\r)& \leq \P \l( \lambda_{\text{max}}\l(H_{\left\lceil\frac{\beta n}{2}\right\rceil,2}\r) \geq 2 \sqrt{n\beta/2}+t\sqrt{\beta/2}n^{-1/6}\r)\nonumber\\
    % &\leq \P \l( \lambda_{\text{max}}\l(H_{\left\lceil\frac{\beta n}{2}\right\rceil,2}\r) \geq 2\sqrt{\left\lceil\frac{\beta n}{2}\right\rceil}+ t\l(\frac{\beta}{2}\r)^{2/3} \left\lceil\frac{\beta n}{2}\right\rceil^{-1/6}\l( 1-\delta\r) \r)\nonumber\\
     &\leq \exp\l(-\frac{2\beta}{3}t^{3/2}\l(1-\vep\r)\r).
\end{align*}
This completes the proof.
%For any $\beta\geq 2$ and $\vep>0$, we can choose $n_{\vep,\beta}$ (different from earlier), $\gamma(\vep)>0$ small enough such that for all $n\geq n_{\vep,\beta}$ and for all $t$ with $1\leq t\leq \gamma(\vep)n^{\frac{2}{3}}$, the following holds 
%\begin{align}\label{eq: final beta right tail upper bound}
%    \P \l( \lambda_1^{\l(n,H,\beta\r)} \geq 2 \sqrt{n}+tn^{-1/6}\r) \leq  \exp\l(-\frac{2\beta}{3}t^{3/2}\l(1-\vep\r)\r).
%\end{align}
\end{proof}
\begin{proof}[\textbf{Proof of Theorem \ref{t: rtub beta geq 1}, (i)}]
In Theorem \ref{t: rtub beta geq 1} the $\beta=1$ case we get from \cite[Lemma 1.1]{BX23}. For $1 \leq \beta <2$ range we apply the similar stochastic domination argument. This completes the proof.
\end{proof}
\begin{remark}  Unlike in the case for $\beta=2,$ we were unable to find a correspondence between $\lambda_1^{(n,H,1)}$ and some LPP model. Therefore, we were unable to apply the super-additivity argument or large deviation asymptotics as we have applied in the other right tail upper bound proofs. But using the moment recursion of \cite{Ledoux09}, we can get optimal right tail bound for the range of $t\leq \gamma(\varepsilon)n^{2/3}$ for $\beta=1$. Then applying the above stochastic domination argument, Theorem \ref{t: rtub beta geq 1}, (i) holds for the same range of $t$ as in Theorem \ref{t: right tail upper bound for beta geq 2}, (i). As it is a completely different argument from the one used in $\beta=2$ case, we skip the details here.
\end{remark}

\subsection{Laguerre Ensemble}
The $\beta=2$ in Theorem \ref{t: right tail upper bound for beta geq 2}, (i) case can be found in \cite{KJ,L07,L18} using a similar super-additivity argument and large deviation asymptotics. We will also give a proof in the Laguerre $\beta=1$ case.
 \begin{proof}[\textbf{Proof of Theorem \ref{t: right tail upper bound for beta geq 2}, (ii)}]
 Once we have Theorem \ref{t: right tail upper bound for beta geq 2}, (ii) for $\beta=2$ case, as before using stochastic domination argument in the proof of \cite[Theorem 3]{JB23} we observe that 
 \[
 2\lambda_{\max}\left(L_{\lceil \frac{\beta n}{2}\rceil,\lceil \frac{\beta n}{2} \rceil \frac mn,2} \right) \succeq
 \beta \lambda_{\max} \left(L_{n,m,\beta} \right).\]
 Now, applying same argument as we did in the Hermite case we get Theorem \ref{t: right tail upper bound for beta geq 2}, (ii) for all $\beta \geq 2.$
 \end{proof}
 We next prove Theorem \ref{t: rtub beta geq 1}, (ii) for $\beta=1$ case. We use the distributional equality of point-to-line last passage time with $\frac 12 \lambda_{\max} \left(L_{2n-1,2n,1} \right)$ \eqref{eq: distributional equalities}.
 %However, we could only use the above recursions in the case $m=n.$ For general $m, n$ we invoke a different approach which is similar to the approach using the superadditivity coming from last passage percolation correspondence taken in \cite{L18}. We first prove Theorem \ref{t: lauguerre right tail upper bound} (1) using this approach. Let us consider the exponential last passage percolation model on $\Z^2.$ i.e., we have a collection of i.i.d.\ Exp(1) $\l\{\tau_v \r\}_{v \in \Z^2}$ random variable associated to each vertices of $\Z^2.$ For any up-right path $\gamma$ between $u$ and $v,$we define 
 %\[
 %\ell\l(\gamma\r):= \sum_{w \in \gamma \setminus \l\{u \r\} } \tau_w.
 %\]
% For any $u,v \in Z^2$ we define
% \[
% T_{u,n}^*:= \max   \l\{\ell \l(\gamma\r): \gamma \text{ is an up-right path between } u \text{ to the line } x+y=2n  \r\}.
% \]
 %When $u=\l(0,0\r)$ then we simply write $T_n^*.$
% We have the following equality in distribution. 
 %\[T_{\l(1,1\r),n}^*+\tau_{\l(1,1\r)} \overset{d}{=}
% \frac 12 \lambda_{\max}\l( W_{2n , 2n-1 , 1} \r)
% \] Thus, it is enough to prove the following. 
 \begin{lemma}
 \label{l: large deviation proof for beta=1}
 For any $\vep>0$ there exists $\delta>0$ and $n$ sufficiently large and $t \leq \delta n^{2/3}$ sufficiently large
 \[
 \P\l(T_n^{\ptoline{}} \geq 4n+ 2^{4/3}tn^{1/3}\r) \leq \exp\l(-\frac{4}{3}t^{3/2}\l(1-\vep\r)\r).
 \]
 \end{lemma}
 Note that the above lemma was already proved in \cite[Lemma A.4]{BBF22} but for the range $t \ll n^{2/9}.$ Here we give an alternate proof and for the range $t \leq \delta n^{2/3}$ for sufficiently small $\delta$. We now provide a proof of Lemma \ref{l: large deviation proof for beta=1}. 
 \begin{proof}
     We fix an $n$ and we claim the following. For any $m \geq 1, $
     \begin{equation}
     \label{eq: claim product}
     \P \l(T_n^{\ptoline{}} \geq 4n+ 2^{4/3}t n^{1/3} \r) ^m \leq \P\l(T_{mn}^{\ptoline{}} \geq 4mn + 2^{4/3}tmn^{1/3}\r).
     \end{equation}
     Recall the definition of $T_n^{\ptoline{}}$. For this proof in the definition of point-to-line passage time we have excluded the initial vertex $\bo{0}$. We prove the above claim with this definition. \\ 
     Let $v_{n}$ be the random vertex on the line $x+y=2n$ hit by the geodesic for $T_{2n}^{\ptoline{}}$ and $T_{v_n,\cl_{4n}}^{\ptoline{}}$ be the last passage time such that $T_{2n}^{\ptoline{}}= T_{n}^{\ptoline{}}+T^{\ptoline{}}_{v_n,\cl_{4n}}$. Note that these two random variables are i.i.d. Hence we have 
     \begin{align*}
     \P \l(T_n^{\ptoline{}} \geq 4n+ 2^{4/3}t n^{1/3} \r) ^2 \leq \P\l(T_{2n}^{\ptoline{}} \geq 8n + 2^{4/3}.2tn^{1/3}\r).
     \end{align*}
     By similar argument for $m> 2$, the claim follows. 
We now look at the joint eigen value density of the matrix $\frac {1}{2n-1} L_{2n-1,2n,1}.$ The density is as follows
 \[
 \frac{1}{Z_V^{2n-1}}\prod_{k=1}^{2n-1}e^{-\frac 12\left(2n-1\right) \lambda_k}\times \prod\limits_{j<k}\left|\lambda_j-\lambda_k\right|,
 \]
 where 
 \[
 V\l(x\r):=\frac x2; \quad x>0
 \]
 and 
 \[
 Z_V^n:=\int_{0}^{\infty} \dots \int_{0}^{\infty}\prod\limits_{j<k}\left|\lambda_j-\lambda_k\right|\prod_{k=1}^n e^{-nV(\lambda_k)}d\lambda_1 \dots d\lambda_n.
 \]
% Let us define
 %\[
 %V\l(x\r):=\frac x2; \quad x>0.
 %\]
% We use the notation $Z^{n}_{V, \beta}$ to denote the partition function corresponding to different $V$ and $n$. 
We calculate using the Selberg's integral formula that
 \[
 \lim_{n \rightarrow \infty } \frac 1{n} \log \frac{Z^{n-1}_{nV/\l(n-1\r)}}{Z^{n}_{V}}=1.
 \]
 By \cite[Theorem 2.6.6]{AGZ09} this implies that the large deviation rate function exists for $\lambda_{\max}\l(\frac {1}{n} L_{n,n+1, 1} \r )$ and it is given by
 \[
J\l(4+\vep\r)=-\int_0^4\log|4+\vep-y|d\sigma\l(y\r)+V\l(4+\vep\r)-1,
 \]
 for $\vep>0.$ Above $V\l(x\r)=\frac{x}{2}$ and $\sigma$ is the Marchenko-Pastor law given by
 \[
 d\sigma=\frac{1}{2\pi x}\sqrt{x\l(4-x\r)} \mathbbm{1}_{[0,4]}dx.
 \]
 Also, note that 
 \begin{equation}
 \label{eq: ldp limit}
 \lim_{\vep \rightarrow 0} \frac{J(4+\vep)}{\vep^{3/2}}=\frac 16.
 \end{equation}
 \begin{comment}
Now, as we have 
\[
\frac 12 \lambda_1 \l(W_{2n-1,2n \beta} \r) \overset{d}{=} T_n^{\ptoline{}}+\text{Exp}(1),
\]
\end{comment}
From Claim \ref{eq: claim product} we get
 \begin{align*}
     m \log \l( \P \l(T_n^{\ptoline{}} \geq 4n+ 2^{4/3}t n^{1/3} \r) \r) \leq \log \P\l(T_{mn}^{\ptoline{}} \geq 4mn + 2^{4/3}tmn^{1/3}\r) &\\ \implies \frac{1}{n}\log \l( \P \l(T_n^{\ptoline{}} \geq 4n+ 2^{4/3}t n^{1/3} \r) \r) \leq \frac{1}{mn}\log \P\l(T_{mn}^{\ptoline{}} \geq 4mn + 2^{4/3}tmn^{1/3}\r).
 \end{align*}
 Now, as we have 
\[
\frac 12 \lambda_{\max} \l(L_{2n-1,2n, 1} \r) \overset{d}{=} T_{n-1}^{\ptoline{}}+\text{Exp}(1),
\]
 taking $m \rightarrow \infty$ and applying \eqref{eq: ldp limit} we get
 \begin{comment}
 \[
 \frac{1}{n}\log \l( \P \l(T_n^{\ptoline{}} \geq 4n+ 2^{4/3}t n^{1/3} \r) \r) \leq J\l(4+\frac{2^{4/3}t}{n^{2/3}} \r).
 \]
 \end{comment}
 \begin{comment}
 Where we have by \cite[Theorem 2.6.6]{AGZ09}
 \[
 J\l(4+\vep\r)=\int_0^4\log|4+\vep-y|d\sigma\l(y\r)-V\l(4+\vep\r)+1,
 \]
 \end{comment}
 \begin{comment}
 where $V\l(x\r)=\frac{x}{2}$ and $\sigma$ is the Marchenko-Pastor law given by
 \[
 d\sigma=\frac{1}{2\pi x}\sqrt{x\l(4-x\r)}dx.
 \]
 \end{comment}
 \begin{comment}
 Now, we calculate 
 \begin{align*}
 \lim_{\vep \rightarrow 0} \frac{J\l(\vep\r)}{\vep^{3/2}}=\lim_{\vep \rightarrow 0}\frac{\frac{1}{2 \pi}\int_0^4\frac{\sqrt{y\l(4-y\r)}\log|4+\vep-y|}{y}d\l(y\r)-\frac{\l(4+\vep\r)}{2}+1}{\vep^{3/2}}=-\frac 16. %=\lim_{\vep \rightarrow 0}\frac{\frac{1}{2 \pi}\int_0^4\frac{\sqrt{y\l(4-y\r)}}{\l(4+\vep-y\r)y}d\l(y\r)-\frac{1}{2}}{\frac 32\sqrt{\vep}} \text{( by using L'Hospital Rule )}=-\frac 16.
 \end{align*}
 \end{comment}for any $\vep>0$, $n$ sufficiently large and $t \leq \delta n^{2/3}$ for some small enough $\delta$
 \[
  \log \l( \P \l(T_n^{\ptoline{}} \geq 4n+ 2^{4/3}t n^{1/3} \r) \r) \leq -\l(1-\vep\r)\frac 43{t^{3/2}}.
 \]
 This proves the lemma.
 \end{proof}
 \begin{proof}[\textbf{Proof of Theorem \ref{t: rtub beta geq 1}, (ii)}] We obtain Theorem \ref{t: rtub beta geq 1}, (ii) for $\lambda_{\max} \l(L_{2n-1,2n,1} \r)$ from above lemma. We use the similar stochastic domination argument and Lemma \ref{l: large deviation proof for beta=1}. Using the stochastic domination argument of \cite[Theorem 1]{JB23} to compare $\lambda_{\max} \l(L_{2n-1,2n,1} \r)$ and $\lambda_{\max} \l(L_{2n-2,2n-1,1} \r)$, one can obtain Theorem \ref{t: rtub beta geq 1}, (ii) for $\lambda_{\max} \l(L_{2n-2,2n-1,1} \r)$. This proves Theorem \ref{t: rtub beta geq 1}, (ii).\end{proof}
\begin{remark}\label{remark: Laguerre n+c}
    Comparing $\chi_{\beta(n-i)}$ with $\chi_{\beta(n+c-i)}$ and using Weyl's inequality it follows immediately that above bounds hold for $(m,n)$ of the form $(n+c,n)$ for any $c\geq 0$.
\end{remark}

\section{Lower bounds for right tails} 
\label{s: rtlb}
For the remaining proofs of the tail bounds we use the below proposition which allows us to study modifications of $H_{n,\beta}$ and $B_{n,m,\beta}$. Let $\wih{H}_{n,\beta,p}$ be the matrix where $\chi_{\beta\l(n-i\r)}$ in $H_{n,\beta}$ are replaced by $\sqrt{\beta\l(n-i\r)}+\frac{\zeta_i}{\sqrt{2}}$ for all $i\leq p$, where $\zeta_i$ are i.i.d.\ $N\l(0,1\r)$ random variables.  Let $\wih{B}_{n,m,\beta,p}$ be the matrix where $D_{i}$s and $C_i$s in $B_{n,m,\beta}$ are replaced by $\sqrt{\beta\l(m+1-i\r)}+\frac{\zeta_i}{\sqrt{2}}$ and $\sqrt{\beta\l(n-i\r)}+\frac{\zeta'_i}{\sqrt{2}}$ respectively for all $i\leq p$. Here $\zeta_i$s and $\zeta'_i$s are independent families of i.i.d.\ $N(0,1)$ random variables. The entries of $\wih{H}_{n,\beta,p}$ and $\wih{B}_{n,m,\beta,p}$ are independent up to symmetry. We define $\wih{L}_{n,m,\beta,p}=\wih{B}_{n,m,\beta,p}^t\wih{B}_{n,m,\beta,p}$. 
% Note that $\wih{L}_{n,m,\beta,p}$ can be written as a sum of two matrices $A_{n,m,\beta,p}$ and $B_{n,m,\beta,p}$, where $B_{n,m,\beta,p}$ is a tridiagonal symmetric matrix which contains entries which are linear combinations of $\zeta_i^2,\l(\zeta'_i\r)^2, \zeta_{i+1}\zeta'_i$ in the top $p\times p$ submatrix and zeros elsewhere. $A_{n,m,\beta,p}$ contains random variables which are Gaussian in the top $p\times p$ submatrix. 

\begin{proposition}\label{eq: estimate of off diagonals of the difference}
    $\left\lVert \wih{H}_{n,\beta,p}-H_{n, \beta}\right\rVert_2\leq c_{\beta} \frac{\log n+2x}{ \sqrt{n}}$ with probability at least $1-C n\l(3e^{-x}+2e^{-\frac{\beta\l(n-p\r)}{2}}\r)$ (for some constants $C,c_{\beta}>0$ and $n$ sufficiently large). We also have for $\frac mn$ uniformly bounded by some constant, $\l\lVert\wih{L}_{n,m,\beta,p}-L_{n,m,\beta}\r\rVert_2\leq c_{\beta} \l({\log n+2x}\r)$ with probability at least $1-C n\l(18e^{-x}+8e^{-\frac{\beta\l(n-p\r)}{2}}\r)$ (for some constants $C,c_{\beta}>0$ and $n$ sufficiently large).
\end{proposition}
\begin{proof}
For the rest of the proof, the constants $C,c_{\beta}$ might vary with each line.
    Using K.M.T. theorem \cite{KMT75}, one can couple $\chi^2_{\beta\l(n-i\r)}$ and $\l(n-i\r)\beta+\sqrt{2\beta\l(n-i\r)}\zeta_i$, where $\zeta_i\sim N\l(0,1\r)$ such that (this is done such that  entries of $\wih{H}_{n,\beta,p}$ are independent up to symmetry),
\begin{align}\label{eq: KMT inequality}
\l|\chi^2_{\beta\l(n-i\r)}-\l(\l(n-i\r)\beta+\sqrt{2\beta\l(n-i\r)}\zeta_i\r)\r|\leq C\l(\log n +x\r)\mbox{ with probability at least  }1-e^{-x},\nonumber\\
    \l|\chi^2_{\beta\l(n-i\r)}-\l(\sqrt{\l(n-i\r)\beta}+\l(\zeta_i/\sqrt{2}\r)\r)^2\r|\leq C\l(\log n +x\r)+\l(\zeta_i^2/2\r)  \mbox{ with probability at least  }1-e^{-x}.
\end{align}
Hence,  
\[\l|\chi_{\beta\l(n-i\r)}-\l(\sqrt{\l(n-i\r)\beta}+\zeta_i/\sqrt{2}\r)\r|\leq \frac{C\l(\log n +x\r)+\l(\zeta_i^2/2\r)}{\l|\chi_{\beta\l(n-i\r)}+\sqrt{\l(n-i\r)\beta}+\frac{\zeta_i}{\sqrt{2}}\r|}  \mbox{ with probability at least  }1-e^{-x}.
\]
 Note that $\P\l(|\zeta_i| \geq \sqrt{2x}\r) \leq 2e^{-x}$ and $\P\l(|\zeta_i| \geq \sqrt{\beta\l(n-i\r)}\r) \leq 2e^{-\frac{\beta\l(n-p\r)}{2}}$.
% We have $\P \l(\chi_{\beta \l(n-i\r)} \leq \frac 12 \sqrt{\l(n-i\r) \beta}\r) \leq e^{-c \sqrt{\beta\l(n-i\r)}} \leq e^{-n^{1/4}}$ for sufficiently large $n$.
Hence, with probability at least $1-3e^{-x}-2e^{-\frac{\beta\l(n-p\r)}{2}}$, we have 
\begin{align}\label{eq: KMT inequality 2}
\left|\chi_{\beta\l(n-i\r)}-\l(\sqrt{\l(n-i\r)\beta}+\zeta_i/\sqrt{2}\r)\right| \leq  \frac{C\l(\log n +x\r)+x}{\l(1-\frac{1}{\sqrt{2}}\r) \sqrt{\l(n-i\r)\beta}} \leq c_{\beta} \frac{\log n+2x}{ \sqrt{n}},
\end{align}
for some $c_{\beta}>0$ and $n$ sufficiently large. By
Gershgorin circle theorem (Theorem $6.1.1$ of \cite{HJ}) it follows that with probability at least $1-C n\l(3e^{-x}+2e^{-\frac{\beta\l(n-p\r)}{2}}\r)$ (for some constant $C>0$ and $n$ sufficiently large), we have that $\left\lVert \wih{H}_{n,\beta,p}-H_{n, \beta}\right\rVert_2\leq c_{\beta} \frac{\log n+2x}{ \sqrt{n}}$.

The diagonal elements of ${L}_{n,m,\beta}-\wih{L}_{n,m,\beta,p}$ are of the form 
\begin{align*}
    \frac{1}{\beta}\l(\chi^2_{\beta\l(m-i+1\r)}-\l(\sqrt{\l(m+1-i\r)\beta}+\l(\zeta_i/\sqrt{2}\r)\r)^2+\chi^2_{\beta\l(n-i\r)}-\l(\sqrt{\l(n-i\r)\beta}+\l(\zeta'_i/\sqrt{2}\r)\r)^2\r)
\end{align*}

and the off-diagonal elements are of the form  
\begin{align*}
    \frac{1}{\beta}\l(\chi_{\beta\l(n-i\r)}\chi_{\beta\l(m-i\r)}-\l(\sqrt{\l(n-i\r)\beta}+\zeta'_i/\sqrt{2}\r)\l(\sqrt{\l(m-i\r)\beta}+\zeta_{i+1}/\sqrt{2}\r)\r).
\end{align*}

By \eqref{eq: KMT inequality}, all the diagonal elements of $L_{n,m,\beta}-\wih{L}_{n,m,\beta,p}$ are less than $c_{\beta}\l(\log m+2x\r)$ with probability at least $1-6ne^{-x}$.
By \eqref{eq: KMT inequality 2}, each of the off-diagonal element is less than 
\begin{align*}
    \l(c_{\beta} \frac{\log m+2x}{ \sqrt{m}}+\l(1+\frac{1}{\sqrt{2}}\r)\sqrt{\beta\l(m-i\r)}\r)\l(c_{\beta} \frac{\log n+2x}{ \sqrt{n}}\r)\\
    +\l(\l(1+\frac{1}{\sqrt{2}}\r)\sqrt{\beta\l(n-i\r)}\r)\l(c_{\beta} \frac{\log m+2x}{ \sqrt{m}}\r)
\end{align*}
with probability at least $1-6e^{-x}-4e^{-\frac{\beta\l(n-p\r)}{2}}$. By
Gershgorin circle theorem (Theorem $6.1.1$ of \cite{HJ}) it follows that with probability at least $1-C n\l(18e^{-x}+9e^{-\frac{\beta\l(n-p\r)}{2}}\r)$ (for some constant $C>0$ and $n$ sufficiently large), we have that $\left\lVert L_{n,m,\beta}-\wih{L}_{n,m,\beta,p}\right\rVert_2\leq c_{\beta} \l({\log n+2x}\r)$.
\end{proof}
\begin{comment}
As a corollary of the above proposition we now show that how we can control the tail probabilities of $\lambda_{\max}\l(H_{n,\beta} \r)$ and $\lambda_{\max} \l(L_{n,m,\beta,p} \r)$ by $\lambda_{\max} \l(\wih{H}_{n,\beta,p} \r)$ and $\lambda_{\max} \l(\wih{L}_{n,m,\beta,p} \r).$
\begin{corollary}.
    \begin{enumerate}[label=(\roman*), font=\normalfont]
        \item 
        \begin{align}\P &\l(\lambda_{\max}\l(\wih{H}_{n,\beta, p}\r) \geq 2\sqrt{n}+tn^{-1/6} \r) \leq \P \l( \lambda_{\max} \l(H_{n,\beta} \r) \geq 2\sqrt{n}+tn^{-1/6}-c_\beta \frac{\log n+2x}{\sqrt{n}}\r) \nonumber\\ &+Cn \l(3 e^{-x}+2 e^{-\frac{\beta (n-p)}{2}} \r).
        \end{align}
    \end{enumerate}
\end{corollary}
\begin{proof}
Note that by Weyl's inequality we have 
\[
\l\{\lambda_{\max} \l(\wih{H}_{n,\beta,p} \r) \geq 2 \sqrt{n}+t n^{-1/6}\r\} \subset \l\{ \lambda_{\max} \l(H_{n, \beta}\r) \geq 2 \sqrt{n}+t n^{-1/6}-{\lambda_{\max} \l(\wih{H}_{n,\beta,p}-H_{n, \beta} \r)}\r\}.
\]
By Proposition \ref{eq: estimate of off diagonals of the difference} and by Gershgorin's circle theorem \cite[Theorem 6.1.1]{HJ} we then get the first inequality.
    
\end{proof}
\end{comment}

Now we prove Theorem \ref{t: rtlb}.

% We make use of the following result.\\ 

% \textcolor{blue}{(taken from stackexchange, https://math.stackexchange.com/q/1393943)}

% \textcolor{blue}{By Darboux's theorem if $f'$ is decreasing 
% \begin{align}\label{eq: Darboux}
%     \frac 12 f'(b)(b-a)^2 \leq \int_a^b f\l(x\r)\,\mathrm{d}x-f\l(a\r)\l(b-a\r) \leq \frac12f'\l(a\r)\l(b-a\r)^2
% \end{align}
% }
We use the result that if $|f'|$ is Riemann integrable and $f'$ is decreasing then
\begin{align}\label{eq: Riemmann sum convergence}
    \lim_{n\to\infty}n\left|\,\sum_{k=0}^\infty f\l(\frac kn\r)\frac1n-\int_0^\infty f\l(x\r)\,\mathrm{d}x\,\right|
\le\frac12\int_0^\infty|f'\l(x\r)|\,\mathrm{d}x.
\end{align}

% This says that if the variation of a function is finite, then the Riemann sums converge to the integral.  This makes sense, because any function of bounded variation can be written as the difference of two monotonic decreasing functions.
% ***
% \textcolor{blue}{
% Using Darboux's theorem to justify \eqref{eq: Darboux}
% \begin{align*}
% \int_a^b f\l(x\r)\,\mathrm{d}x-f\l(a\r)\l(b-a\r)
% &=\int_a^b\l(f\l(x\r)-f\l(a\r)\r)\,\mathrm{d}x\\
% &=\int_a^b\l(x-a\r)\,\frac{f\l(x\r)-f\l(a\r)}{x-a}\,\mathrm{d}x\\\
% &=\int_a^b\l(x-a\r)\,f'\l(\xi\l(x\r)\r)\,\mathrm{d}x\\
% & \leq \frac12\l(b-a\r)^2f'\l(a\r)
% \end{align*}
% The other inequality follows similarly.
% }

% \textcolor{blue}{Using \eqref{eq: Darboux} to justify \eqref{eq: Riemmann sum convergence}}
% \textcolor{blue}{
% \begin{align*}
% &n \left(\int_0^\infty f\l(x\r)\,\mathrm{d}x-\sum_{k=0}^\infty f\l(\frac kn\r)\frac1n \right)
% =n\sum_{k=0}^\infty\l(\int_{\frac kn}^{\frac{k+1}n}f\l(x\r)\,\mathrm{d}x-f\l(\frac kn\r)\frac1n\r).
% \end{align*}
% We have 
% \begin{align*}
% &\frac 12 \sum_{k=0}^{\infty } f'\l(\frac {k+1}{n}\r)\frac 1n \leq n\sum_{k=0}^\infty\l(\int_{\frac kn}^{\frac{k+1}n}f\l(x\r)\,\mathrm{d}x-f\l(\frac kn\r)\frac1n\r)\leq \frac 12\sum_{k=0}^{\infty } f'\l(\frac kn\r)\frac 1n\\
% %&\le\lim_{n\to\infty}\frac12\sum_{k=0}^\infty\left|f'\l(\xi_{n,k}\r)\right|\frac1n\\
% & \implies \lim_{n \rightarrow \infty} n \left| \int_{0}^{\infty}f(x) dx-\sum_{k=0}^{\infty}f \left( \frac kn \right)\frac 1n \right| \leq \frac12\int_0^\infty|f'\l(x\r)|\,\mathrm{d}x.
% \end{align*}
% }

\subsection{Hermite Ensemble}
% \textcolor{blue}{I think this works for $t \leq \delta n^{2/3}$ with $\delta$ sufficiently small (SB).}
 For any $\beta>0$ and small $\varepsilon>0$, let $t$ and $ \gamma(\vep)$ be such that $t_{\varepsilon,\beta}\leq t\leq \gamma(\vep)n^{2/3}$ {($t_{\varepsilon,\beta}, \gamma(\vep)$ to be made precise later)}. For this subsection, we consider the matrix $\wih{H}_{n,\beta,p}$ where $p:= \left\lceil \l(2\sqrt{t}+2\sech\l(\sqrt{t}\r)\r)\frac{n^{1/3}}{\sqrt{t}}\right\rceil$ (recall the definition of $\wih{H}_{n,\beta,p}$ from beginning of this section). We first obtain lower bounds for the right tails of $\lambda_{\text{max}}\l(\wih{H}_{n,\beta,p}\r)$. We choose a non-negative vector $v$ {(this vector depends on $n,t$ and $\varepsilon$)} such that $v\l(i\r)=0$ for all $i> p$ and $v\l(0\r)=0$ to show that for all  $t_{\varepsilon,\beta}\leq t\leq \gamma(\vep)n^{2/3}$, we have
\begin{align*}
    \< \wih{H}_{n,\beta,p}v,v\>\geq\l(2\sqrt{n}+tn^{-1/6}\r) \<v,v\>
\end{align*}
with probability at least $\exp\l(-\l(\frac{2\beta}{3}+\varepsilon\r)t^{3/2}\r)$. 
% we define $g\l(x\r)=\sech\l(\sqrt{t}\l(x-n^{1/6}\r)\r)$ and set $v\l(k\r)=g\l(k/n^{1/3}\r)$ for $k\leq \l\lfloor n^{\frac{1}{2}-\delta}\r\rfloor$. For $k>\l\lfloor n^{\frac{1}{2}-\delta}\r\rfloor$, we take $v\l(k\r)=0$.
Note that $\< \wih{H}_{n,\beta,p}v,v\>$ is a Gaussian random variable $\mu+\sigma Z$, where $Z\sim N\l(0,1\r)$. Using the fact $\sqrt{n-k}\geq \sqrt{n}-\frac{k+1}{2\sqrt{n}}$ for $k\leq 2\sqrt{n}-1$ (choose $n_{\varepsilon}$ such that $p\leq 2\sqrt{n}-1$ for all $n\geq n_{\vep}$), we have that the mean $\mu$ of $\< \wih{H}_{n,\beta,p}v,v\>$ is given by ,
\begin{align}\label{eq: mean lower bound for right tail lower bound}
    \mu&= 2\sum\limits_{k=1}^{p-1}\sqrt{n-k}\ v\l(k\r)v\l(k+1\r)\nonumber\\
    &\geq \sum\limits_{k=1}^{p-1}\l(\sqrt{n}-\frac{k+1}{2\sqrt{n}}\r)\l(v^2\l(k\r)+v^2\l(k+1\r)-\l(v\l(k+1\r)-v\l(k\r)\r)^2\r)\nonumber\\
    &\geq 2\sqrt{n}\<v,v\>-\sqrt{n}\sum\limits_{k=0}^{p}\l(v\l(k+1\r)-v\l(k\r)\r)^2-\frac{1}{\sqrt{n}}\sum\limits_{k=1}^{p} \l(k+1\r) v^2\l(k\r) .
\end{align}

Variance of $\< \wih{H}_{n,\beta,p}v,v\>$, denoted as $\sigma^2$ is given by,
\begin{align*}
    \sigma^2&=
    % \frac{2}{\beta}\sum\limits_{k=1}^{p}v^4\l(k\r)+\frac{2}{\beta}\sum\limits_{k=1}^{p-1}v^2\l(k\r)v^2\l(k+1\r) \\
    % &=\frac{2}{\beta}\sum\limits_{k=1}^{p}v^4\l(k\r)+\frac{1}{\beta}\sum\limits_{k=1}^{p-1}\l(v^4\l(k\r)+v\l(k+1\r)^4-\l(v^2\l(k+1\r)-v^2\l(k\r)\r)^2\r) \\
    \frac{4}{\beta}\sum\limits_{k=1}^{p}v^4\l(k\r)-\frac{1}{\beta}\sum\limits_{k=0}^{p}\l(v^2\l(k+1\r)-v^2\l(k\r)\r)^2.
\end{align*}

{Note that $\mu\leq \l(2\sqrt{n}+tn^{-1/6}  \r)\<v,v\>$.} We have,
\begin{align*}
    \P\l(\mu+\sigma Z\geq\l(2\sqrt{n}+tn^{-1/6}\r)\<v,v\>\r)\geq \P\l(\rho+\sigma Z\geq\l(2\sqrt{n}+tn^{-1/6}\r)\<v,v\>\r) 
\end{align*} 
where $\rho<\mu$. Using the notation $\Delta\l(v\l(k\r)\r)=v\l(k+1\r)-v\l(k\r)$, we consider the event (taking lower bound in \eqref{eq: mean lower bound for right tail lower bound} as $\rho$),
% probability of the event \l(this probability increases with $\mu$ as long as $\mu\leq2\sqrt{n}+tn^{-1/6}$\r),
\begin{equation}
\label{eq: Gaussian for right tail lower bound}
    2\sqrt{n}\<v,v\>-\sqrt{n}\sum\limits_{k=0}^{p}\l(\Delta\l(v\l(k\r)\r)\r)^2-\frac{1}{\sqrt{n}}\sum\limits_{k=1}^{p} \l(k+1\r) v^2\l(k\r) +\sigma Z\geq \l(2\sqrt{n}+tn^{-1/6}\r)\<v,v\>.
\end{equation}

 Let $f_{t}\l(x\r)=\sech\l(x-\sqrt{t}-\sech\l(\sqrt{t}\r)\r)$ for $\sech\l(\sqrt{t}\r)\leq x\leq 2\sqrt{t}+\sech\l(\sqrt{t}\r)$.
For $0 \leq x\leq \sech\l(\sqrt{t}\r)$ and $2\sqrt{t}+\sech\l(\sqrt{t}\r)\leq x\leq 2\sqrt{t}+2\sech\l(\sqrt{t}\r)$, we use linear interpolation (of slope $1$) to get a continuous function with compact support. Denote $\tau_1=\l\lfloor\frac{\sech\l(\sqrt{t}\r)n^{1/3}}{\sqrt{t}}\r\rfloor$ and $\tau_2=\left\lceil\frac{\l(2\sqrt{t}+\sech\l(\sqrt{t}\r)\r)n^{1/3}}{\sqrt{t}}\right\rceil$ and $g(t)=2\sqrt{t}+2\sech\l(\sqrt{t}\r)$. Following \cite{RRV}, we define $v\l(k\r)=f_{t}\l(\frac{k\sqrt{t}}{n^{1/3}}\r)$ for $k\leq p$, and $v\l(k\r)=0$ otherwise.
By the definition of $v$ we have,
\begin{align}\label{eq: l2 norm of vector }
\sum\limits_{k=1}^{n}v^2\l(k\r)&=
% \sum_{k=1}^{p}f^2_t\l(\frac{k\sqrt{t}}{n^{1/3}} \r)\nonumber\\
\sum\limits_{k=1}^{\tau_1}\frac{k^2t}{n^{2/3}} +\sum\limits_{k=\tau_2}^{p}\l(g(t)-\frac{k\sqrt{t}}{n^{1/3}}\r)^2 +\sum\limits_{k\geq \tau_1+1}^{\tau_2-1}\sech^2\l(\frac{k\sqrt{t}}{n^{1/3}}-\sqrt{t}-\sech\l(\sqrt{t}\r)\r).
 \end{align}

As $\sech(x)$ decays exponentially, for any $\varepsilon>0$, one can choose $n_{\varepsilon}, t_{\varepsilon},\gamma(\vep)$ such that for all $n\geq n_{\varepsilon}$ and $t_{\varepsilon}\leq t\leq \gamma(\vep)n^{\frac{2}{3}}$, we have that $\sum\limits_{k=1}^{\tau_1}\frac{k^2t}{n^{2/3}} +\sum\limits_{k=\tau_2}^{p}\l(g(t)-\frac{k\sqrt{t}}{n^{1/3}}\r)^2\leq \varepsilon \frac{n^{1/3}}{\sqrt{t}}$.
For any $\varepsilon>0$, one can choose $M_{\varepsilon}>0$ such that 
\begin{align}\label{eq: sech tail integrals}
    \int_{-\infty}^{-M_{\varepsilon}}\sech^2(x)dx,\ \int_{M_{\varepsilon}}^{\infty}\sech^2(x)dx\leq \varepsilon.
\end{align}

 This allows us to write, for large enough $t$ (in particular, $\sqrt{t}>M_\vep$) and $n$, we have 
 \begin{align*}
     \sum\limits_{k\geq 0}^{\infty}\sech^2\l(\frac{-k\sqrt{t}}{n^{1/3}}-M_{\varepsilon}\r)+\sech^2\l(\frac{k\sqrt{t}}{n^{1/3}}+M_{\varepsilon}\r)\leq \varepsilon \frac{n^{1/3}}{\sqrt{t}}.
 \end{align*} 
 Here we have used \eqref{eq: Riemmann sum convergence} and \eqref{eq: sech tail integrals}. Hence we also have
 \begin{align*}
 \sum\limits_{k\geq 0}^{\infty}\sech^2\l(\frac{-k\sqrt{t}}{n^{1/3}}-\sqrt{t}\r)+\sech^2\l(\frac{k\sqrt{t}}{n^{1/3}}+\sqrt{t}\r)\leq \varepsilon \frac{n^{1/3}}{\sqrt{t}}.    
 \end{align*}
  This implies that, for large enough $n$ and $t$, 
 \begin{align}\label{eq: l2 norm calc1}
     \l|\int_{-\infty}^{\infty}\sech^2\l(x\r)dx-\frac{\sqrt{t}}{n^{1/3}}\sum\limits_{k\geq \tau_1+1}^{\tau_2-1}\sech^2\l(\frac{k\sqrt{t}}{n^{1/3}}-\sqrt{t}-\sech\l(\sqrt{t}\r)\r)\r|\leq \varepsilon.
 \end{align}
 By repeating the exact same argument for $\sum\limits_{k=1}^{n}v^4\l(k\r)$, we have
 \begin{align}\label{eq: right tail L2,L4 norm}
\l|\sum\limits_{k=1}^{n}v^2\l(k\r)-2\frac{n^{1/3}}{\sqrt{t}}\r|\leq \frac{n^{1/3}}{\sqrt{t}}\varepsilon\quad \quad \l|\sum\limits_{k=1}^{n}v^4\l(k\r)-\frac{n^{1/3}}{\sqrt{t}}\l(\frac{4}{3}\r)\r|\leq \frac{n^{1/3}}{\sqrt{t}}\varepsilon .
 \end{align}

Choosing $\sqrt{t}>M_{\varepsilon}$ and  by choosing $t,n$ large enough, by similar argument it can be checked that
\begin{align*}
    \l|\sum\limits_{k=\tau_1+1}^{\tau_2-2}\l(\Delta\l(v\l(k\r)\r)\r)^2-\frac{\sqrt{t}}{n^{1/3}}\int_{-\infty}^{\infty}\l(\sech'\l(x\r)\r)^2dx\r|\leq \varepsilon\frac{\sqrt{t}}{n^{1/3}}.
\end{align*}
Hence we have
% &\leq 2\frac{t}{3n^{2/3}}\l(\frac{\sech\l(\sqrt{t}\r)n^{1/3}}{\sqrt{t}}+1\r)^{3}+\sum\limits_{-\infty}^{\infty}sech^2\l(\frac{k\sqrt{t}}{n^{1/3}}-\sqrt{t}-\sech\l(\sqrt{t}\r)\r)

% $\sum\limits_{k=1}^{n}v^2\l(k\r)=\sum\limits_{k=1}^{p}v^2\l(k\r)=\sum_{k=1}^{\frac{r\l(\vep\r)n^{1/3}}{\sqrt{t}}}f^2_\vep\l(\frac{k \sqrt{t}}{n^{1/3}} \r)\sim  \frac{n^{1/3}}{\sqrt{t_n}}\int_0^{r\l(\varepsilon\r)} f_{\varepsilon}\l(x\r)^2 dx \\
% \red{=\frac{n^{1/3}}{\sqrt{t}}\int_{0}^{g\l(a\l(\vep\r)\r)}x^2dx+\int_{g\l(a\l(\vep\r)\r)}^{2a\l(\vep\r)+g\l(a\l(\vep\r)\r)}g^2\l(x-a\l(\vep\r)-g\l(a\l(\vep\r)\r)dx+\int_{2a\l(\vep\r)+g\l(a\l(\vep\r)\r)}^{2a\l(\vep\r)+2g\l(a\l(\vep\r)\r)}\l(x-2a\l(\vep\r)\r)-2g\l(a\l(\vep\r)\r)^2dx\r)}.$\\
% \red{Hence, by the above calculation we get
% \begin{equation}
% \label{eq: right tail lower bound sum integration 1}
% \sum\limits_{k=1}^{n}v^2\l(k\r) \sim  \frac{n^{1/3}}{\sqrt{t}} \l(2+e_1\l(\varepsilon\r)\r)
% \end{equation}
% , where $e_1\l(\vep\r) \rightarrow 0$ as $\vep \rightarrow 0.$ Similar calculations show the following}
    \begin{equation}
    \label{eq: right tail derivative norm}
\l|\sum\limits_{k=0}^{p}\l(\Delta\l(v\l(k\r)\r)\r)^2-\frac{\sqrt{t}}{n^{1/3}}\l(\frac{2}{3}\r)\r|\leq\frac{\sqrt{t}}{n^{1/3}}\varepsilon.
    \end{equation}

Next consider,
\begin{align*}
\sum\limits_{k=1}^{p}kv^2\l(k\r)=\sum\limits_{k=1}^{\tau_1}\frac{k^3t}{n^{2/3}} +\sum\limits_{k=\tau_2}^{p}\l(g(t)-\frac{k\sqrt{t}}{n^{1/3}}\r)^2k +\sum\limits_{k\geq \tau_1+1}^{\tau_2-1}k\sech^2\l(\frac{k\sqrt{t}}{n^{1/3}}-\sqrt{t}-\sech\l(\sqrt{t}\r)\r).
\end{align*}

% \textcolor{blue}{Note that $\sum\limits_{k=1}^{\tau_1}\frac{k^3t}{n^{2/3}}\leq \frac{t\l(\tau_1+1\r)^4}{4n^{2/3}}\leq \frac{n^{2/3}}{\sqrt{t}}\varepsilon$ (for large enough $t,n$). Also}
% \textcolor{blue}{
% \begin{align}
%     \sum\limits_{k=\tau_2}^{p}\l(g(t)-\frac{k\sqrt{t}}{n^{1/3}}\r)^2k &=\sum\limits_{j=0}^{p-\tau_2}\l(g(t)-\frac{\l(\tau_2+j\r)\sqrt{t}}{n^{1/3}}\r)^2\l(\tau_2+j\r)\nonumber\\
%     &=\sum\limits_{j=0}^{p-\tau_2}\l(2\sqrt{t}+2\sech\l(\sqrt{t}\r)-\frac{\l(\left\lceil\frac{\l(2\sqrt{t}+\sech\l(\sqrt{t}\r)\r)n^{1/3}}{\sqrt{t}}\right\rceil+j\r)\sqrt{t}}{n^{1/3}}\r)^2\l(\tau_2+j\r)\nonumber\\
%     &\leq p\l(2\sqrt{t}+2\sech\l(\sqrt{t}\r)-\frac{p\sqrt{t}}{n^{1/3}}\r)^2+ \sum\limits_{j=0}^{p-\tau_2-1}\l(\sech\l(\sqrt{t}\r)-\frac{j\sqrt{t}}{n^{1/3}}\r)^2\l(\tau_2+j\r)\nonumber\\
%     &\leq \frac{n^{2/3}}{\sqrt{t}}\varepsilon \quad (\text{for large enough }t,n ).
% \end{align}
% }

Writing $a(k,t,n)=\l(\frac{k\sqrt{t}}{n^{1/3}}-\sqrt{t}-\sech\l(\sqrt{t}\r)\r)$, we have
\begin{align}
    \sum\limits_{k\geq \tau_1+1}^{\tau_2-1}k\sech^2\l(\frac{k\sqrt{t}}{n^{1/3}}-\sqrt{t}-\sech\l(\sqrt{t}\r)\r)&=\frac{n^{1/3}}{\sqrt{t}}\sum\limits_{k\geq \tau_1+1}^{\tau_2-1}\frac{k\sqrt{t}}{n^{1/3}}\sech^2\l(\frac{k\sqrt{t}}{n^{1/3}}-\sqrt{t}-\sech\l(\sqrt{t}\r)\r)\nonumber\\
    &=\frac{n^{1/3}}{\sqrt{t}}\sum\limits_{k\geq \tau_1+1}^{\tau_2-1}a(k,t,n)\sech^2\l(a(k,t,n)\r)\nonumber\\
    &+\l(\sqrt{t}+\sech\l(\sqrt{t}\r)\r)\frac{n^{1/3}}{\sqrt{t}}\sum\limits_{k\geq \tau_1+1}^{\tau_2-1}\sech^2\l(a(k,t,n)\r)\nonumber.
\end{align}
% \textcolor{blue}{
%     Using \eqref{eq: Riemmann sum convergence} as done in \eqref{eq: l2 norm calc1} we get,
%     \begin{align*}
%         \l|\sum\limits_{k\geq \tau_1+1}^{\tau_2-1}a(k,t,n)\sech^2\l(a(k,t,n)\r)-\frac{n^{1/3}}{\sqrt{t}}\int_{-\infty}^{\infty}x\sech^2\l(x\r)dx\r|&\leq\frac{n^{1/3}}{\sqrt{t}}\varepsilon\nonumber\\
%         \l|\sum\limits_{k\geq \tau_1+1}^{\tau_2-1}\sech^2\l(a(k,t,n)\r)-\frac{n^{1/3}}{\sqrt{t}}\int_{-\infty}^{\infty}\sech^2\l(x\r)dx\r|&\leq\frac{n^{1/3}}{\sqrt{t}}\varepsilon.\nonumber
%     \end{align*}
% }
{
    It can be checked that
    \begin{align}\label{eq: kv(k)^2 final bound}
        \l|\sum\limits_{k=1}^{p}kv^2\l(k\r)-\frac{n^{2/3}}{\sqrt{t}}2\r|&\leq \l|\sum\limits_{k\geq \tau_1+1}^{\tau_2-1}k\sech^2\l(\frac{k\sqrt{t}}{n^{1/3}}-\sqrt{t}-\sech\l(\sqrt{t}\r)\r)-\frac{n^{2/3}}{\sqrt{t}}2\r|+\frac{n^{2/3}}{\sqrt{t}}2\varepsilon\nonumber\\
        &\leq \frac{n^{2/3}}{\sqrt{t}}3\varepsilon
    \end{align}
}

and also
    \begin{equation}
    \label{eq: right tail derivative square norm}
\l|\sum\limits_{k=0}^{p}\l(\Delta\l(v^2\l(k\r)\r)\r)^2-\frac{\sqrt{t}}{n^{1/3}}\int_{-\infty}^{\infty}\l(2ff'\l(x\r)\r)^2dx\r|\leq\frac{\sqrt{t}}{n^{1/3}}\varepsilon.
    \end{equation}
    
%     \begin{equation}
%     \label{eq: right tail lower bound sum integration 3}
%     \sum\limits_{k=1}^{n}kv^2\l(k\r)\sim  \frac{n^{2/3}}{{t_n}}\int_0^{r\l(\varepsilon\r)} xf_{\varepsilon}\l(x\r)^2 dx\sim  \frac{n^{2/3}}{{t_n}} \l(2+e_3\l(\varepsilon\r)\r)
%     \end{equation}
%     \begin{equation}
%     \label{eq: right tail lower bound sum integration 4}
% \sum\limits_{k=1}^{n}\Delta^2\l(v\l(k\r)\r)\sim  \frac{\sqrt{t_n}}{n^{1/3}}\int_0^{r\l(\varepsilon\r)} f'_{\varepsilon}\l(x\r)^2 dx\sim   \frac{\sqrt{t_n}}{n^{1/3}}\l(\frac{2}{3}+e_4\l(\varepsilon\r)\r)
% \end{equation}
% \begin{equation}
% \label{eq: right tail lower bound sum integration 5}
% \sum\limits_{k=1}^{n}\Delta^2\l(v^2\l(k\r)\r)=O\l(\frac{\sqrt{t_n}}{n^{1/3}}\r),
% \end{equation}

% where $e_i\l(\varepsilon\r)\rightarrow 0$ as $\varepsilon\rightarrow 0$ (since we have $a\l(\varepsilon\r)\rightarrow \infty$).

Note that from \eqref{eq: Gaussian for right tail lower bound} we need a lower bound for the probability of the event 
\begin{equation}
\label{eq: Gaussian event for right tail lower bound}
    Z \geq \frac{\sqrt{n}\sum\limits_{k=0}^{p}\l(\Delta\l(v\l(k\r)\r)\r)^2+\frac{1}{\sqrt{n}}\sum\limits_{k=1}^{p} \l(k+1\r) v^2\l(k\r)+tn^{-1/6}\<v,v\>}{\sqrt{\frac{4}{\beta}\sum\limits_{k=1}^{p}v^4\l(k\r)-\frac{1}{\beta}\sum\limits_{k=0}^{p}\l(v^2\l(k+1\r)-v^2\l(k\r)\r)^2}}.
\end{equation}

{\begin{remark}
     For given $\varepsilon$, the step size should be sufficiently small to approximate integral by Riemann sums. Choose $\gamma(\vep)$ such that for all step sizes smaller than $\gamma(\vep)$, the difference in \eqref{eq: Riemmann sum convergence} is less than $\varepsilon$. Choose $t$ large enough such that difference in sums and integrals in \eqref{eq: l2 norm calc1} and others hold. Thus for the range of $t_{\varepsilon}\leq t\leq \gamma(\vep)n^{2/3}$, the step size $\frac{\sqrt{t}}{n^{1/3}}$ is sufficiently small to approximate integrals by Riemann sums up to error $\varepsilon$. 
\end{remark}
}
Using the bounds from \eqref{eq: right tail L2,L4 norm},\eqref{eq: right tail derivative norm} and \eqref{eq: kv(k)^2 final bound} we consider,
\begin{align*}
    Z\geq \frac{\sqrt{n}\frac{\sqrt{t}}{n^{1/3}}\l(\frac{2}{3}+\varepsilon\r)+\frac{1}{\sqrt{n}}\l(\frac{n^{2/3}}{\sqrt{t}}\l(2+\varepsilon\r)+\frac{n^{1/3}}{\sqrt{t}}\l(2+\varepsilon\r)\r)+\frac{t}{n^{1/6}}\frac{n^{1/3}}{\sqrt{t}}\l(2+\varepsilon\r)}{\sqrt{\frac{4n^{1/3}}{\beta\sqrt{t}}\l(\frac{4}{3}-\varepsilon\r)-\frac{1\sqrt{t}}{\beta n^{1/3}}\l(c+\varepsilon\r)}},
\end{align*}
where $c= \int_{-\infty}^{\infty}\l(2ff'\l(x\r)\r)^2dx$. Note that we have taken slightly larger value for numerator and lower value for denominator as we are interested in lower bound. Hence we get for any $\vep>0,$ there exists $n_{\varepsilon},t_{\varepsilon}$ sufficiently large such that for all $n\geq n_{\varepsilon}$ and $t\geq t_{\varepsilon}$, the probability of the event that we want is lower bounded by the probability of following event
\[
Z \geq \frac{2\sqrt{\frac \beta3}t^{3/4}\l(1+\vep\r)}{\l(1-\vep\r)},
\]
where $Z$ is a $N\l(0,1\r)$ variable. Hence for any $\vep>0$, there exist $n_{\varepsilon,\beta}$ and $t_{\varepsilon,\beta}$ such that for all $n\geq n_{\varepsilon,\beta}$ and $t_{\varepsilon,\beta}\leq t\leq \gamma(\vep)n^{2/3}$, we have \begin{align*}
    \mathbb{P}\l(\lambda_{\text{max}}\l(\wih{H}_{n,\beta,p}\r)\geq 2\sqrt{n}+tn^{-1/6}\r)\geq \exp\l(- \frac{2\beta}{3}t^{3/2}(1+\varepsilon)\r).
\end{align*}

Further, note that by Weyl's inequality we have
\[
\l\{\lambda_{\max} \l(\wih{H}_{n,\beta,p} \r) \geq 2 \sqrt{n}+t n^{-1/6}\r\} \subset \l\{ \lambda_{\max} \l(H_{n, \beta}\r) \geq 2 \sqrt{n}+t n^{-1/6}-{\lambda_{\max} \l(\wih{H}_{n,\beta,p}-H_{n, \beta} \r)}\r\}.
\]
We take $x=\varepsilon t n^{1/3}$ 
% Note that $\P\l(|Z_i| \geq \sqrt{2x}\r) \leq e^{-x}$ and $\P\l(|Z_i| \geq \sqrt{\beta\l(n-i\r)}\r) \leq e^{-\frac{\beta\l(n-p\r)}{2}}$.
% % We have $\P \l(\chi_{\beta \l(n-i\r)} \leq \frac 12 \sqrt{\l(n-i\r) \beta}\r) \leq e^{-c \sqrt{\beta\l(n-i\r)}} \leq e^{-n^{1/4}}$ for sufficiently large $n$.
% Hence, we have with probability at least $1-2e^{-x}-e^{-\frac{\beta\l(n-p\r)}{2}}$,
% \begin{equation}
% \label{eq: estimate of off diagonals of the difference}
% |\chi_{\beta\l(n-i\r)}-\l(\sqrt{\l(n-i\r)\beta}+Z_i/\sqrt{2}\r)| \leq  \frac{C\l(\log n +x\r)+x}{\l(1-\frac{1}{\sqrt{2}}\r) \sqrt{\l(n-i\r)\beta}} \leq c' \frac{C\l(\log n+x\r)+x}{ \sqrt{n}},
% \end{equation}
% for some $c'>0$ and $n$ sufficiently large.  
%  Now, by \eqref{eq: estimate of off diagonals of the difference} and Gershgorin circle theorem (Theorem $6.1.1$ of \cite{HJ}) it follows that with probability at least $1-C \sqrt{n}\l(2e^{-x}+e^{-\frac{\beta\l(n-p\r)}{2}}\r)$ (for some constant $C>0$ and $n$ sufficiently large) all eigen values $\lambda$ of $\wih{H}_{n,\beta,p}-T_{n, \beta}$ have $|\lambda| \leq 2  c_{\beta} \frac{C\l(\log n+x\r)+x}{ \sqrt{n}} \leq \vep tn^{-1/6}$ for any small $\vep>0$ and sufficiently large $n$. 
and apply Proposition \ref{eq: estimate of off diagonals of the difference}, to get for sufficiently large $n$
\[
\P \l( \lambda_{\max} \l(H_{n, \beta}\r) \geq 2 \sqrt{n}+ \l(1-\vep\r) tn^{-1/6}\r)+C n\l(3e^{-x}+2e^{-\frac{\beta\l(n-p\r)}{2}}\r)\geq \P \l( \wih{H}_{n,\beta,p}  \geq 2 \sqrt{n}+t n^{-1/6} \r).
\]
By elementary calculations one can show that for any $\beta>0$ and any small $\varepsilon>0$ there exists an $n_{\varepsilon,\beta}$ and $t_{\varepsilon,\beta}$ such that for $n\geq n_{\varepsilon,\beta}$ and $t_{\varepsilon,\beta}\leq t\leq \gamma(\vep)n^{\frac{2}{3}}$, we have that
\begin{align*}
    \mathbb{P}\l(\lambda_{\text{max}}\l(H_{n,\beta}\r)\geq 2\sqrt{n}+tn^{-1/6}\r)\geq \exp\l(- \frac{2\beta}{3}t^{3/2}\l(1+\varepsilon\r)\r).
\end{align*}

% Using Theorem $1.1$ of \cite{BCG14}, $d_{TV}\l(\chi_{\nu}-\sqrt{\nu},\frac{1}{\sqrt{2}}N\l(0,1\r)\r)\leq C/\sqrt{\nu}$. As a result, for $n'=\l\lfloor n^{\frac{1}{2}-\delta}\r\rfloor$ we have 
% \begin{align*}
%     d_{TV}\l(\l(X_j\r)_{1\leq j\leq n'}, \l(\sqrt{\beta\l(n-j\r)}+\frac{1}{\sqrt{2}}Z_j\r)_{1\leq j\leq n'}\r)\leq \frac{C}{\sqrt{\beta}n^{\delta}}.
% \end{align*}

% We have $\mathbb{P}\l( \lambda_{\text{max}}\l(H_{n,\beta}\r)\geq 2\sqrt{n}+t_nn^{-1/6}\r)-\mathbb{P}\l( \lambda_{\text{max}}\l(\wih{H}_{n,\beta,p}\r)\geq 2\sqrt{n}+t_nn^{-1/6}\r)\leq\frac{C}{\sqrt{\beta}n^{\delta}}$.\\
% For any $\varepsilon>0$, taking $t_n\leq \l(\delta\log\l(n\r)\r)^{2/3}$ and $t_n\rightarrow\infty$, where $\delta<1/6$, we have 
% \begin{align*}
% \mathbb{P}\l( \lambda_{\text{max}}\l(H_{n,\beta}\r)\geq 2\sqrt{n}+t_nn^{-1/6}\r)\geq \exp\l(- \l(1+o\l(1\r)\r)\l(\frac{2\beta}{3}+\varepsilon\r)t_n^{3/2}\r).
% \end{align*}
\subsection{Laguerre Ensemble}\label{subsec: Laguerre lower for left}

Fix $\beta>0$ and fix small $\varepsilon>0$.
Let $m=m\l(n\r)\geq n$ be such that for all $n\geq n_{\varepsilon}$, we have $\frac{m}{n}\leq M$ for some $M<\infty$. We denote $\gamma:=m/n$ and $a=\l(1+\sqrt{\gamma}\r)^2$ and $b=\gamma^{-1/6}\l(1+\sqrt{\gamma}\r)^{4/3}.$ We consider  matrix $\wih{L}_{n,m,\beta,p}$ where $p:= \left\lceil \l(2\sqrt{t}+2\sech\l(\sqrt{t}\r)\r)\frac{n^{1/3}\gamma^{1/4}}{\sqrt{tb}}\right\rceil$(recall the definition of $\l(\wih{L}_{n,m,\beta,p}\r)$ from Proposition \ref{eq: estimate of off diagonals of the difference}). We obtain lower bound for probability of the event that $\lambda_{\text{max}}\l(\wih{L}_{n,m,\beta,p}\r)\geq an\l(1+\frac{tbn^{-2/3}}{a}\r)$. Here $t_{\varepsilon}\leq t\leq \rho(\vep,M)n^{2/3}$. The values $t_{\varepsilon},\rho(\vep,M)$ will be made precise later. As before for $\varepsilon>0$, we construct a non-negative vector $v$ (this vector depends on $n,m,t$ and $\varepsilon$) such that $v\l(i\r)=0$ for all $i> p$ and $v\l(0\r)=0$ to show that for all $n\geq n_{\varepsilon,\beta}$ and  $t_{\varepsilon,\beta}\leq t\leq \rho(\vep,M)n^{2/3}$, we have
\begin{align}\label{eq: Laguerre before KMT}
    \langle \wih{L}_{n,m,\beta,p}v,v\rangle\geq an\l(1+\frac{tbn^{-2/3}}{a}\r)\langle v,v\rangle
\end{align}
with probability at least $\exp\l(- \frac{2\beta}{3}t^{3/2}\l(1+\varepsilon\r)\r)$. Note that $\wih{L}_{n,m,\beta,p}$ can be written as a sum of two matrices $L'$ and $L''$, where $L''$ is a tridiagonal symmetric matrix which contains entries which are linear combinations of $\zeta_k^2,\l(\zeta'_k\r)^2, \zeta_{k+1}\zeta'_k$ in the top $p\times p$ submatrix and zeros elsewhere. $L'$ contains random variables which are Gaussian in the top $p\times p$ submatrix. We first show that for the range $t_{\varepsilon}\leq t\leq \rho(\vep,M)n^{\frac{2}{3}}$,
\begin{align*}
    \langle L'v,v\rangle\geq an\l(1+\frac{tbn^{-2/3}}{a}\r)\langle v,v\rangle
\end{align*}
with probability at least $\exp\l(- \frac{2\beta}{3}t^{3/2}\l(1+\varepsilon\r)\r)$. We take $v\l(k\r)=0$ for $k\geq p$ and use $\sqrt{n-k}\geq \sqrt{n}-\frac{k+1}{2\sqrt{n}}$ for $k\leq 2\sqrt{n}-1$. We choose $n_{\vep}$ such that $p\leq 2\sqrt{n}-1$ for all $n\geq n_{\vep}$.
Now $\langle L'v,v\rangle$ is a Gaussian random variable with mean $\mu$ and variance $\sigma^2$ given by,
\begin{align*}
    \mu&= \sum\limits_{k=1}^{p} \l(\l(m+1-k\r)+\l(n-k\r)\r)v^2\l(k\r) + 2\sum\limits_{k=1}^{p-1} \l(\sqrt{\l(m-k\r)\l(n-k\r)}\r)v\l(k\r)v\l(k+1\r)\nonumber\\
    &\geq \sum\limits_{k=1}^{p} \l(m+n-2k\r)v^2\l(k\r) +\sum\limits_{k=1}^{p-1} \l(\sqrt{m}-\frac{k{+1}}{2\sqrt{m}}\r)\l(\sqrt{n}-\frac{k+{1}}{2\sqrt{n}}\r)\l(v^2\l(k\r)+v^2\l(k+1\r)-\l(\Delta\l(v\l(k\r)\r)\r)^2 \r)\nonumber\\
    &\geq \l(\sqrt{m}+\sqrt{n}\r)^2\langle v,v\rangle-\sum\limits_{k=1}^{p} \l(2k+\frac{\l({m}+{n}\r)\l(k+1\r)}{\sqrt{mn}}\r)v^2\l(k\r)-\sum\limits_{k=0}^{p}\l(\sqrt{mn}+\frac{\l(k+1\r)^2}{4\sqrt{mn}}\r)\l(\Delta\l(v\l(k\r)\r)\r)^2. 
\end{align*}
\begin{align*}
    \sigma^2&=  \frac{2}{\beta}\sum\limits_{k=0}^{p-1}\l({\sqrt{m-k}\ v^2\l(k+1\r)}+\sqrt{n-k}\ v\l(k\r)v\l(k+1\r)\r)^2\\  
    &+\frac{2}{\beta}\sum\limits_{k=1}^{p} \l({\sqrt{n-k}\ v^2\l(k\r)}+\sqrt{m-k}\ v\l(k\r)v\l(k+1\r)\r)^2. 
\end{align*}
We are interested in the event 
\begin{align*}
    \mu+\sigma Z \geq \l(\sqrt{m}+\sqrt{n}\r)^2\langle v,v\rangle + tbn^{1/3}\sum\limits_{k=1}^{p} v^2\l(k\r).
\end{align*}
{
Note that $\mu\leq \l(\sqrt{m}+\sqrt{n}\r)^2\langle v,v\rangle + tbn^{1/3}\sum\limits_{k=1}^{p} v^2\l(k\r) $.} As a result, as in the case of Hermite ensemble, we can take the lower bound of $\mu$ to bound the probability of the above event. It suffices to find a lower bound for the probability of the event,
\begin{equation}\label{eq: Laguerre event lower bound}
    \sigma Z\geq  tbn^{1/3}\sum\limits_{k=1}^{p} v^2\l(k\r) +\sum\limits_{k=1}^{p} \l(2k+\frac{\l({m}+{n}\r)\l(k+1\r)}{\sqrt{mn}}\r)v^2\l(k\r)+\sum\limits_{k=0}^{p}\l(\sqrt{mn}+\frac{\l(k+1\r)^2}{4\sqrt{mn}}\r)\l(\Delta\l(v\l(k\r)\r)\r)^2. 
    %\l(2+\frac{\sqrt{m}+\sqrt{n}}{\sqrt{mn}}\r)\sum\limits_{i=1}^{n'} iv\l(i\r)^2+\red{\l(\frac{1}{ \sqrt{m}}+\frac{1}{\sqrt{n}}\r)\sum_{k=1}^{n'-1}v^2\l(k\r)} &\\+ \sqrt{mn}\sum\limits_{i=1}^{n'} \Delta\l(v\l(i\r)\r)^2+\frac{1}{4\sqrt{mn}}\sum\limits_{i=1}^{n'} i^2\Delta\l(v\l(i\r)\r)^2.
\end{equation}

We also have that 
\begin{align*}
    \sigma^2
    % \geq \frac{2}{\beta}\l(m+n-2p\r)\sum\limits_{k=1}^{p}\l(v^4\l(k\r)+v^2\l(k\r)v^2\l(k+1\r)\r)\\
    % &+\frac{4}{\beta}\sqrt{\l(m-p\r)\l(n-p\r)}\sum\limits_{k=1}^{p}\l(v^3\l(k\r)v\l(k+1\r)+v^3\l(k+1\r)v\l(k\r)\r)\\
    &\geq \l(\frac{2}{\beta}\l(m+n-2p\r)\r)\l(\sum\limits_{k=1}^{p}2v^4\l(k\r)-\frac{1}{2}\sum\limits_{k=0}^{p}\left(\Delta\l(v^2\l(k\r)\r)\right)^2\r)\nonumber\\
    &\quad +\frac{2}{\beta}\sqrt{\l(m-p\r)\l(n-p\r)}\l(\sum\limits_{k=1}^{p}4v^4\l(k\r)-\sum\limits_{k=0}^{p}\l(\Delta\l(v^2\l(k\r)\r)\r)^2-\sum\limits_{k=0}^{p}\l(\Delta\l(v\l(k\r)\r)\r)^2\l(v^2\l(k\r)+v^2\l(k+1\r)\r)\r)
\end{align*}

\begin{align*}
    \sigma^2&\geq \frac{4}{\beta}\l(\sqrt{m-p}+\sqrt{n-p}\r)^2\l(\sum\limits_{k=1}^{p}v^4\l(k\r)-\frac{1}{4}\sum\limits_{k=0}^{p}\l(\Delta\l(v^2\l(k\r)\r)\r)\r)^2\nonumber\\
    &\quad\quad-\frac{2}{\beta}\sqrt{\l(m-p\r)\l(n-p\r)}\sum\limits_{k=0}^{p}\l(\Delta\l(v\l(k\r)\r)\r)^2\l(v^2\l(k\r)+v^2\l(k+1\r)\r)\nonumber.
\end{align*}

% \red{
% $\sigma^2 \geq \frac 2\beta \sum_{i=1}^{n'}\l(\l(m+n-2i\r)v^4\l(i\r)\r)+\sum_{i=1}^{n'}\frac 2\beta \l(m+n-2i\r)\l( v^2\l(i\r)v^2\l(i+1\r)\r)+\frac 2\beta \sum_{i=1}^{n'}2\sqrt{\l(m-i\r)\l(n-i\r)}v\l(i\r)v\l(i+1\r)\l(v^2\l(i\r)+v^2\l(i+1\r)\r)$\\
% $= \frac 2\beta \sum_{i=1}^{n'}\l(\l(m+n-2i\r)v^4\l(i\r)\r)+\sum_{i=1}^{n'} \l(m+n-2i\r)\frac 1\beta \l(v^4\l(i\r)+v^4\l(i+1\r)-\Del\l(v^2\l(i\r)\r)^2\r)\\+\frac 2\beta \sum_{i=1}^{n'}\sqrt{\l(m-i\r)\l(n-i\r)}\l(v^2\l(i\r)+v^2\l(i+1\r)-\Del\l(v\l(i\r)\r)^2 \r)\l(v^2\l(i\r)+v^2\l(i+1\r)\r)\\\geq -\frac 1\beta\l(\sqrt{m}+\sqrt{n}\r)^2v^4\l(1\r)+ \frac 2\beta \sum_{i=1}^{n'}\l(\l(2m+2n-4i\r)v^4\l(i\r) \r)-\frac 1\beta\sum_{i=1}^{n'}2i\Del\l(v^2\l(i\r)\r)^2+\frac 4\beta \sum_{i=1}^{n'}\sqrt{\l(m-i\r)\l(n-i\r)}v^4\l(i\r)\\-\frac 2\beta \sum_{i=1}^{n'}\sqrt{\l(m-i\r)\l(n-i\r)}\l(v^2\l(i\r)+v^2\l(i+1\r)\r)\Del\l(v\l(i\r)\r)^2\\=-\frac 1\beta \l(\sqrt{m}+\sqrt{n}\r)^2v^4\l(1\r)+\frac 4\beta \sum_{i=1}^{n'}\l(\sqrt{m}+\sqrt{n}\r)^2v^4\l(i\r)-\frac8\beta \sum_{i=1}^{n'}iv^4\l(i\r)-\frac 1\beta \sum_{i=1}^{n'}2i\Del\l(v^2\l(i\r)\r)^2+\frac4\beta \sum_{i=1}^{n'}\l(\sqrt{\l(m-i\r)\l(n-i\r)}-\sqrt{mn}\r)v^4\l(i\r)\\-\frac 2\beta \sum_{i=1}^{n'}\sqrt{mn}\l(v^2\l(i\r)+v^2\l(i+1\r)\r)\Del\l(v\l(i\r)\r)^2.$
% }
We denote $\nu$ to be the LHS of above inequality.
Let $f_{t}\l(x\r)$ be the function as in Hermite case. Denote $\tau_1=\l\lfloor\frac{\sech\l(\sqrt{t}\r)n^{1/3}\gamma^{1/4}}{\sqrt{t}\sqrt{b}}\r\rfloor$ and $g(t)=2\sqrt{t}+2\sech\l(\sqrt{t}\r)$ and $\tau_2=\left\lceil\frac{\l(2\sqrt{t}+\sech\l(\sqrt{t}\r)\r)n^{1/3}\gamma^{1/4}}{\sqrt{t}\sqrt{b}}\right\rceil$.  We define $v\l(k\r)=f_{t}\l(\frac{k\sqrt{tb}}{\gamma^{1/4}n^{1/3}}\r)$ for $k\leq p$ and $0$ otherwise. By the choice of vector $v$ and comparing sums with integrals as we did in the Hermite case, one can check that for any given $\varepsilon>0$, one can choose $n_{\varepsilon}, \rho(\vep,M)$ and $t_{\varepsilon}$ such that for all $t_{\varepsilon}\leq t \leq \rho(\vep,M)n^{\frac{2}{3}}$ and for all $n\geq n_{\varepsilon}$ with $\frac{m}{n}\leq M$ for some $M<\infty$, the following bounds hold.

\begin{align*}
%\label{eq: Laguerre right tail L2,L4 norm}
\l|\sum\limits_{k=1}^{n}v^2\l(k\r)-2\frac{n^{1/3}\gamma^{1/4}}{\sqrt{t}\sqrt{b}}\r|\leq \frac{n^{1/3}\gamma^{1/4}}{\sqrt{t}\sqrt{b}}\varepsilon\quad \quad \l|\sum\limits_{k=1}^{n}v^4\l(k\r)-\frac{n^{1/3}\gamma^{1/4}}{\sqrt{t}\sqrt{b}}\l(\frac{4}{3}\r)\r|\leq \frac{n^{1/3}\gamma^{1/4}}{\sqrt{t}\sqrt{b}}\varepsilon .
 \end{align*}

 \begin{align*}
    %\label{eq: Laguerre right tail derivative norm}
\l|\sum\limits_{k=0}^{p}\l(\Delta\l(v\l(k\r)\r)\r)^2-\frac{\sqrt{t}\sqrt{b}}{n^{1/3}\gamma^{1/4}}\l(\frac{2}{3}\r)\r|\leq\frac{\sqrt{t}\sqrt{b}}{n^{1/3}\gamma^{1/4}}\varepsilon.
    \end{align*}

 \begin{align*}
    %\label{eq: Laguerre right tail derivative square norm}
\l|\sum\limits_{k=0}^{p}\l(\Delta\l(v^2\l(k\r)\r)\r)^2-\frac{\sqrt{t}\sqrt{b}}{n^{1/3}\gamma^{1/4}}\int_{-\infty}^{\infty}\l(2ff'\l(x\r)\r)^2dx\r|\leq\frac{\sqrt{t}\sqrt{b}}{n^{1/3}\gamma^{1/4}}\varepsilon.
    \end{align*}

    \begin{align*}
    %\label{eq: Laguerre right tail derivative square times vector squared norm}
\l|\sum\limits_{k=0}^{p}\l(\Delta\l(v\l(k\r)\r)\r)^2\l(v^2\l(k\r)+v^2\l(k+1\r)\r)-\frac{\sqrt{t}\sqrt{b}}{n^{1/3}\gamma^{1/4}}\int_{-\infty}^{\infty}2\l(ff'\l(x\r)\r)^2dx\r|\leq\frac{\sqrt{t}\sqrt{b}}{n^{1/3}\gamma^{1/4}}\varepsilon.
    \end{align*}

\begin{align*}
%\label{eq: Laguerre kv(k)^2 final bound}
    \l|\sum\limits_{k=1}^{p}kv^2\l(k\r)-\frac{n^{2/3}\sqrt{\gamma}}{\sqrt{t}b}2\r|\leq \frac{n^{2/3}\sqrt{\gamma}}{\sqrt{t}b}\varepsilon.
\end{align*}

% { $\sigma^2 \geq x_{1,n}+\frac 4\beta \l(\sqrt{m}+\sqrt{n} \r)^2\l(\frac{\gamma^{1/4}n^{1/3}}{\sqrt{t_nb}}\l(\frac 43+e_2\l(\vep\r)\r)\l(1+x_{2,n}\r)\r)-\frac8\beta\l(\frac{\gamma^{1/2}n^{2/3}}{t_nb}\l(c_1+e_6\l(\vep\r)\r)\l(1+x_{3,n}\r)\r)-\frac 1\beta\l(\l(c_2+e_7\l(\vep\r)\r)\l(1+x_{4,n}\r)\r)-\frac4\beta\l(\sqrt{\l(m-i\r)\l(n-i\r)}+\sum_{i=1}^{n'} \sqrt{mn}\r)\l(\frac{\gamma^{1/4}n^{1/3}}{\sqrt{t_nb}}\r)v\l(i\r)^4-\frac 2\beta \sqrt{mn}\l( \frac{\sqrt{t_nb}}{\gamma^{1/4}n^{1/3}}\l(c_3+e_8\l(\vep\r)\r)\l(1+x_{5,n}\r)\r)$
% ,\\where $x_{i,n} \rightarrow 0$ and $n \rightarrow \infty$ and $e_i\l(\vep\r) \rightarrow 0$ as $\vep \rightarrow 0$.

We use the above bounds to get lower bound for the probability of the event 
\begin{align*}
    Z&\geq \frac{tbn^{1/3}\sum\limits_{k=1}^{p} v^2\l(k\r) +\sum\limits_{k=1}^{p} \l(2k+\frac{\l({m}+{n}\r)\l(k+1\r)}{\sqrt{mn}}\r)v^2\l(k\r)+\sum\limits_{k=0}^{p}\l(\sqrt{mn}+\frac{\l(k+1\r)^2}{4\sqrt{mn}}\r)\l(\Delta\l(v\l(k\r)\r)\r)^2}{\sqrt{\nu}}.
    \end{align*}

    Denoting 
    \begin{align*}
        \nu_{\ell}:= {\frac{4}{\beta}\l(\sqrt{m-p}+\sqrt{n-p}\r)^2\l(\frac{n^{1/3}\gamma^{1/4}}{\sqrt{t}\sqrt{b}}\l(\frac{4}{3}-\varepsilon\r)-\frac{1}{4}\frac{\sqrt{t}\sqrt{b}}{n^{1/3}\gamma^{1/4}}\l(\int_{-\infty}^{\infty}6\l(ff'\l(x\r)\r)^2dx+\varepsilon\r)\r)},
    \end{align*}
    one can check that $\nu\geq \nu_{\ell}$. As a result, we have that it is enough to find lower bound for the probability of the following event.
    \begin{align*}
    Z &\geq \frac{\l(tbn^{1/3}+\frac{\l(m+n\r)}{\sqrt{mn}}\r)\l(2+\varepsilon\r)\frac{n^{1/3}\gamma^{1/4}}{\sqrt{t}\sqrt{b}}+\l(2+\frac{\l({m}+{n}\r)}{\sqrt{mn}}\r)\frac{n^{2/3}\sqrt{\gamma}}{\sqrt{t}b}\l(2+\varepsilon\r)}{\sqrt{\nu_{\ell}}}\\
    &+\frac{\l(\sqrt{mn}+\frac{\l(p+1\r)^2}{4\sqrt{mn}}\r)\frac{\sqrt{t}\sqrt{b}}{n^{1/3}\gamma^{1/4}}\l(\frac{2}{3}+\varepsilon\r)}{\sqrt{\nu_{\ell}}}.
\end{align*}

% Hence, we get for any $\vep>0$, choosing $n,t$ sufficiently large such that 
% \[
% \sigma^2 \geq \frac 4\beta \l(\sqrt{m}+\sqrt{n} \r)^2\frac{\gamma^{1/4}n^{1/3}}{\sqrt{tb}}\frac 43\l(1-\vep\r).
% \]
% Similarly, by \eqref{eq: right tail lower bound sum integration 1}, \eqref{eq: right tail lower bound sum integration 2}, \eqref{eq: right tail lower bound sum integration 3}, \eqref{eq: right tail lower bound sum integration 4}, \eqref{eq: right tail lower bound sum integration 5}, we get we want the probability of the event\\
% $
% \sigma Z \geq \l(\gamma^{1/4}\sqrt{tb}n^{2/3}\l(2+e_1\l(\vep\r) \r)\l(1+x_{1,n}\r)\r)+\sqrt{mn}v\l(1\r)^2+\l(2+\frac{\sqrt{m}+\sqrt{n}}{\sqrt{mn}}\r)\l(\frac{\gamma^{1/2}n^{2/3}}{tb}\l(2+e_3\l(\vep\r)\r)\l(1+x_{3,n}\r)\r)\\+\l(\frac{1}{\sqrt{m}}+\frac{1}{\sqrt{n}} \r)\frac{\gamma^{1/4}n^{1/3}}{\sqrt{tb}}\l(\l(\frac 43+e_2\l(\vep\r)\r)\l(1+x_{2,n}\r) \r)+\sqrt{mn}\frac{\sqrt{tb}}{n^{1/3}\gamma^{1/4}}( \l(\frac 23+e_4\l(\vep\r)\r)\l(1+x_{4,n}\r).
% $
% \\
% where same as before $x_{i,n}\rightarrow 0$ as $n \rightarrow \infty$ and $e_{\l(\vep\r)} \rightarrow 0$ as $\vep \rightarrow 0$. Hence, we get for any $\vep>0$ there are sufficiently large $t$ and $n$ such that the right hand side is smaller than $\gamma^{1/4}\sqrt{tb}n^{2/3}\frac {8}{3}\l(1+\vep\r).$
Hence, for any $\vep>0$ and for sufficiently large $t$ and $n$ our desired lower bound is the lower bound of probability of the event,
\[
Z \geq \frac{\gamma^{1/4}\sqrt{tb}n^{2/3}\frac 83\l(1+\vep\r)}{\frac{2}{\sqrt{\beta}}\l(\sqrt{m}+\sqrt{n}\r)\frac{\gamma^{1/8}n^{1/6}}{\l(tb\r)^{1/4}}\frac{2}{\sqrt{3}}\l(1-\vep\r)}.
\]
By putting the value of $b$ and $\gamma$ one gets the above event is equal to
\[
Z \geq \frac{2\sqrt{\beta}t^{3/4}}{\sqrt{3}}\frac{\l(1+\vep\r)}{\l(1-\vep\r)}.
\]
Applying lower bound for Gaussian tail we get that for any $\beta>0$ and small $\varepsilon>0$ there exist $n_{\varepsilon,\beta},t_{\varepsilon,\beta}$ such that for the range $t_{\varepsilon,\beta}\leq t\leq \rho(\vep,M)n^{\frac{2}{3}}$ and $n\geq n_{\varepsilon,\beta}$ with $m$ such that $\frac{m}{n}\leq M$ for some $M<\infty$, we have
\[\langle L'v,v\rangle\geq an\l(1+\frac{tbn^{-2/3}}{a}\r)\langle v,v\rangle\]
with probability at least $\exp\l(- \frac{2\beta}{3}t^{3/2}\l(1+\varepsilon\r)\r)$.

 Using Gershgorin circle theorem (Theorem $6.1.1$ of \cite{HJ}) and using the tail behaviour of normal random variable we get that the eigenvalues of the matrix $L''$ (introduced after \eqref{eq: Laguerre before KMT}) are located in an interval $\l(-2\varepsilon tn^{\frac13},2\vep tn^{\frac13}\r)$ with probability at least $1-6p\exp\l(-\varepsilon tn^{\frac13}/2\r)$. Let us denote this event by $\ce_1.$ By Weyl's inequality and Gershgorin's circle theorem we have 
 \begin{align*}
\lambda_{\text{max}}\l(\wih{L}_{n,m,\beta,p}\r)\geq an\l(1+\frac{tbn^{-2/3}}{a}\r)-2\vep tn^{\frac13}
 \end{align*}
with probability at least $\exp\l(- \frac{2\beta}{3}t^{3/2}\l(1+\varepsilon\r)\r)-6p\exp\l(-\vep tn^{\frac13}/2\r)$.

% By repeating the same calculations as in the case of Hermite ensemble, we need the lower bound for the probability of the event,
% \begin{align*}
%     Z\geq {\l(\frac{\l(2+e_1\l(\varepsilon\r)\r)t_nbn^{1/3}}{d}+\l(\frac{2}{3}+e_2\l(\varepsilon\r)\r)\sqrt{mn}d\r)\sqrt{3d\beta}}/{4
%     \l(1+\sqrt{\gamma}\r)\sqrt{n}}
% \end{align*}
% where $d=\frac{\sqrt{t_nb}}{\gamma^{1/4}n^{1/3}}$.
% Using Gaussian bounds, we have that for any $\varepsilon>0$, the event $\lambda_{\text{max}}\l(\wih{L}_{n,m,\beta,p}\r)\geq an\l(1+\frac{t_nbn^{-2/3}}{a}\r)$ has probability at least $\exp\l(-\l(1+o\l(1\r)\r)\l(\frac{2\beta}{3}+\varepsilon\r)t_n^{3/2}\r)$.
% Using the coupling argument by KMT theorem as before, we have
% \begin{align}\label{eq: chi squared coupling difference}
%     \l|\chi^2_{\beta\l(n-i\r)}-\l(\sqrt{\l(n-i\r)\beta}+\l(\zeta'_i/\sqrt{2}\r)\r)^2\r|\leq C\l(\log n +x\r)+x \mbox{ with probability at least  }1-2e^{-x}.
% \end{align}
Further, note that by Weyl's inequality we have
\begin{align*}
&\l\{\lambda_{\max} \l(\wih{L}_{n,m,\beta,p} \r) \geq an+{tbn^{1/3}}-2\vep tn^{\frac13}\r\}\\
&\subset \l\{ \lambda_{\max} \l(L_{n,m,\beta}\r) \geq an+{tbn^{1/3}}-2\vep tn^{\frac13}-{\lambda_{\max} \l(\wih{L}_{n,m,\beta,p}-L_{n,m,\beta} \r)}\r\}.
\end{align*}
% The diagonal elements of $L_{n,m,\beta}-\wih{L}_{n,m,\beta,p}$ are of the form 
% \begin{align*}
%     \frac{1}{\beta}\l(\chi^2_{\beta\l(m-i+1\r)}-\l(\sqrt{\l(m+1-i\r)\beta}+\l(\zeta_i/\sqrt{2}\r)\r)^2+\chi^2_{\beta\l(n-i\r)}-\l(\sqrt{\l(n-i\r)\beta}+\l(\zeta'_i/\sqrt{2}\r)\r)^2\r)
% \end{align*}

% and the off-diagonal elements are of the form  
% \begin{align}\label{eq: Matrix perturb bound 2}
%     \frac{1}{\beta}\l(\chi_{\beta\l(n-i\r)}\chi_{\beta\l(m-i\r)}-\l(\sqrt{\l(n-i\r)\beta}+\l(\zeta'_i/\sqrt{2}\r)\r)\l(\sqrt{\l(m-i\r)\beta}+\l(\zeta_{i+1}/\sqrt{2}\r)\r)\r).
% \end{align}

%  By \eqref{eq: chi squared coupling difference} all the diagonal elements of $L_{n,m,\beta}-\wih{L}_{n,m,\beta,p}$ are less than $\l(2C\l(\log m +x\r)+2x\r)/\beta$ with probability at least $1-4pe^{-x}$.
% By \ref{eq: estimate of off diagonals of the difference}, each of the off-diagonal element is less than 
% \begin{align*}
%     \l(c'\frac{C\l(\log m +x\r)+x}{\sqrt{m}}+\l(1+\frac{1}{\sqrt{2}}\r)\sqrt{\beta\l(m-i\r)}\r)\l(c_{\beta}\frac{C\l(\log n +x\r)+x}{\sqrt{n}}\r)\\
%     +\l(\l(1+\frac{1}{\sqrt{2}}\r)\sqrt{\beta\l(n-i\r)}\r)\l(c_{\beta}\frac{C\l(\log m +x\r)+x}{\sqrt{m}}\r)
% \end{align*}
% with probability at least $1-4e^{-x}-4e^{-\frac{\beta\l(n-p\r)}{2}}$.
By taking $x=\vep t n^{\frac{1}{3}}$ and applying Proposition \ref{eq: estimate of off diagonals of the difference} we have that 

% By Gershgorin circle theorem, all the eigenvalues $\lambda$ of $L_{n,m,\beta}-\wih{L}_{n,m,\beta,p}$ satisfy $\l|\lambda\r|\leq \varepsilon btn^{1/3} $ with probability at least $1-4pe^{-x}-4pe^{-\frac{\beta\l(n-p\r)}{2}}$ for any $\varepsilon$, by choosing $n$ large enough. Hence we have
\begin{align*}
&\mathbb{P}\l( \lambda_{\text{max}}\l(L_{n,m,\beta}\r)\geq an+t b n^{1/3}-\varepsilon t b n^{1/3}-2\vep tn^{\frac13}\r)\\
&\geq \exp\l(- \frac{2\beta}{3}t^{3/2}\l(1+\varepsilon\r)\r)-6p\exp\l(-\vep tn^{\frac13}/2\r)-6p\exp\l(-\vep tn^{\frac13}/2\r)-8pe^{-x}-4pe^{-\frac{\beta\l(n-p\r)}{2}}.
\end{align*}

Hence we have that for any $\beta>0$ and any small $\varepsilon>0$ there exists $n_{\varepsilon,\beta}, \rho(\vep,M)$ and $t_{\varepsilon,\beta}$ such that for $n\geq n_{\varepsilon,\beta}$ and $t_{\varepsilon,\beta}\leq t\leq \rho(\vep,M)n^{\frac{2}{3}}$ with $m$ such that $\frac{m}{n}\leq M$ for some $M<\infty$, we have
\begin{align}\label{eq: Laguerre right tail lower bound final form}
\mathbb{P}\l( \lambda_{\text{max}}\l(L_{n,m,\beta}\r)\geq an+t b n^{1/3}\r)
\geq \exp\l(- \frac{2\beta}{3}t^{3/2}\l(1+\varepsilon\r)\r).
\end{align}
\section{Upper bounds for left tails} 
\label{s: ltub}
\subsection{Hermite Ensemble}

Fix $\beta>0$. Fix small $\varepsilon,\delta>0$. For this subsection, we consider the matrix $\wih{H}_{n,\beta,p}$ for $p:=\l\lfloor tn^{1/3}\r\rfloor$. We take $t_{\varepsilon}\leq t\leq n^{\frac{1}{6}-\delta}$ (so that $p\leq \sqrt{n}$) and $n\geq n_{\varepsilon}$, where $t_{\varepsilon}$ and $n_{\varepsilon}$ will be made precise later. We first obtain upper bounds for the left tails of $\lambda_{\text{max}}\l(\wih{H}_{n,\beta,p}\r)$. As $\lambda_{\text{max}}\l(\wih{H}_{n,\beta,p}\r)\< v,v\>\geq\< \wih{H}_{n,\beta,p}v,v\>$, for the event $\lambda_{\text{max}}\l(\wih{H}_{n,\beta,p}\r)\leq 2\sqrt{n}-tn^{-1/6}$ to occur, it is necessary that for any $v\in\mathbb{R}^n$, the event $\< \wih{H}_{n,\beta,p}v,v\>\leq \l(2\sqrt{n}-tn^{-1/6}\r)\< v,v\>$ should occur. We fix $t$ and consider the following function. Let $g_t\l(x\r)=x\sqrt{t}\wedge \sqrt{\l(t-x\r)^+}\wedge \l(t-x\r)^+$. Let $\rho_t$ be the positive solution of the quadratic equation in $x$, which is $x^2t+x-t=0$.  We define $v\l(k\r)=g_t\l(\frac{k}{n^{1/3}}\r)$ for all $k \leq p$ and $0$ otherwise. $\<\wih{H}_{n,\beta,p}v,v\>$ is a normal random variable with mean $\mu$ given by,
\begin{align}\label{eq: Hermite left tail mean bound}
    \mu&= 2\sum_{k=1}^{p}\sqrt{n-k}\ v\l(k\r)v\l(k+1\r)\nonumber\\
    &\geq \sum_{k=1}^{p}\l(\sqrt{n}-\frac{k+1}{2\sqrt{n}}\r)\l(v^2\l(k\r)+v^2\l(k+1\r)-\l(\Delta v\l(k\r)\r)^2 \r)\nonumber\\
    & \geq 2\sqrt{n}\langle v,v \rangle-\sqrt{n}\sum_{k=0}^{p}\l(\Delta v\l(k\r)\r)^2-\frac{1}{\sqrt{n}}\sum_{k=1}^{p}\l(k+1\r)v^2\l(k\r).
\end{align}

Variance of $\langle \wih{H}_{n,\beta,p}v,v\rangle$, denoted as $\sigma^2$ is given by,
\begin{align}\label{eq: Hermite left tail variance bound}
    \sigma^2&=
    % \frac{2}{\beta}\sum\limits_{k=1}^{p}v^4\l(k\r)+\frac{2}{\beta}\sum\limits_{k=1}^{p}v^2\l(k\r)v^2\l(k+1\r)\nonumber\\
    % &=\frac{2}{\beta}\sum\limits_{k=1}^{p}v^4\l(k\r)+\frac 1\beta\sum_{k=1}^{p}\l(v^4\l(k\r)+v\l(k+1\r)^4-\l(\Delta \l(v^2\l(k\r)\r)\r)^2 \r)\nonumber\\
    \frac{4}{\beta}\sum_{k=1}^{p}v^4\l(k\r)-\frac 1\beta\sum_{k=0}^{p}\l(\Delta \l(v^2\l(k\r)\r)\r)^2. 
\end{align}
We observe the following.

\begin{align*}
\sqrt{n}\sum_{k=0}^{p}\l(\Delta v\l(k\r)\r)^2 &=\sqrt{n}\l(\sum_{k=0}^{\l\lfloor\rho_tn^{1/3}\r\rfloor} \frac{t}{n^{2/3}}+\sum_{k=\l\lceil\rho_tn^{1/3}\r\rceil}^{\l\lfloor\l(t-1\r)n^{1/3}\r\rfloor}\l(\sqrt{t-\frac{k+1}{n^{1/3}}}-\sqrt{t-\frac{k}{n^{1/3}}}\r)^2\r)\\
&+\sqrt{n}\l(\sum_{k=\l\lceil\l(t-1\r)n^{1/3}\r\rceil}^{\l\lfloor tn^{1/3}\r\rfloor}\frac{1}{n^{2/3}}+f\l(t,n\r)\r).
\end{align*}

% % \textcolor{blue}{\begin{align*}
% % &\leq tn^{1/6}\frac{\sqrt{1+4t^2}-1}{2t}+\sqrt{n}\sum_{k=\l\lceil\rho_tn^{1/3}\r\rceil}^{\l\lfloor\l(t-1\r)n^{1/3}\r\rfloor}\frac{1}{n^{2/3}}\l(\frac{1}{\sqrt{t-\frac{k+1}{n^{1/3}}}+\sqrt{t-\frac{k}{n^{1/3}}}} \r)^2+n^{1/6}\\
% % &+\sqrt{n}f\l(t,n\r)\\
% % &\leq tn^{1/6}\frac{\sqrt{1+4t^2}-1}{2t}+\frac{\l(t-1\r)n^{1/6}}{4\l(1-\frac{1}{n^{1/3}}\r)}+n^{1/6}+\sqrt{n}f\l(t,n\r).
% \end{align*}
% }

{Above $f(t,n)=\l(\frac{\sqrt{t}\l\lfloor\rho_tn^{1/3}\r\rfloor}{n^{1/3}}-\sqrt{t-\frac{\l\lceil\rho_tn^{1/3}\r\rceil}{n^{1/3}}}\r)^2+\l(\frac{\l\lceil\l(t-1\r)n^{1/3}\r\rceil}{n^{1/3}}-\sqrt{t-\frac{\l\lfloor\l(t-1\r)n^{1/3}\r\rfloor}{n^{1/3}}}\r)^2$. By looking at the slope around the point $\rho_t$,} it can be checked that $\sqrt{n}f(t,n)\leq\frac{t+1}{n^{2/3}}$.
So, we have that 
\begin{align}\label{eq: derivative square bound}
    \sqrt{n}\sum_{k=0}^{p}\l(\Delta v\l(k\r) \r)^2\leq Ctn^{1/6}
\end{align}for some fixed constant $C$ which is independent of $n$ and $t$. 
Next we look at 

\begin{align*}
    \frac{1}{\sqrt{n}}\sum_{k=1}^{p}kv^2\l(k\r)
&=\frac{1}{\sqrt{n}}\l(\sum_{k=1}^{{\l\lfloor\rho_tn^{1/3}\r\rfloor}}\frac{k^3t}{n^{2/3}}+\sum_{k=\l\lceil\rho_tn^{1/3}\r\rceil}^{\l\lfloor\l(t-1\r)n^{1/3}\r\rfloor}k\l(t-\frac{k}{n^{1/3}} \r)+\sum_{k=\l\lceil\l(t-1\r)n^{1/3}\r\rceil}^{\l\lfloor tn^{1/3}\r\rfloor}k\l(t-\frac{k}{n^{1/3}}\r)^2\r).
% &\textcolor{blue}{\leq \frac{\rho_t^2n^{2/3}\l(\rho_tn^{1/3}+1\r)^2t}{4n^{7/6}}+t\frac{\l(\l(t-1\r)n^{1/3}+1\r)^2-\l(\rho_tn^{1/3}+1\r)^2}{2\sqrt{n}}}\\
% &\textcolor{blue}{-\frac{\l\lfloor\l(t-1\r)n^{1/3}\r\rfloor^3-\l\lceil\rho_tn^{1/3}\r\rceil^3}{3n^{5/6}}+\frac{t}{2\sqrt{n}}\l(\l(tn^{1/3}+1\r)^2-\l(\l(t-1\r)n^{1/3}+1\r)^2\r)}\\
% &\textcolor{blue}{+\frac{1}{4n^{7/6}}\l(\l(tn^{1/3}+1\r)^4-\l(\l(t-1\r)n^{1/3}+1\r)^4\r)-\frac{2t}{3n^{5/6}}\l(\l\lfloor tn^{1/3}\r\rfloor^3-\l\lceil\l(t-1\r)n^{1/3}\r\rceil^3\r)}
\end{align*}

% $=\frac{1}{\sqrt{n}}\frac{t}{n^{2/3}}\frac{n^{2/3}\l(\frac{\sqrt{1+4t^2}-1}{2t} \r)^2\l(n^{2/3}\frac{\sqrt{1+4t^2}-1}{2t}+1 \r)^2}{4}
% \\+\frac{t}{\sqrt{n}}\frac{\l(\l(t-1\r)n^{1/3} \r)\l(\l(t-1\r)n^{1/3}+1 \r)}{2}-\frac{t}{2\sqrt{n}}\l(n^{1/3}\frac{\sqrt{1+4t^2}-1}{2t}-1\r)\l(n^{1/3}\frac{\sqrt{1+4t^2}-1}{2t}\r)\\
% -\frac{1}{n^{5/6}}\frac{\l(\l(t-1\r)n^{1/3} \r)\l(\l(t-1\r)n^{1/3}+1 \r)\l(2\l(t-1\r)n^{1/3}+1\r)+1}{6}+\frac{1}{6n^{5/6}}\l(n^{1/3}\frac{\sqrt{1+4t^2}-1}{2t}-1 \r)\l( n^{1/3}\frac{\sqrt{1+4t^2}}{2t}\r)\l(2n^{1/3}\frac{\sqrt{1+4t^2}-1}{2t}-1 \r)
% \\+\frac{t^2}{\sqrt{n}}\frac{\l(tn^{1/3}\r)\l(tn^{1/3}+1\r)}{2}-\frac{t^2}{\sqrt{n}}\frac{\l(t-1\r)n^{1/3}-1)\l(\l(t-1\r)n^{1/3}\r)}{2}-\frac{2t}{n^{5/6}}\frac{\l(tn^{1/3}\r)\l(tn^{1/3}+1\r)\l(2tn^{1/3}+1\r)}{6}\\+\frac{2t}{n^{5/6}}\frac{\l(\l(t-1\r)n^{1/3}-1\r)\l(t\l(n-1\r)^{1/3}\r)\l(2\l(t-1\r)n^{1/3}-1\r)}{6}+\frac{1}{n^{7/6}}\frac{\l(tn^{1/3}+1\r)^2\l(tn^{1/3}\r)^2}{4}-\frac{1}{n^{7/6}}\frac{\l(\l(t-1\r)n^{1/3}-1\r)^2\l(\l(t-1\r)n^{1/3}\r)^2}{4}.
% $\\

As $\rho_t\rightarrow 1$ as $t\rightarrow\infty$, by elementary calculations it can be checked that for any $\vep>0$ we can choose $n$ and $t$ sufficiently large such that 
\begin{align}\label{eq: Hermite left tail k vk^2 sum}
    \frac{1}{\sqrt{n}}\sum_{k=1}^{p}kv^2\l(k\r) \leq \frac{n^{1/6}t^3}{6}\l(1+\vep\r).
\end{align}

% Next observe that 
% \[
% \sqrt{n}v\l(1\r)=\frac{\sqrt{t}}{{n^{1/3}}},
% \]
% which goes to 0 as $n \rightarrow 0$ uniformly over $t$ as $t \ll n^{1/6}.$
Now, we have

\begin{align*}
\sum_{k=1}^{p}v^2\l(k\r)&=\sum_{k=1}^{{\l\lfloor\rho_tn^{1/3}\r\rfloor}}\frac{k^2t}{n^{2/3}}+\sum_{k=\l\lceil\rho_tn^{1/3}\r\rceil}^{\l\lfloor\l(t-1\r)n^{1/3}\r\rfloor}\l(t-\frac{k}{n^{1/3}} \r)+\sum_{k=\l\lceil\l(t-1\r)n^{1/3}\r\rceil}^{\l\lfloor tn^{1/3}\r\rfloor}\l(t-\frac{k}{n^{1/3}}\r)^2.
%     &\textcolor{blue}{\leq \frac{\l(\rho_tn^{1/3}+1\r)^3t}{3n^{2/3}}+t\l(\l(t-1\r)n^{1/3}\r)-\frac{\l\lfloor\l(t-1\r)n^{1/3}\r\rfloor^2-\l\lceil\rho_tn^{1/3}\r\rceil^2}{2n^{1/3}}}\\
% &\textcolor{blue}{+t^2n^{1/3}+\frac{1}{3n^{2/3}}\l(\l(tn^{1/3}+1\r)^3-\l(\l(t-1\r)n^{1/3}+1\r)^3\r)-t\frac{\l\lfloor tn^{1/3}\r\rfloor^2-\l\lceil\l(t-1\r)n^{1/3}\r\rceil^2}{n^{1/3}}}
\end{align*}
% $
% \frac{1}{2\sqrt{n}}\sum_{k=1}^{tn^{1/3}}v^2\l(k\r)\\
% =\frac{1}{2\sqrt{n}}\sum_{k=1}^{k=n^{1/3}\frac{\sqrt{1+4t^2}-1}{2t}}\frac{tk^2}{n^{2/3}}+\frac{1}{2 \sqrt{n}}\sum_{k=n^{1/3}\frac{\sqrt{1+4t^2}-1}{2t}}^{\l(t-1\r)n^{1/3}}\l(t-\frac{k}{n^{1/3}}\r)+\frac{1}{2\sqrt{n}}\sum_{\l(t-1\r)n^{1/3}}^{tn^{1/3}}\l(t-\frac{k}{n^{1/3}} \r)^2\\
% \leq \frac{1}{2 \sqrt{n}}\frac{t}{n^{2/3}}\frac{n^{1/3}\frac{\sqrt{1+4t^2}-1}{2t}\l(n^{1/3}\frac{\sqrt{1+4t^2}-1}{2t}+1\r)\l(2n^{1/3}\frac{\sqrt{1+4t^2}-1}{2t}+1\r)}{6}+\frac{t}{2\sqrt{n}}\l(t-1- \frac{\sqrt{1+4t^2}-1}{2t}\r)n^{1/3}\\-
% \frac{1}{4n^{5/6}}\l(\l(t-1\r)n^{1/3} \r)\l(\l(t-1\r)n^{1/3}+1 \r)+\frac{1}{4n^{5/6}}\l( \frac{\sqrt{1+4t^2}-1}{2t}n^{1/3}-1\r)\l(\frac{\sqrt{1+4t^2}-1}{2t}n^{1/3} \r)+\frac{1}{2}n^{-1/6}.$\\
One can check that for any $\vep>0$ there exists $n$ and $t$ sufficiently large such that 
\begin{align}\label{eq: Hermite left tail v^2 sum}
    \l|\frac{1}{2 \sqrt{n}}\sum_{k=1}^{p}v^2\l(k\r)-\frac{n^{-1/6}t^2}{4}\r| \leq {n^{-1/6}t^2\vep}.
\end{align}

Next we look at 
\begin{align*}
\sum_{k=1}^{p}v^4\l(k\r)&=\sum_{k=1}^{{\l\lfloor\rho_tn^{1/3}\r\rfloor}}\frac{k^4t^2}{n^{4/3}}+\sum_{k=\l\lceil\rho_tn^{1/3}\r\rceil}^{\l\lfloor\l(t-1\r)n^{1/3}\r\rfloor}\l(t-\frac{k}{n^{1/3}} \r)^2+\sum_{k=\l\lceil\l(t-1\r)n^{1/3}\r\rceil}^{\l\lfloor tn^{1/3}\r\rfloor}\l(t-\frac{k}{n^{1/3}}\r)^4.
% &\textcolor{blue}{\leq \frac{\l(\rho_tn^{1/3}+1\r)^5t^2}{5n^{4/3}}+t^2\l(\l(t-1\r)n^{1/3}\r)+\frac{\l(\l(t-1\r)n^{1/3}+1\r)^3}{3n^{2/3}}-t\frac{\l\lfloor\l(t-1\r)n^{1/3}\r\rfloor^2-\l\lceil\rho_tn^{1/3}\r\rceil^2}{n^{1/3}}}\\
% &\textcolor{blue}{+t^4n^{1/3}-2t^3\frac{\l\lfloor tn^{1/3}\r\rfloor^2-\l\lceil\l(t-1\r)n^{1/3}\r\rceil^2}{n^{1/3}}+\frac{2t^2}{n^{2/3}}\l(\l(tn^{1/3}+1\r)^3-\l(\l(t-1\r)n^{1/3}+1\r)^3\r)}\\
% &\textcolor{blue}{-t\frac{\l\lfloor tn^{1/3}\r\rfloor^4-\l\lceil\l(t-1\r)n^{1/3}\r\rceil^4}{n}+\frac{1}{5n^{4/3}}\l(\l(tn^{1/3}+1\r)^5-\l(\l(t-1\r)n^{1/3}+1\r)^5\r)}
\end{align*}

% $\sum_{k=1}^{tn^{1/3}}v^4\l(k\r)\\=\sum_{k=1}^{k=n^{1/3}\frac{\sqrt{1+4t^2}-1}{2t}}\frac{tk^4}{n^{4/3}}+\sum_{k=n^{1/3}\frac{\sqrt{1+4t^2}-1}{2t}}^{\l(t-1\r)n^{1/3}}\l(t-\frac{k}{n^{1/3}}\r)^2+\sum_{\l(t-1\r)n^{1/3}}^{tn^{1/3}}\l(t-\frac{k}{n^{1/3}} \r)^4\\\leq ctn^{1/3}+t^2\l( \l(t-1-\frac{\sqrt{1+4t^2}-1}{2t}\r)n^{1/3}\r)-\frac{t}{n^{1/3}}\l(\l(t-1\r)n^{1/3} \r)\l(\l(t-1\r)n^{1/3}+1 \r)+\\\frac{t}{n^{1/3}}\l(n^{1/3}\frac{\sqrt{1+4t^2}-1}{2t}-1\r)\l(n^{2/3}\frac{\sqrt{1+4t^2}-1}{2t} \r)+\frac{1}{6n^{1/3}}\l(\l(t-1\r)n^{1/3} \r)\l(\l(t-1\r)n^{1/3}+1 \r)\l(2\l(t-1\r)n^{1/3}+1 \r)\\-\frac{1}{6n^{1/3}}\l(n^{1/3}\frac{\sqrt{1+4t^2}-1}{2t}-1 \r)\l( n^{1/3}\frac{\sqrt{1+4t^2}-1}{2t}\r)\l(2n^{1/3}\frac{\sqrt{1+4t^2}-1}{2t} -1\r)+n^{1/3}$.\\ 

Hence, one can check that for any $\vep>0$ there exists $n$ and $t$ sufficiently large such that 

\begin{align*}
     \sum_{k=1}^{p}v^4\l(k\r) \leq \frac{n^{1/3}t^3}{3}\l(1+\vep\r).
\end{align*}

Finally we have the sum
\begin{align*}
    \sum_{k=0}^{p}\l(\Delta \l(v^2\l(k\r)\r)\r)^2 &=\sum_{k=1}^{{\l\lfloor\rho_tn^{1/3}\r\rfloor}}\frac{\l(2k+1\r)^2t^2}{n^{4/3}}+\sum_{k=\l\lceil\rho_tn^{1/3}\r\rceil}^{\l\lfloor\l(t-1\r)n^{1/3}\r\rfloor}\frac{1}{n^{2/3}}\\
    &+\sum_{k=\l\lceil\l(t-1\r)n^{1/3}\r\rceil}^{\l\lfloor tn^{1/3}\r\rfloor}\l(-\frac{2t}{n^{1/3}}+\frac{2k+1}{n^{2/3}} \r)^2+g(t,n).
    % &\textcolor{blue}{\leq \frac{4\l(\rho_tn^{1/3}+2\r)^3t^2}{3n^{4/3}}+\frac{\l\lfloor\l(t-1\r)n^{1/3}\r\rfloor-\l\lceil\rho_tn^{1/3}\r\rceil}{n^{2/3}}+\frac{1}{n}+g(t,n)}
\end{align*}

% $\sum_{k=0}^{tn^{1/3}}\Delta^2 \l(v^2\l(k\r)\r) \\=\frac{t^2}{n^{4/3}}\sum_{k=0}^{k=n^{1/3}\frac{\sqrt{1+4t^2}-1}{2t}}\l(2k+1 \r)+\sum_{k=n^{1/3}\frac{\sqrt{1+4t^2}-1}{2t}}^{\l(t-1\r)n^{1/3}}\frac{1}{n^{2/3}}+\sum_{\l(t-1\r)n^{1/3}}^{tn^{1/3}}\l(-\frac{2}{n^{1/3}}+\frac{2k+1}{n^{2/3}} \r)^2\\\leq \frac{t^2\l(n^{1/3}\frac{\sqrt{1+4t^2}-1}{2t} \r)\l(n^{1/3}\frac{\sqrt{1+4t^2}-1}{2t}+1\r)+t^2n^{1/3}\frac{\sqrt{1+4t^2}-1}{2t}}{n^{4/3}}+n^{1/3}\l( t-1-\frac{\sqrt{1+4t^2}-1}{2t}\r)+n^{1/3}\l(-\frac{2}{n^{1/3}}+\frac{2tn^{1/3}+1}{n^{2/3}} \r)^2.$
% \\
Above $g(t,n)=\l(\frac{t\l\lfloor\rho_tn^{1/3}\r\rfloor^2}{n^{2/3}}-\l({t-\frac{\l\lceil\rho_tn^{1/3}\r\rceil}{n^{1/3}}}\r)\r)^2+\l(\frac{\l\lceil\l(t-1\r)n^{1/3}\r\rceil^2}{n^{2/3}}-\l({t-\frac{\l\lfloor\l(t-1\r)n^{1/3}\r\rfloor}{n^{1/3}}}\r)\r)^2$. It can be checked that for large enough $t$ and $n$, we have 
\begin{align}\label{eq: sigma^2 lower bound}
    \sum_{k=0}^{p}\l(v^4\l(k\r)-\frac{1}{4}\l(\Delta\l(v^2\l(k\r)\r)\r)^2\r) \leq \frac{n^{1/3}t^3}{3}\l(1+\vep\r).
\end{align}

From the bounds \eqref{eq: derivative square bound}, \eqref{eq: Hermite left tail k vk^2 sum}, \eqref{eq: Hermite left tail v^2 sum} in the inequality \eqref{eq: Hermite left tail mean bound} we get for any $\vep>0$ there exists $n$ and $t$ sufficiently large such that $\l(2\sqrt{n}-tn^{-1/6}\r)\left<v,v\right>-\mu \leq -\frac 13 t^3n^{1/6}\l(1-\vep\r)$. Also by \eqref{eq: sigma^2 lower bound} and \eqref{eq: Hermite left tail variance bound}, we can chose $t$ and $n$ sufficiently large such that $\sigma^2 \leq \frac{4}{3 \beta}t^3n^{1/3}\l(1+\vep\r).$ Hence, we need an upper bound of the following event 
\[
Z \leq \frac{-\frac 13t^3n^{1/6}\l(1-\vep\r)}{\frac{2}{\sqrt{3\beta}}t^{3/2} n^{1/6}\l(1+\vep\r)}.
\]
Hence we have that for any $\beta>0$ and any small $\varepsilon,\delta>0$ there exists $n_{\varepsilon,\beta}$ and $t_{\varepsilon,\beta}$ such that for $n\geq n_{\varepsilon,\beta}$ and $t_{\varepsilon,\beta}\leq t\leq n^{\frac{1}{6}-\delta}$ we have
\begin{align}\label{eq: Hermite tail lower bound before KMT}
\mathbb{P}\l( \lambda_{\text{max}}\l(\wih{H}_{n,\beta,p}\r)\leq 2\sqrt{n}-tn^{-1/6}\r)
\leq \exp\l(- \l(\frac{\beta}{24}-\varepsilon\r)t^{3}\r).
\end{align}

Further, by application of Weyl's inequality we have 
\[
\l\{ \lambda_{\max}\l(H_{n,\beta}\r) \leq 2\sqrt{n}-tn^{-1/6} \r\} \subset \l\{ \lambda_{\max}\l(\wih{H}_{n,\beta,p}\r)\leq 2\sqrt{n}-tn^{-1/6}+\lambda_{\max}\l(H_{n,\beta}-\wih{H}_{n,\beta,p}\r) \r\}.
\]

By Proposition \ref{eq: estimate of off diagonals of the difference} we have 
\begin{align*}
&\P \l(\lambda_{\max}\l(H_{n,\beta}\r) \leq 2\sqrt{n}-tn^{-1/6} \r) \leq \P \l( \lambda_{\max}\l(\wih{H}_{n,\beta,p}\r)\leq 2\sqrt{n}-tn^{-1/6}+ c_{\beta} \frac{\log n+2x}{ \sqrt{n}}\r)\\
&+C n\l(3e^{-x}+2e^{-\frac{\beta\l(n-p\r)}{2}}\r).
\end{align*}
Taking $x=tn^{\frac{1}{3}-\frac{\delta}{3}}$ and using \eqref{eq: Hermite tail lower bound before KMT}, we have that for any $\beta>0$ and any small $\varepsilon,\delta>0$ there exists $n_{\varepsilon,\beta}$ and $t_{\varepsilon,\beta}$ such that for $n\geq n_{\varepsilon,\beta}$ and $t_{\varepsilon,\beta}\leq t\leq n^{\frac{1}{6}-\delta}$ we have
\begin{align}\label{eq: Hermite tail lower bound final form}
\mathbb{P}\l( \lambda_{\text{max}}\l(H_{n,\beta}\r)\leq 2\sqrt{n}-tn^{-1/6}\r)
\leq \exp\l(- \l(\frac{\beta}{24}-\varepsilon\r)t^{3}\r).
\end{align}

\subsection{Laguerre Ensemble}\label{sec: Tridiagonal for Laguerre}

Fix $\beta>0$ and fix small $\varepsilon,\delta>0$.
Let $m=m\l(n\r)\geq n$ be such that for all $n\geq n_{\varepsilon}$, we have $\frac{m}{n}\leq M$ for some $M<\infty$. We denote $\gamma:=m/n$ and $a=\l(1+\sqrt{\gamma}\r)^2$ and $b=\gamma^{-1/6}\l(1+\sqrt{\gamma}\r)^{4/3}.$ We consider  matrix $\wih{L}_{n,m,\beta,p}$ for $ p:= \l\lfloor \frac{tn^{1/3}}{d}\r\rfloor$, where $d=\frac{\sqrt{b}}{\gamma^{1/4}}$. We obtain upper bound for probability of the event that $\lambda_{\text{max}}\l(\wih{L}_{n,m,\beta,p}\r)\leq an\l(1-\frac{tbn^{-2/3}}{a}\r)$. Here $t_{\varepsilon}\leq t\leq n^{\frac{1}{6}-\delta}$ (one of the reason for this range of $t$ is so that $p\leq \sqrt{n}$). The value $t_{\varepsilon}$ will be made precise later. As before for $\varepsilon>0$, we construct a non-negative vector $v$ (this vector depends on $n,t,C,\delta$ and $\varepsilon$) such that $v\l(i\r)=0$ for all $i> p$ and $v\l(0\r)=0$ to show that for all $n\geq n_{\varepsilon,\beta}$ and  $t_{\varepsilon,\beta}\leq t\leq n^{\frac{1}{6}-\delta}$, we have
\begin{align}\label{eq: Laguerre before KMT left tail}
    \langle \wih{L}_{n,m,\beta,p}v,v\rangle\leq an\l(1-\frac{tbn^{-2/3}}{a}\r)\langle v,v\rangle
\end{align}
with probability at most $\exp\l(-\l(\frac{\beta}{24}-\varepsilon\r)t^3\r)$. Note that $\wih{L}_{n,m,\beta,p}$ can be written as a sum of two matrices $L'$ and $L''$, where $L''$ is a tridiagonal symmetric matrix which contains entries which are linear combinations of $\zeta_k^2,\l(\zeta'_k\r)^2, \zeta_{k+1}\zeta'_k$ in the top $p\times p$ submatrix and zeros elsewhere. $L'$ contains random variables which are Gaussian in the top $p\times p$ submatrix. We first show that for the range $t_{\varepsilon}\leq t\leq n^{\frac{1}{6}-\delta}$,
\begin{align*}
    \langle L'v,v\rangle\leq an\l(1-\frac{tbn^{-2/3}}{a}\r)\langle v,v\rangle
\end{align*}
with probability at most $\exp\l(-\l(\frac{\beta}{24}-\varepsilon\r)t^{3}\r)$.

We fix $t$ and consider the following function. Let $g_t\l(x\r)=x\sqrt{t}\wedge \sqrt{\l(t-x\r)^+}\wedge \l(t-x\r)^+$. Let $\rho_t$ be the positive solution of the quadratic equation in $x$, which is $x^2t+x-t=0$.  We define $v\l(k\r)=g_t\l(\frac{kd}{n^{1/3}}\r)$ for all $k \leq p$ and $0$ otherwise.

Now $\langle L'v,v\rangle$ is a Gaussian random variable with mean $\mu$ and variance $\sigma^2$ given by ($v\l(k\r)=0$ for $k\geq p$ and using the fact $\sqrt{n-k}\geq \sqrt{n}-\frac{k+1}{2\sqrt{n}}$ for $k\leq 2\sqrt{n}-1$),
\begin{align}\label{eq: Laguerre mean lower bound}
    \mu&= \sum\limits_{k=1}^{p} \l(\l(m+1-k\r)+\l(n-k\r)\r)v^2\l(k\r) + 2\sum\limits_{k=1}^{p-1} \l(\sqrt{\l(m-k\r)\l(n-k\r)}\r)v\l(k\r)v\l(k+1\r)\nonumber\\
    &\geq \sum\limits_{k=1}^{p} \l(m+n-2k\r)v^2\l(k\r) +\sum\limits_{k=1}^{p-1} \l(\sqrt{m}-\frac{k{+1}}{2\sqrt{m}}\r)\l(\sqrt{n}-\frac{k+{1}}{2\sqrt{n}}\r)\l(v^2\l(k\r)+v^2\l(k+1\r)-\l(\Delta\l(v\l(k\r)\r) \r)^2\r)\nonumber\\
    &\geq \l(\sqrt{m}+\sqrt{n}\r)^2\langle v,v\rangle-\sum\limits_{k=1}^{p} \l(2k+\frac{\l({m}+{n}\r)\l(k+1\r)}{\sqrt{mn}}\r)v^2\l(k\r)-\sum\limits_{k=0}^{p}\l(\sqrt{mn}+\frac{\l(k+1\r)^2}{4\sqrt{mn}}\r)\l(\Delta\l(v\l(k\r)\r)\r)^2.
\end{align}

\begin{align*}
    \sigma^2&=  \frac{2}{\beta}\sum\limits_{k=0}^{p-1}\l({\sqrt{m-k}\ v^2\l(k+1\r)}+\sqrt{n-k}\ v\l(k\r)v\l(k+1\r)\r)^2 \\
    &+\frac{2}{\beta}\sum\limits_{k=1}^{p} \l({\sqrt{n-k}\ v^2\l(k\r)}+\sqrt{m-k}\ v\l(k\r)v\l(k+1\r)\r)^2. 
\end{align*}

We observe the following.
\begin{align*}
\sum_{k=0}^{p}\l(\Delta v\l(k\r)\r)^2 &=\sum_{k=0}^{\l\lfloor\rho_tn^{1/3}/d\r\rfloor} \frac{td^2}{n^{2/3}}+\sum_{k=\l\lceil\rho_tn^{1/3}/d\r\rceil}^{\l\lfloor\l(t-1\r)n^{1/3}/d\r\rfloor}\l(\sqrt{t-\frac{\l(k+1\r)d}{n^{1/3}}}-\sqrt{t-\frac{kd}{n^{1/3}}}\r)^2\\
&+\sum_{k=\l\lceil\l(t-1\r)n^{1/3}/d\r\rceil}^{\l\lfloor tn^{1/3}/d\r\rfloor}\frac{d^2}{n^{2/3}}+f\l(t,n\r).
% &\textcolor{blue}{\leq td\frac{\sqrt{1+4t^2}-1}{2tn^{1/3}}+\sum_{k=\l\lceil\rho_tn^{1/3}/d\r\rceil}^{\l\lfloor\l(t-1\r)n^{1/3}/d\r\rfloor}\frac{d^2}{n^{2/3}}\l(\frac{1}{\sqrt{t-\frac{\l(k+1\r)d}{n^{1/3}}}+\sqrt{t-\frac{kd}{n^{1/3}}}} \r)^2+\frac{d}{n^{1/3}}+f\l(t,n\r)}\\
% &\textcolor{blue}{\leq td\frac{\sqrt{1+4t^2}-1}{2tn^{1/3}}+\frac{\l(t-1\r)d}{4n^{1/3}\l(1-\frac{d}{n^{1/3}}\r)}+\frac{d}{n^{1/3}}+f\l(t,n\r)}.
\end{align*}

Above $f(t,n)=\l(\frac{\sqrt{t}\l\lfloor\rho_tn^{1/3}/d\r\rfloor}{n^{1/3}}-\sqrt{t-\frac{\l\lceil\rho_tn^{1/3}/d\r\rceil d}{n^{1/3}}}\r)^2+\l(\frac{\l\lceil\l(t-1\r)n^{1/3}/d\r\rceil d}{n^{1/3}}-\sqrt{t-\frac{\l\lfloor\l(t-1\r)n^{1/3}/d\r\rfloor d}{n^{1/3}}}\r)^2$. By looking at the slope around the point $\rho_t$, it can be checked that $f(t,n)\leq\frac{\l(t+1\r)d}{n^{2/3}}$.
So, we have that 
\begin{align*}
    \sum\limits_{k=0}^{p}\l(\Delta\l(v\l(k\r)\r)\r)^2 \leq ctn^{-1/3}
\end{align*}for some fixed constant $c$ which is independent of $n$ and $t$. 
Hence we have 
\begin{align}\label{eq: Laguerre derivative square bound}
\sum\limits_{k=0}^{p}\l(\sqrt{mn}+\frac{\l(k+1\r)^2}{4\sqrt{mn}}\r)\l(\Delta\l(v\l(k\r)\r)\r)^2 \leq ctn^{2/3}.
\end{align}
Next we look at 
\begin{align*}
    \sum_{k=1}^{p}kv^2\l(k\r)
&=\sum_{k=1}^{{\l\lfloor\rho_tn^{1/3}/d\r\rfloor}}\frac{k^3td^2}{n^{2/3}}+\sum_{k=\l\lceil\rho_tn^{1/3}/d\r\rceil}^{\l\lfloor\l(t-1\r)n^{1/3}/d\r\rfloor}k\l(t-\frac{kd}{n^{1/3}} \r)+\sum_{k=\l\lceil\l(t-1\r)n^{1/3}/d\r\rceil}^{\l\lfloor tn^{1/3}/d\r\rfloor}k\l(t-\frac{kd}{n^{1/3}}\r)^2.
% &\textcolor{blue}{\leq \frac{\rho_t^2n^{2/3}\l(\frac{\rho_tn^{1/3}}{d}+1\r)^2t}{4n^{2/3}}+t\frac{\l(\frac{\l(t-1\r)n^{1/3}}{d}+1\r)^2-\l(\frac{\rho_tn^{1/3}}{d}+1\r)^2}{2}}\\
% &\textcolor{blue}{-\frac{d\l\lfloor\l(t-1\r)n^{1/3}/d\r\rfloor^3-d\l\lceil\rho_tn^{1/3}/d\r\rceil^3}{3n^{1/3}}+\frac{t^2}{2}\l(\l(\frac{tn^{1/3}}{d}+1\r)^2-\l(\frac{\l(t-1\r)n^{1/3}}{d}+1\r)^2\r)}\\
% &\textcolor{blue}{+\frac{d^2}{4n^{2/3}}\l(\l(\frac{tn^{1/3}}{d}+1\r)^4-\l(\frac{\l(t-1\r)n^{1/3}}{d}+1\r)^4\r)}\\
% &\textcolor{blue}{-\frac{2td}{3n^{1/3}}\l(\l\lfloor tn^{1/3}/d\r\rfloor^3-\l\lceil\l(t-1\r)n^{1/3}/d\r\rceil^3\r).}
\end{align*}
As $\rho_t\rightarrow 1$ as $t\rightarrow\infty$, by elementary calculations it can be checked that for any $\vep>0$ we can choose $n$ and $t$ sufficiently large such that 
\begin{align}\label{eq: Laguerre left tail k vk^2 sum}
    \sum\limits_{k=1}^{p} kv^2\l(k\r)\leq 
    \frac{t^3n^{2/3}\l(1+\varepsilon\r)}{6d^2}.
\end{align}
Now, we have
\begin{align*}
\sum_{k=1}^{p}v^2\l(k\r)&=\sum_{k=1}^{{\l\lfloor\rho_tn^{1/3}/d\r\rfloor}}\frac{k^2td^2}{n^{2/3}}+\sum_{k=\l\lceil\rho_tn^{1/3}/d\r\rceil}^{\l\lfloor\l(t-1\r)n^{1/3}/d\r\rfloor}\l(t-\frac{kd}{n^{1/3}} \r)+\sum_{k=\l\lceil\l(t-1\r)n^{1/3}/d\r\rceil}^{\l\lfloor tn^{1/3}/d\r\rfloor}\l(t-\frac{kd}{n^{1/3}}\r)^2.
%     &\textcolor{blue}{\leq \frac{\l(\frac{\rho_tn^{1/3}}{d}+1\r)^3td^2}{3n^{2/3}}+t\l(\l(t-1\r)n^{1/3}/d\r)-\frac{d\l\lfloor\l(t-1\r)n^{1/3}/d\r\rfloor^2-d\l\lceil\rho_tn^{1/3}/d\r\rceil^2}{2n^{1/3}}}\\
% &\textcolor{blue}{+\frac{t^2n^{1/3}}{d}+\frac{d^2}{3n^{2/3}}\l(\l(\frac{tn^{1/3}}{d}+1\r)^3-\l(\frac{\l(t-1\r)n^{1/3}}{d}+1\r)^3\r)}\\
% &\textcolor{blue}{-td\frac{\l\lfloor tn^{1/3}/d\r\rfloor^2-\l\lceil\l(t-1\r)n^{1/3}/d\r\rceil^2}{n^{1/3}}.}
\end{align*}

 We get that for any $\vep>0$ there exists $n$ and $t$ sufficiently large such that 
\begin{align}\label{eq: Laguerre left tail v^2 sum}
    \l|\sum_{k=1}^{p}v^2\l(k\r)-\frac{t^2n^{1/3}}{2d}\r| \leq \frac{t^2n^{1/3}}{d}\vep.
\end{align}

It can also be checked that for any $\vep>0$ we can choose $n$ and $t$ sufficiently large such that 
\begin{align}\label{eq: Laguerre left tail mean term}
    \sum\limits_{k=1}^{p} \l(2k+\frac{\l({m}+{n}\r)\l(k+1\r)}{\sqrt{mn}}\r)v^2\l(k\r)&\leq 
    \frac{\l(\sqrt{m}+\sqrt{n}\r)^2t^3n^{2/3}\l(1+\varepsilon\r)}{\sqrt{mn}6d^2}+\frac{(m+n)t^2n^{1/3}\l(1+\varepsilon\r)}{\sqrt{mn}2d}\nonumber\\
    &\leq \frac{\l(\sqrt{m}+\sqrt{n}\r)^2t^3n^{2/3}\l(1+2\varepsilon\r)}{\sqrt{mn}6d^2}.
\end{align}

We also have that 
\begin{align*}
    \sigma^2&\leq \frac{2}{\beta}\l(m+n\r)\sum\limits_{k=1}^{p}\l(v^4\l(k\r)+v^2\l(k\r)v^2\l(k+1\r)\r)+\frac{4}{\beta}\sqrt{mn}\sum\limits_{k=1}^{p}\l(v^3\l(k\r)v\l(k+1\r)+v^3\l(k+1\r)v\l(k\r)\r)\\
    &\leq \l(\frac{2}{\beta}\l(m+n\r)\r)\l(\sum\limits_{k=1}^{p}2v^4\l(k\r)-\frac{1}{2}\sum\limits_{k=0}^{p}\l(\Delta\l(v^2\l(k\r)\r)\r)^2\r)\nonumber\\
    &\quad\quad +\frac{2}{\beta}\sqrt{mn}\l(\sum\limits_{k=1}^{p}4v^4\l(k\r)-\sum\limits_{k=0}^{p}\l(\Delta\l(v^2\l(k\r)\r)\r)^2-\sum\limits_{k=0}^{p}\l(\Delta\l(v\l(k\r)\r)\r)^2\l(v^2\l(k\r)+v^2\l(k+1\r)\r)\r).
\end{align*}

Next we look at 
\begin{align*}
\sum_{k=1}^{p}v^4\l(k\r)&=\sum_{k=1}^{{\l\lfloor\rho_tn^{1/3}/d\r\rfloor}}\frac{k^4t^2d^4}{n^{4/3}}+\sum_{k=\l\lceil\rho_tn^{1/3}/d\r\rceil}^{\l\lfloor\l(t-1\r)n^{1/3}/d\r\rfloor}\l(t-\frac{kd}{n^{1/3}} \r)^2+\sum_{k=\l\lceil\l(t-1\r)n^{1/3}/d\r\rceil}^{\l\lfloor tn^{1/3}/d\r\rfloor}\l(t-\frac{kd}{n^{1/3}}\r)^4.
% &\textcolor{blue}{\leq \frac{\l(\frac{\rho_tn^{1/3}}{d}+1\r)^5t^2d^4}{5n^{4/3}}+t^2\l(\l(t-1\r)n^{1/3}/d\r)+\frac{d^2\l(\frac{\l(t-1\r)n^{1/3}}{d}+1\r)^3}{3n^{2/3}}}\\
% &\textcolor{blue}{-td\frac{\l\lfloor\l(t-1\r)n^{1/3}/d\r\rfloor^2-\l\lceil\rho_tn^{1/3}/d\r\rceil^2}{n^{1/3}}+\frac{t^4n^{1/3}}{d}-2t^3d\frac{\l\lfloor tn^{1/3}/d\r\rfloor^2-\l\lceil\l(t-1\r)n^{1/3}/d\r\rceil^2}{n^{1/3}}}\\
% &\textcolor{blue}{+\frac{2t^2d^2}{n^{2/3}}\l(\l(\frac{tn^{1/3}}{d}+1\r)^3-\l(\frac{\l(t-1\r)n^{1/3}}{d}+1\r)^3\r)-td^3\frac{\l\lfloor tn^{1/3}/d\r\rfloor^4-\l\lceil\l(t-1\r)n^{1/3}/d\r\rceil^4}{n}}\\
% &\textcolor{blue}{+\frac{d^4}{5n^{4/3}}\l(\l(\frac{tn^{1/3}}{d}+1\r)^5-\l(\frac{\l(t-1\r)n^{1/3}}{d}+1\r)^5\r)}
\end{align*}

Hence we have that for any $\vep>0$ there exists $n$ and $t$ sufficiently large such that 
\begin{align*}
     \sum_{k=1}^{p}v^4\l(k\r) \leq \frac{n^{1/3}t^3}{3d}\l(1+\vep\r).
\end{align*}

Let $g(t,n)=\l(\frac{t\l\lfloor\rho_tn^{1/3}/d\r\rfloor^2d^2}{n^{2/3}}-\l({t-\frac{\l\lceil\rho_tn^{1/3}/d\r\rceil d}{n^{1/3}}}\r)\r)^2+\l(\frac{\l\lceil\l(t-1\r)n^{1/3}/d\r\rceil^2d^2}{n^{2/3}}-\l({t-\frac{\l\lfloor\l(t-1\r)n^{1/3}/d\r\rfloor d}{n^{1/3}}}\r)\r)^2$.
We observe the following.
\begin{align*}
\sum_{k=0}^{p}\l(\Delta v^2\l(k\r)\r)^2 &=\sum_{k=0}^{\l\lfloor\rho_tn^{1/3}/d\r\rfloor} \frac{td^2\l(2k+1\r)}{n^{2/3}}+\sum_{k=\l\lceil\rho_tn^{1/3}/d\r\rceil}^{\l\lfloor\l(t-1\r)n^{1/3}/d\r\rfloor}\frac{d^2}{n^{2/3}}
+\sum_{k=\l\lceil\l(t-1\r)n^{1/3}/d\r\rceil}^{\l\lfloor tn^{1/3}/d\r\rfloor}\frac{d^2\l(2k+1\r)}{n^{2/3}}+g\l(t,n\r).
\end{align*}
By looking at the slope of the function $g_t(x)$ and elementary calculations, it can be checked that for any $\vep>0$ there exists $n$ and $t$ sufficiently large such that 
\begin{align*}
     \sum\limits_{k=1}^{p}v^4\l(k\r)-\sum\limits_{k=0}^{p}\l(\Delta\l(v^2\l(k\r)\r)\r)^2-\sum\limits_{k=0}^{p}\l(\Delta\l(v\l(k\r)\r)\r)^2\l(v^2\l(k\r)+v^2\l(k+1\r)\r) \leq \frac{n^{1/3}t^3}{3d}\l(1+\vep\r).
\end{align*}
We can chose $t$ and $n$ sufficiently large such that $\sigma^2 \leq \frac{4\l(\sqrt{m}+\sqrt{n}\r)^2}{3 d \beta}t^3n^{1/3}\l(1+\vep\r)$. Using the bounds \eqref{eq: Laguerre derivative square bound}, \eqref{eq: Laguerre left tail k vk^2 sum}, \eqref{eq: Laguerre left tail v^2 sum}, \eqref{eq: Laguerre left tail mean term} in \eqref{eq: Laguerre mean lower bound}, we get for any $\vep>0$ there exists $n$ and $t$ sufficiently large such that $an\l(1-\frac{tbn^{-2/3}}{a}\r)\langle v,v\rangle-\mu\leq -\frac{bt^3n^{2/3}}{3d}\l(1-\varepsilon\r).$ 

Hence, we need an upper bound of the following event 
\[
Z \leq \frac{-\frac {b}{3d}t^3n^{2/3}\l(1-\vep\r)}{\frac{2\l(\sqrt{m}+\sqrt{n}\r)}{\sqrt{3d\beta}}t^{3/2} n^{1/6}\l(1+\vep\r)}.
\]

Applying upper bound for Gaussian tail we get that for any $\beta>0$ and small $\varepsilon,\delta>0$ there exist $n_{\varepsilon,\beta},t_{\varepsilon,\beta}$ such that for the range $t_{\varepsilon,\beta}\leq t\leq n^{\frac{1}{6}-\delta}$ and $n\geq n_{\varepsilon,\beta}$ with $m$ such that $\frac{m}{n}\leq C$ for some $C<\infty$, we have
\[\langle L'v,v\rangle\leq an\l(1-\frac{tbn^{-2/3}}{a}\r)\langle v,v\rangle\]
with probability at most $\exp\l(-\l(\frac{\beta}{24}-\varepsilon\r)t^{3}\r)$.

 Using Gershgorin circle theorem (Theorem $6.1.1$ of \cite{HJ}) and using the tail behaviour of normal random variable we get that the eigenvalues of the matrix $L''$ (introduced after \eqref{eq: Laguerre before KMT left tail}) are located in an interval $\l(-2tn^{\frac13-\frac{\delta}{3}},2tn^{\frac13-\frac{\delta}{3}}\r)$ with probability at least $1-6p\exp\l(-tn^{\frac13-\frac{\delta}{3}}/2\r)$. By Weyl's inequality we get that
 \begin{align*}
\lambda_{\text{max}}\l(\wih{L}_{n,m,\beta,p}\r)\leq an\l(1-\frac{tbn^{-2/3}}{a}\r)+2tn^{\frac13-\frac{\delta}{3}}
 \end{align*}
with probability at most $\exp\l(-\l(\frac{\beta}{24}+\varepsilon\r)t^{3}\r)-6p\exp\l(-tn^{\frac13-\frac{\delta}{3}}/2\r)$.

Further, note that by Weyl's inequality we have
\begin{align*}
&\l\{\lambda_{\max} \l(\wih{L}_{n,m,\beta,p} \r) \leq an-{tbn^{1/3}}+2tn^{\frac13-\frac{\delta}{3}}\r\}\\
&\subset \l\{ \lambda_{\max} \l(L_{n,m,\beta}\r) \leq an-{tbn^{1/3}}+2tn^{\frac13-\frac{\delta}{3}}+{\lambda_{\max} \l(\wih{L}_{n,m,\beta,p}-L_{n,m,\beta} \r)}\r\}.
\end{align*}
We take $x=tn^{\frac{1}{3}-\frac{\delta}{3}}$ and apply Proposition \ref{eq: estimate of off diagonals of the difference} to get,
\begin{align*}
&\mathbb{P}\l( \lambda_{\text{max}}\l(L_{n,m,\beta}\r)\leq an-t b n^{1/3}+\varepsilon t b n^{1/3}+2tn^{\frac13-\frac{\delta}{3}}\r)\\
&\leq \exp\l(- \l(\frac{\beta}{24}-\varepsilon\r)t^{3}\r)+6p\exp\l(-tn^{\frac13-\frac{\delta}{3}}/2\r)+8pe^{-tn^{\frac{1}{3}-\frac{\delta}{3}}}+4pe^{-\frac{\beta\l(n-p\r)}{2}}.
\end{align*}

 Hence we have that for any $\beta>0$ and any small $\varepsilon,\delta>0$ there exists $n_{\varepsilon,\beta}$ and $t_{\varepsilon,\beta}$ such that for $n\geq n_{\varepsilon,\beta}$ and $t_{\varepsilon,\beta}\leq t\leq n^{\frac{1}{6}-\delta}$ with $m$ such that $\frac{m}{n}\leq M$ for some $M<\infty$, we have
\begin{align}\label{eq: Laguerre left tail upper bound final form}
\mathbb{P}\l( \lambda_{\text{max}}\l(L_{n,m,\beta}\r)\leq an-t b n^{1/3}\r)
\leq \exp\l(- \l(\frac{\beta}{24}-\varepsilon\r)t^{3}\r).
\end{align}

\section{Lower bounds for left tails} \label{s: ltlb}

We make use of the following proposition. Let $A_n$ be an $n\times n$ tridiagonal matrix as shown below with $b_i, c_i > 0$, which makes them symmetrizable (by conjugating with appropriate diagonal matrices). 

% $\begin{bmatrix}
%     a_{1,1} & a_{1,2} & 0 & 0 & \cdots & 0 \\
%     a_{2,1} & a_{2,2} & a_{2,3} & 0 & \cdots & 0 \\
%     0 & a_{3,2} & a_{3,3} & a_{3,4} & \cdots & 0 \\
%     \vdots & \vdots & \vdots & \vdots & \ddots & \vdots \\
%     0 & \cdots & 0 & a_{n-1,n-2} & a_{n-1,n-1} & a_{n-1,n} \\
%     0 & \cdots & 0 & 0 & a_{n,n-1} & a_{n,n} \\
% \end{bmatrix}
% $
\begin{align*}
    A_n=
\begin{bmatrix}
    a_{1} & c_1 & 0 & 0 & \cdots & 0 & 0 & \cdots & 0 \\
    b_1 & a_{2} & c_2 & 0 & \cdots & 0 & 0 & \cdots & 0\\
    0 & b_2 & a_{3} & c_3 & \cdots & 0 & 0 & \cdots & 0\\
    \vdots & \vdots & \vdots & \vdots & \vdots & \vdots & \vdots& \vdots& \vdots\\
    0 & \cdots  & b_{k-2} & a_{k-1} & c_{k-1}& 0 & 0 & \cdots & 0 \\
    0 & \cdots & 0 & b_{k-1} & a_{k} & c_{k} & 0 & \cdots & 0\\
    0 & \cdots & 0 & 0 & b_{k} & a_{k+1} & c_{k+1} & \cdots & 0\\
    \vdots & \vdots & \vdots & \vdots & \vdots & \vdots & \vdots& \vdots& \vdots\\
    0 & \cdots & 0 & 0 & \cdots &  \cdots &  \cdots &  \cdots &  a_{n}\\
\end{bmatrix}.
\end{align*}

\begin{proposition}\label{prop: Tridiag resultwith1}

 We denote the top and bottom $k\times k$ submatrices of $A_n$ by $A_{\text{top}_k}$ and $A_{\text{low}_k}$ respectively. Fix $b_0=u_0=0$ and $u_1>0$ with $c_n=1$. Define the following recursion with $\lambda\in\mathbb{R}$ for $1\leq i\leq n$,
 \begin{align*}
      c_iu_{i+1}=\l(\lambda-a_{i}\r)u_{i}-b_{i-1}u_{i-1}.
  \end{align*}
\end{proposition}
 % \item If $\lambda_{\text{max}}\l(A_n\r)<\lambda$ then $u_i>0$ for all $i\in[n]$.
    % \item If $u_i>0$ for all $i\in[n]$ then $\lambda_{\text{max}}\l(A_{\text{top}_{n-1}}\r)<\lambda$.
    % Suppose $u_i>0$ for $1\leq i\leq k$ and $u_k>u_{k-1}$. Let $\overline{{A}}_{\text{low}_{n-k+1}}$ be $A_{\text{low}_{n-k+1}}$ with $a_{k}$ entry replaced by $a_{k}+b_{k-1}^2$. If $\lambda_{\text{max}}\l(\overline{A}_{n-k+1}\r)<\lambda$ then $\lambda_{\text{max}}\l(A_{n}\r)<\lambda$. 
Assume that for some $1<k<n$,
\begin{enumerate}
\item $u_i>0$ for $1\leq i\leq k$ and $u_k>u_{k-1}$. 
\item $\lambda_{\text{max}}\l(\overline{A}_{n-k+1}\r)<\lambda$ where  $\overline{{A}}_{n-k+1}$ is the matrix got after  replacing $a_{k}$  by $a_{k}+b_{k-1}$ in the matrix  $A_{\text{low}_{n-k+1}}$.
\end{enumerate}
Then $\lambda_{\max}(A_n)<\lambda$.
\begin{proof}
We use the well known fact of tridiagonal matrices that for the recursion defined above $u_i>0$ for all $1\leq i\leq n+1$ if and only if $\lambda_{\text{max}}\l(A_{n}\r)<\lambda$. Hence it is enough to show that $u_i>0$ for $k+1\leq i\leq n+1$.

    Let $(\overline{u}_{k-1},\ldots ,\overline{u}_{n+1})$ satisfy the same recursion with the same $\lambda$, but for the matrix $\overline{A}_{n-k+1}$. For convenience, we index it starting from $k-1$ instead of $0$. In particular, $\overline{u}_{k-1}=0$ and $\overline{u}_k=1$ and $c_{k}\overline{u}_{k+1}=\lambda-\l(a_{k}+b_{k-1}\r)$.  By the second assumption in the proposition, we have   $\overline{u}_j>0$ for $k\le j\le n+1$. Now define $U_{i}=u_i-u_k\overline{u}_i$ for $k\le i\le n+1$. Then $U_{k}=0$ and 
\begin{align*}
c_{k}U_{k+1} &= ((\lambda-a_{k})u_{k}-b_{k-1}u_{k-1})-u_{k}(\lambda -a_k-b_{k-1}) \\
&=b_{k-1}(u_{k}-u_{k-1}),
\end{align*}
which is positive by the first assumption. 

For $i\ge k+2$ , the recursion equations (which are linear) are the same for $u_i$ and $\overline{u}_i$, and therefore also the same for $U_i$. As $U_{k}=0$ and $U_{k+1}>0$, it follows that {for $k+1\leq i\leq n+1$ we have} $U_{i}=U_{k+1}\Tilde{u}_i$, where $\Tilde{u}=(\Tilde{u}_{k},\ldots ,\Tilde{u}_{n+1})$ is the solution of the recursion for the matrix $A_{\text{low}_{n-k}}$ (with $\Tilde{u}_{k}=0$ and $\Tilde{u}_{k+1}=1$). As $A_{\text{low}_{n-k}}$ is a principal submatrix of $\overline{A}_{n-k}$, its largest eigenvalue is less than $\lambda$, and hence $U_{i}>0$ for $k+1\le i\le n+1$. Consequently $u_i>u_k\overline{u}_i>0$. This shows that $u_i>0$ for all $i\le n+1$.\end{proof}

 % \textbf{Cameron Martin Girsanov Formula:} Let $B=\l(B_t\r)_{t\leq 1}$ be standard Brownian motion and let $W=B+f$ where $f\in C[0,1]$ and $f\l(0\r)=0$. Let $\mu $ and $\nu$ be the distributions of $B$ and $W$ respectively. Then if $f$ is absolutely continuous and $f'\in L^2[0,1]$, then $\mu$ and $\nu$ are absolutely continuous and 
 % \begin{align*}
 %     \frac{d\nu}{d\mu}=\exp\l\{ \int_0^1 f' dB\  -\ \frac{1}{2}\int_0^1\l(f'\r)^2 dx\r\}.
 % \end{align*}

 % Also for a non-random differentiable function $f$, we have that $\int_0^tf dB=f\l(t\r)B_t-\int_0^tB_sf'\l(s\r)ds$. Let $f$ be a twice differentiable function, then for any $a,b>0$
 % \begin{align}\label{eq: CMG formula}
 %     &\P\l(-a\leq W_s\leq b \mbox{ for } 0\leq s\leq t\r)=\exp\l( - \frac{1}{2}\int_0^t\l(f'\r)^2 dx\r)\\
 %     &\times \E_{\mu}\l[\exp\l(f'\l(t\r)B_t-\int_0^tB_sf''\l(s\r)ds\r)\mathbbm{1}\l(-a\leq B_s\leq b \mbox{ for } 0\leq s\leq t\r)\r].\nonumber
 % \end{align}

 % We note here that for fixed $a,b>0$, 
 % \begin{align}\label{eq: BM tail bound}
 %     \P\l(-a\leq B_s\leq b \mbox{ for } 0\leq s\leq t\r)=\exp\l(-O\l(t\r)\r)
 % \end{align}

\subsection{Hermite Ensemble}

% Consider the matrix $\wih{H}_{n,\beta,p}$ where $\chi_{\beta\l(n-i\r)}$ of $H_{n,\beta}$ are replaced by $\sqrt{\beta\l(n-i\r)+\sqrt{2\beta\l(n-i\r)}\zeta_i}$ for $i\leq \l\lfloor tn^{1/3}\r\rfloor$. Here $\zeta_i$ are i.i.d.\ $N\l(0,1\r)$ random variables. Note that for any small $\delta>0$, the probability of the event $|\zeta_i|\geq n^{\delta}$ for at least one $\zeta_i$ is at most $tn^{1/3}\exp\l(-n^{2\delta}/4\r)$ for all $n\geq n_{\delta}$. 
Fix $\varepsilon,\delta>0$. For this subsection, we work with the asymmetric matrix $\widetilde{H}_{n,\beta}$, 
\begin{align*}
     \widetilde{H}_{n,\beta}:=\begin{bmatrix}
X_1 & 1 & 0 &\cdots 0\\
Y_1^2 & X_2 & 1 & \cdots 0\\
0 & Y_2^2 & X_3 &\cdots 0\\
\vdots & \vdots & \vdots &\vdots
\end{bmatrix}_{n\times n},
\end{align*}
where $X_i\sim N\l(0,\frac{2}{\beta n}\r)$ i.i.d.\ random variables and $Y_i^2\sim \chi_{\beta\l(n-i\r)}^2/\beta n$ and all the entries of $\widetilde{H}_{n,\beta}$ are independent. By conjugation of $H_{n,\beta}/\sqrt{n} $ with the diagonal matrix $D$ where $D\l(i,i\r)=1/Y_1\dots Y_{i-1}$ and $D\l(1,1\r)=1$, we get the matrix $\widetilde{H}_{n,\beta}$. We work with the matrix $H^{\dagger}_{n,\beta}$ where $Y_i^2$ are replaced by $G_i=\frac{1}{n\beta}\l(\beta\l(n-i\r)+\sqrt{2\beta\l(n-i\r)}\zeta_i\r)$ for $i\leq p:=\l\lfloor tn^{1/3}\r\rfloor$, where $t_{\varepsilon}\leq t\leq n^{\frac{1}{9}-\delta}$ ($t_{\varepsilon}$ to be made precise later).  Here $\zeta_i$ are i.i.d.\ $N\l(0,1\r)$ random variables. We define a vector $u\in \R^{n+1}$ by $u\l(1\r)=1$ and $u\l(0\r)=0$ and the recursion
\begin{align}\label{eq: vector recursion}
    u\l(k+1\r)=\l(\lambda-X_k\r)u\l(k\r)-G_{k-1}u\l(k-1\r),
\end{align}
where $\lambda=2-tn^{-2/3}$. One can check using the tridiagonal matrix structure of $H^{\dagger}_{n,\beta}$ that if $u\l(k\r)>0$ for all $k$, then $\lambda_{\text{max}}\l(H^{\dagger}_{n,\beta}\r)\leq 2-tn^{-2/3}$.
% Similarly we denote $\widehat{T}_{n-p,\beta,\ell}$ to be the bottom $\l(n-p\r)\times\l(n-p\r)$ .

% Take $W_0=0$ and consider the event $\mathcal{A}=\l\{\l(\lambda-X_k-W_{k-1}\r)\geq 1 \text{ and } W_{k}>0 \text{ for all } 1\leq k\leq p \r\}$. Note that on the event $\mathcal{A}$ we have $u\l(1\r)\leq u\l(2\r) \dots \leq u\l(p\r)$. We compute lower bound for $\P\l(\mathcal{A}\r)$. Using independence we have,
% \begin{align*}
%     \P\l(\l(\lambda-X_k-W_{k-1}\r)\geq 1 \text{ for all } 1\leq k\leq p\r)&=\prod\limits_{k=1}^{p}\P\l(\l(\lambda-X_k-W_{k-1}\r)\geq 1 \r)\\
%     &=\P\l(X_1\leq 1-tn^{-\frac{2}{3}}\r) 
% \end{align*}

Writing $\frac{u\l(k+1\r)}{u\l(k\r)}=1+w\l(k\r)$ and changing index from $k$ to $k+1$ in \eqref{eq: vector recursion}, we have that
\begin{align*}
    w\l(k+1\r)-w\l(k\r)&=-tn^{-2/3}-\frac{w\l(k\r)^2}{1+w\l(k\r)}+\frac{k}{n\l(1+w\l(k\r)\r)}-\frac{1}{\sqrt{\beta n}}\left[\sqrt{\beta n}X_{k+1}+\frac{\sqrt{2}\sqrt{1-\frac{k}{n}}}{1+w\l(k\r)}\zeta_{k}\right]\\
    &=-tn^{-2/3}-\frac{w\l(k\r)^2}{1+w\l(k\r)}+\frac{k}{n\l(1+w\l(k\r)\r)}-\frac{2}{\sqrt{\beta n}}\sigma\l(k\r)Z_k.
\end{align*}
where $Z_k$ are i.i.d.\ $N\l(0,1\r)$ random variables and $\sigma\l(k\r)^2=\frac{1}{2}+\frac{1-\frac{k}{n}}{2\l(1+w\l(k\r)\r)^2}$. Further $w\l(1\r)=1-tn^{-2/3}-X_1$.

The recursion of $w\l(k\r)$ can be written as
\begin{align*}
    w\l(k+1\r)=\frac{w\l(k\r)}{1+w\l(k\r)}-tn^{-2/3}+\frac{k}{n\l(1+w\l(k\r)\r)}-\frac{2}{\sqrt{\beta n}}\sigma\l(k\r)Z_k.
\end{align*}

 Denote the event $E_1=\{|Z_k|\leq n^{\frac{1}{6}-\frac{\delta}{6}} \text{ for all } k\leq p\}$. The event $E_1$ occurs with probability at least $1-n\exp\l(-n^{\frac{1}{3}-\frac{\delta}{3}}/2\r)$. Then on the event $E_1$, one can check that $w\l(k+1\r)$ is an increasing function of $w\l(k\r)$, if $w\l(k\r)>0$. Thus we start with $w\l(1\r)=5n^{-1/3}$ and calculate the probability that $w\l(k\r)>n^{-1/3}$ for $1\leq k\leq p$. Define $\mathcal{A}:=\l\{w\l(k\r)>n^{-1/3} \mbox{ for }k\leq p\r\}$.

% If $w\l(k+1\r)<0$ and $E_1$ does not occur, then $w\l(k\r)<3n^{-1/3}$. Thus, $w$ cannot jump to the negative side without entering the interval $[0,3n^{-1/3}]$. We also have that $w\l(k+1\r)$ is an increasing function of $w\l(k\r)$. Thus to get lower bound of probability of the event that vector $u>0$, we may start with $w\l(1\r)=5n^{-1/3}$ and calculate the probability that $w\l(k\r)>-1$ for $1\leq k\leq n$.

% and $\mathcal{B}=\l\{\lambda_{\text{max}}\l(\widehat{T}_{n-tn^{1/3},\beta,\ell}<2\r)\r\}$.

\vspace{1mm}
\textbf{Estimating $\P\l(\mathcal{A}\r):$} Let $W\l(k\r)=n^{1/3}w\l(k\r)$ so that the recursion becomes,
\begin{align*}
    W\l(k+1\r)-W\l(k\r)=-tn^{-1/3}-\frac{W\l(k\r)^2}{n^{1/3}+W\l(k\r)}+\frac{k}{n^{2/3}\l(1+n^{-1/3}W\l(k\r)\r)}-\frac{2}{\sqrt{\beta }n^{1/6}}\sigma_k\l(W_k\r)Z_k,
\end{align*}
with $\sigma_k^2\l(x\r)=\frac{1}{2}+\frac{1-\frac{k}{n}}{2\l(1+n^{-1/3}x\r)^2}$. Note that $W\l(1\r)=5.$ Define
\begin{align*}
    \widehat{W}\l(k+1\r)-\widehat{W}\l(k\r)=-tn^{-1/3}-\frac{49}{n^{1/3}}+\frac{k}{n^{2/3}}-\frac{2}{\sqrt{\beta }n^{1/6}}Z_k,
\end{align*}
where $\widehat{W}\l(1\r)=5$. Call $\Delta\l(W\l(k\r)\r)=W\l(k+1\r)-W\l(k\r)$ and $\mathcal{A}_1=\l\{1\leq \widehat{W}\l(k\r)\leq 6\mbox{ for all }k\leq p\r\}$. We look at the behaviour of $W\l(k\r)$ on the event $\mathcal{A}_1\cap E_1$. Note that, if $W\l(k\r)<13/2,$ we have that $\Delta\l(W\l(k\r)\r)>\Delta\l(\widehat{W}\l(k\r)\r)$. And if $W\l(k\r)>15/2,$ we have that $\Delta\l(W\l(k\r)\r)<\Delta\l(\widehat{W}\l(k\r)\r)$. Also for the range $1\leq W\l(k\r)\leq 15$, we have that $\Delta\l(W\l(k\r)\r),\Delta\l(\widehat{W}\l(k\r)\r)\leq c_{\beta}/n^{\delta/6}$ (due to the event $E_1$). This is why when $W(k)>13/2$, it does not jump below $6$ and $\widehat{W}(k)$ in the next step. Since $W\l(1\r)=\widehat{W}\l(1\r)=5$, the random walks $\wih{W}(k)\leq W(k)$ as long as $W(k)\leq 13/2$. As the step sizes are of the order $1/n^{\delta/6}$, the first time when $W(k)>15/2$ and $W(k)$ also crosses $15$, the random walk $\wih{W}(k)$ would also have to cross $6$. But that cannot happen. Hence on the event $\mathcal{A}_1\cap E_1$ we have that $1\leq \widehat{W}\l(k\r)\leq W\l(k\r)\leq 15$.

As a result we compute the lower bound for $\P\l(\mathcal{A}_1\r)$. We have that 
\begin{align*}
    \widehat{W}\l(k+1\r)-5=-\l(t+49\r)\frac{k}{n^{1/3}}+\frac{k\l(k+1\r)}{2n^{2/3}}-\frac{2}{\sqrt{\beta}}\frac{1}{n^{1/6}}\sum\limits_{j=1}^kZ_j.
\end{align*}

Let $B$ be a standard Brownian motion and embed 
\begin{align*}
    k\rightarrow\frac{2}{\sqrt{\beta}}\frac{1}{n^{1/6}}\sum\limits_{j=1}^k Z_j
\end{align*}
in to it in the natural way. 
% Thus the entire random walk $\widehat{W}(k)$ is embedded in $B[0,t].$
Therefore the event $\mathcal{A}_1$ occurs if the following event occurs:
\begin{align*}
    -4\leq -\l(t+49+\frac{1}{2n^{1/3}}\r)s+\frac{1}{2}s^2-\frac{2}{\sqrt{\beta}}B\l(s\r)\leq 1 \quad \mbox{ for }0\leq s\leq t.
\end{align*}

This is the same as 
\begin{align*}
    -\frac{\sqrt{\beta}}{2}\leq B\l(s\r)-g\l(s\r)\leq \frac{4\sqrt{\beta}}{2} \quad \mbox{ for } 0\leq s\leq t,
\end{align*}

where $g\l(s\r)=\frac{\sqrt{\beta}}{2}\l(s\l(t+49+\frac{1}{2n^{1/3}}\r)-\frac{1}{2}s^2\r)$. By the Cameron-Martin formula, for large enough $t$ (choose $t_{\varepsilon}$ here) 
\begin{align*}
    \P\l(\mathcal{A}_1\r)&\geq e^{-\frac{1}{2}\int\limits_{0}^{t}g'\l(s\r)^2ds}\P\l(-\frac{\sqrt{\beta}}{2}\leq B\l(s\r)\leq \frac{4\sqrt{\beta}}{2} \mbox{ for } 0\leq s\leq t\r)\\
    &\geq e^{-\l(\frac{\beta}{24}+\varepsilon\r)t^3}
\end{align*}
since the probability that a standard Brownian motion stays in a fixed interval around the origin for duration $t$ is at least $e^{-O\l(t\r)}$.

\vspace{1mm}

 We have shown that, $\mathcal{A}$ occurs with probability at least $e^{-\l(\frac{\beta}{24}+\varepsilon\r)t^3}-n\exp\l(-n^{\frac{1}{3}-\frac{\delta}{3}}/2\r)$ (the probability of the event $E_1$ is accounted in this) and we have, $u\l(k\r)>0$ for $k\leq p+1$ and $u\l(p+1\r)>u\l(p\r)$ (as $w\l(p\r)>0$). To apply part Proposition \ref{prop: Tridiag resultwith1}, we need to show that the matrix $\overline{H^{\dagger}}_{\text{low}_{n-p}}$ obtained by replacing first diagonal entry of $H^{\dagger}_{\text{low}_{n-p}}$ (bottom $n-p$ sub-matrix of $H^{\dagger}_{n,\beta}$) with $X_{p+1}+G_p^2$ has $\lambda_{\text{max}}\leq \lambda$. Note that $\lambda_{\text{max}}\l(\overline{H^{\dagger}}_{\text{low}_{n-p}}\r)$ is independent of $u\l(k\r)$ for $k\leq p+1$. In particular, $\P\l(\lambda_{\text{max}}\l(\overline{H^{\dagger}}_{\text{low}_{n-p}}\r)<\lambda|\mathcal{A}\r)=\P\l(\lambda_{\text{max}}\l(\overline{H^{\dagger}}_{\text{low}_{n-p}}\r)<\lambda\r)\geq c$ for some positive constant $c$ for all large enough $n$. Indeed this follows from the spiked random matrix result, Theorem $1.5$ of \cite{BV12}, by taking $\mu=1$. By Proposition \ref{prop: Tridiag resultwith1} we have that for any given $\varepsilon,\delta>0$ and $t_{\varepsilon}\leq t\leq n^{\frac{1}{9}-\delta}$, for all $n\geq n_{\varepsilon,\delta}$, we have $\P\l(\lambda_{\text{max}}\l(H^{\dagger}_{n,\beta}\r)\leq 2-tn^{-2/3}\r)\geq ce^{-\l(\frac{\beta}{24}+\varepsilon\r)t^3}$. Also with probability at least $1-2ne^{-\frac{\beta\l(n-p\r)}{8}}$, all the off-diagonal entries of $H^{\dagger}_{n,\beta}$ are non-negative. On that event we conjugate $H^{\dagger}_{n,\beta}$ to get back symmetric tridiagonal matrix. Note that this matrix is slightly different from $\wih{H}_{n,\beta,p}/\sqrt{n}$. Applying Gershgorin circle theorem one can check that $\P\l(\lambda_{\text{max}}\l(\wih{H}_{n,\beta,p}/\sqrt{n}\r)\leq 2-tn^{-2/3}+\frac{c_{\beta}x}{n}\r)\geq ce^{-\l(\frac{\beta}{24}+\varepsilon\r)t^3}-2ne^{-\frac{\beta\l(n-p\r)}{8}}-4ne^{-x}$. Taking $x=tn^{\frac{1}{3}-\frac{\delta}{3}}$ and applying Proposition \ref{eq: estimate of off diagonals of the difference} we have that for any given $\varepsilon,\delta>0$ and $t_{\varepsilon,\delta,\beta}\leq t\leq n^{\frac{1}{9}-\delta}$, for all $n\geq n_{\varepsilon,\delta,\beta}$, we have $\P\l(\lambda_{\text{max}}\l(H_{n,\beta}/\sqrt{n}\r)\leq 2-tn^{-2/3}\r)\geq e^{-\l(\frac{\beta}{24}+\varepsilon\r)t^3}$. 

\subsection{Laguerre Ensemble}

Fix $\beta>0$ and fix small $\varepsilon,\delta>0$.
Let $m=m\left(n\right)\geq n$ be such that there exist $M>0$ such that $1  \leq \frac mn <M$. We denote $\gamma:=m/n$ and $a=\left(1+\sqrt{\gamma}\right)^2$ and $b=\gamma^{-1/6}\left(1+\sqrt{\gamma}\right)^{4/3}.$ We consider  matrix $\widehat{L}_{n,m,\beta,p}$ where $p=\left\lfloor\frac{t n^{1/3} \left(\sqrt{\gamma}\right)^{2/3}}{\left(1+\sqrt{\gamma}\right)^{2/3}}\right\rfloor$ and obtain lower bound for probability of the event that $\lambda_{\text{max}}\left(\widehat{L}_{n,m,\beta,p}\right)\leq an\left(1-\frac{t(1-\delta)bn^{-2/3}}{a}\right)$, for some $\delta>0$ to be chosen later. Here $t_{\varepsilon}\leq t\leq n^{\frac{1}{9}-\delta}$ and the value $t_{\varepsilon}$ will be made precise later. We take $\lambda=\left(1+\sqrt{\gamma}\right)^2-\frac{t(1-\delta)\left(1+\sqrt{\gamma} \right)^{4/3}}{n^{2/3}\left(\sqrt{\gamma}\right)^{1/3}}$.

Note that $\widehat{L}_{n,m,\beta,p}$ can be written as a sum of two matrices $L'$ and $L''$, where $L''$ is a tridiagonal symmetric matrix whose entries are linear combinations of $\zeta_k^2,\left(\zeta'_k\right)^2, \zeta_{k+1}\zeta'_k$ in the top $p\times p$ submatrix and zeros elsewhere. We will work with the matrix $L'.$ The conclusion about the matrix $L'+L''$ can be done by arguing similarly as we did before using Gershgorin circle theorem.
%$A$ contains random variables which are Gaussian in the top $p\times p$ submatrix.

% We denote $\gamma:=\gamma\l(n\r)=m/n$ and $a=\l(1+\sqrt{\gamma}\r)^2$ and $b=\gamma^{-1/6}\l(1+\sqrt{\gamma}\r)^{4/3}.$

 % Here we restrict ourselves to $m=n+c$, for some fixed $c>0$.

 %   We want to show that for any $\varepsilon>0$ and $\delta>0$, there exists $n_{\varepsilon,\delta}$ and $M>0$ such that for $M\leq t_n\leq n^{\frac{1}{9}-\varepsilon}$ and $n\geq n_{\varepsilon,\delta} $
% \begin{align*}
% \mathbb{P}\l(\lambda_{\text{max}}\l(\wih{T}_{2n,\beta}\r)\leq \l(\sqrt{m}+\sqrt{n}\r)\sqrt{1-\frac{t_nbn^{-2/3}}{a}+\frac{t_n^2b^2}{4a^2n^{4/3}}}\r)\geq \exp\l(-\l(1+\varepsilon\r)\frac{\beta}{24}t_n^3\r). 
% \end{align*}

% We will show that 
% \begin{align*}
% \mathbb{P}\l(\lambda_{\text{max}}\l(\wih{S}_{2n,\beta}\r)\leq \l(\sqrt{\frac{m}{n}}+1\r)\sqrt{1-\frac{t_nbn^{-2/3}}{a}}\r)\geq \exp\l(-\l(1+\varepsilon\r)\frac{\beta}{24}t_n^3\r). 
% \end{align*}

% As before we replace chi random variables by Gaussians and repalce $\sqrt{1-\frac{t_nbn^{-2/3}}{a}}$ by $1-\l(t_n/2\r) n^{-2/3}$. 
We define $L=\frac{1}{\beta n}L'.$ We wish to apply Proposition \ref{prop: Tridiag resultwith1} with $b_i=c_i$ for all $i\leq n$. To do this first we look at the top $p \times p$ sub-matrix of $L$ and consider the event $\{u(1) <u(2)< \dots < u(p+1) \}$ where the $u(i)$'s are defined as follows. We define a vector $u\in \mathbb{R}^{n+1}$ by $u\left(1\right)=1$ and $u\left(0\right)=0$ and the recursion
\begin{align}\label{eq: vector recursion 1}
    L_{k,k+1}u\left(k+1\right)=\left(\lambda-L_{k,k}\right) u\left(k\right)-L_{k-1,k}u\left(k-1\right),
\end{align}
where $L_{n,n+1}=1$
By the tridiagonal structure of $L$ we have that if $u\left(k\right)>0$ for all $k$, then $\lambda_{\text{max}}\left(L\right)\leq \lambda$.
Writing $\frac{L_{k,k+1}u\left(k+1\right)}{u\left(k\right)}=L_{k,k+1}+w\left(k\right)$, and observing that \\ $L_{k,k}=\frac{1}{\beta n}\l(\beta (m+1-k)+\beta (n-k)+\sqrt{2\beta(m+1-k)}\zeta_k+\sqrt{2 \beta (n-k)}\zeta_k' \r)$ we have that for all $k\leq p$,
\begin{align*}
% \label{eq: RW recursion for laguerre1}
    w\left(k\right)&=1+2\sqrt{\gamma}+\gamma-\frac{t(1-\delta)\left(1+\sqrt{\gamma}\right)^{4/3}}{\left(\sqrt{\gamma}\right)^{1/3}n^{2/3}}-L_{k,k}-\frac{L_{k-1,k}^2}{L_{k-1,k}+w\left(k-1\right)}-L_{k,k+1}\\
\label{eq: RW recursion for laguerre2}    &=2\sqrt{\gamma}+\frac{2k-1}{n}-\frac{\sqrt{2\beta\left(m-k+1\right)}\zeta_k}{n\beta}-\frac{\sqrt{2\beta\left(n-k\right)}{\zeta'}_k}{n\beta}-\frac{t(1-\delta)\left(1+\sqrt{\gamma}\right)^{4/3}}{\left(\sqrt{\gamma}\right)^{1/3}n^{2/3}}\\
&-L_{k,k+1}-L_{k-1,k}+\frac{L_{k-1,k}w(k-1)}{L_{k-1,k}+w(k-1)}. \nonumber
    %&-\frac{\beta^2\left(m-k+1\right)\left(n-k+1\right)+\beta\left(n-k+1\right)\sqrt{2\beta\left(m-k+1\right)}\zeta_k+\beta\left(m-k+1\right)\sqrt{2\beta\left(n-k+1\right)}{\zeta'}_{k-1}}{n^2\beta^2\left(L_{k,k-1}+w\left(k-1\right)\right)}\nonumber\\
    %&-\frac{\beta(m+1-k)\frac{\zeta'^2_{k-1}}{2}+2\beta \sqrt{(m+1-k)(n+1-k)}\zeta_k\zeta_{k-1}'+\beta(n-k+1)\frac{\zeta_k^2}{2}}{n^2\beta^2\left(L_{k,k-1}+w\left(k-1\right)\right)}\nonumber
    % & -\frac{\beta\l(n+c-k+1\r)\frac{\widehat{\zeta}_k^2}{2}+\beta\l(n-k\r)\frac{\zeta_k^2}{2}+\l(\sqrt{2\beta\l(n+c-k+1\r)}{\zeta}_k+\frac{\zeta_k^2}{2}\r)\l(\sqrt{2\beta\l(n-k\r)}\widehat{\zeta}_k+\frac{\widehat{\zeta}_k^2}{2}\r)}{n^2\beta^2\l(1+w\l(k-1\r)\r)}\nonumber
    % % \label{eq: RW recursion for laguerre2}
    % % &=\sqrt{c}+ \frac{\alpha}{2\sqrt{c}n}  - g_n\frac{t_n}{2n^{2/3}d_n}-\left[\frac{\beta\l(n-k\r)+\sqrt{2\beta}\sqrt{{n-k}}\zeta_{k}}{\beta n\l(1+w\l(k\r)\r)}\right] \mbox{ if k is even 
\end{align*}
Note $w\left(k\right)$ is an increasing function of $w\left(k-1\right)$ for all $k
   \leq p$. Also, we have 
   \begin{align*}
   &w\left(1\right)=L_{1,2}u(2)-L_{1,2}=\left(\lambda-L_{1,1} \right)-L_{1,2}.
   % &\textcolor{blue}{=2\sqrt{\gamma}-\frac{t(1-\delta)\left(1+\sqrt{\gamma} \right)^{4/3}}{\left(\sqrt{\gamma} \right)^{1/3}n^{2/3}}-\sqrt{\frac{2\gamma}{\beta}}\frac{\zeta_2}{\sqrt{n}}-\sqrt{\frac{2(1-\frac 1n)}{\beta}}\frac{\zeta_1'}{\sqrt{n}}}\\
   % &\textcolor{blue}{-\sqrt{\left(1-\frac 1n \right)\left( \gamma-\frac 1n\right)}-\sqrt{\frac{1-\frac 1n}{2\beta }}\frac{\zeta_1}{\sqrt{n}}-\sqrt{\frac{\gamma-\frac 1n}{2\beta}}\frac{\zeta_1'}{\sqrt{n}}}.
   \end{align*} 
   Note that the above quantity is positive and close to $\sqrt{\gamma}$ for $n$ large on an event $\mathcal{E}_1$ on which $|\zeta_i|,|\zeta_i'| \leq n^{1/6-\delta}$. We see that $\P \left( \mathcal{E}_1^c\right) \leq 2n \exp \left(-\frac{n^{1/3-2\delta}}{2} \right)$. We start at $w\left(1\right)=5n^{-1/3}$ and compute the probability of the event that $w\left(k\right)>0$ for $k\leq p$.

%     \begin{align*}
%     w\l(k+1\r)&=\frac{\l(\sqrt{c}+ \frac{\alpha}{2\sqrt{c}n}-c-\frac{\alpha+1-k}{n}\r)}{1+w\l(k\r)}  - g_n\frac{t_n}{2n^{2/3}d_n}+\frac{\l(\sqrt{c}+\frac{\alpha}{2\sqrt{c}n}\r)w\l(k\r)}{1+w\l(k\r)}-\left[\frac{\sqrt{2\beta}\sqrt{{m+1-k}}\zeta_{k}}{\beta n\l(1+w\l(k\r)\r)}\right]\mbox{\l(k odd\r)}\\ 
%     &=\frac{\l(\sqrt{c}+ \frac{\alpha}{2\sqrt{c}n}-1\r)}{1+w\l(k\r)}  - g_n\frac{t_n}{2n^{2/3}d_n}+\frac{\l(\sqrt{c}+\frac{\alpha}{2\sqrt{c}n}\r)w\l(k\r)}{1+w\l(k\r)}-\left[\frac{\sqrt{2\beta}\sqrt{{n-k}}\zeta_{k}}{\beta n\l(1+w\l(k\r)\r)}\right]\mbox{\l(k even\r)}
% \end{align*}

% \textcolor{red}{We would like $w\l(k\r)>0$, so in order to cancel out the constant order terms in \eqref{eq: RW recursion for laguerre1} and \eqref{eq: RW recursion for laguerre2}, $\sqrt{c}=c=1$ is forced. So we take $c=1$}. Taking $\Delta w\l(k+1\r)=w\l(k+1\r)-w\l(k\r)$, when $k$ is odd we have that,
% \begin{align*}
%     \Delta w\l(k+1\r)&=-\frac{g_nt_n}{2n^{2/3}d_n}+\frac{\alpha}{2n}-\frac{\alpha}{n\l(1+w\l(k\r)\r)}-\frac{w\l(k\r)^2}{1+w\l(k\r)}+\frac{k-1}{n\l(1+w\l(k\r)\r)}-\frac{1}{\sqrt{\beta n}}\left[\frac{\sqrt{2}\sqrt{1+\frac{\alpha+1-k}{n}}\zeta_{k}}{1+w\l(k\r)}\right].
% \end{align*}

% When $k$ is even we have,
% \begin{align*}
%     \Delta w\l(k+1\r)&=-\frac{g_nt_n}{2n^{2/3}d_n}+\frac{\alpha}{2n}-\frac{w\l(k\r)^2}{1+w\l(k\r)}+\frac{k}{n\l(1+w\l(k\r)\r)}-\frac{1}{\sqrt{\beta n}}\left[\frac{\sqrt{2}\sqrt{1-\frac{k}{n}}\zeta_{k}}{1+w\l(k\r)}\right].
% \end{align*}

We take $W\left(k\right)=n^{1/3}w\left(k\right)$ (we get $W\left(1\right)=5$) and we look at the differences, 
% \begin{align*}
%     W\l(3\r)-W\l(1\r)&=-\frac{g_nt_n}{n^{1/3}d_n}+\frac{\alpha}{n^{2/3}}-\frac{\alpha}{n^{2/3}\l(1+n^{-1/3}W\l(1\r)\r)}-\frac{W\l(1\r)^2}{n^{1/3}+W\l(1\r)}-\frac{W\l(2\r)^2}{n^{1/3}+W\l(2\r)}+\frac{2}{n^{2/3}\l(1+n^{-1/3}W\l(2\r)\r)}\\
%     &+{\frac{\sqrt{2}}{\sqrt{\beta} n^{1/6}}}\sigma_1Z_1
% \end{align*}
% where $Z_1$ is $N\l(0,1\r)$ and $\sigma_1^2=\frac{1+\frac{\alpha}{n}}{\l(1+n^{-1/3}W\l(1\r)\r)^2}+\frac{1-\frac{2}{n}}{\l(1+n^{-1/3}W\l(2\r)\r)^2}$. In general for all $k\leq s/2$, we have
% \begin{align}\label{eq: Laguerre lower bound W recursion}
%     W\l(k\r)-W\l(k-1\r)&=-\frac{g_nt_n}{n^{1/3}d_n}+\frac{\alpha}{n^{2/3}}-\frac{\alpha}{n^{2/3}\l(1+n^{-1/3}W\l(2k-1\r)\r)}-\frac{W\l(2k-1\r)^2}{n^{1/3}+W\l(2k-1\r)}-\frac{W\l(2k\r)^2}{n^{1/3}+W\l(2k\r)}\\
%     &+\frac{2k-2}{n^{2/3}\l(1+n^{-1/3}W\l(2k-1\r)\r)}+\frac{2k}{n^{2/3}\l(1+n^{-1/3}W\l(2k\r)\r)}+{\frac{\sqrt{2}}{\sqrt{\beta} n^{1/6}}}\sigma_{2k-1}Z_{2k-1}\nonumber
% \end{align}
%  where $Z_{2k-1}$ are i.i.d.\ standard normal random variables and $\sigma_{2k-1}^2=\frac{1+\frac{\alpha-2k+2}{n}}{\l(1+n^{-1/3}W\l(2k-1\r)\r)^2}+\frac{1-\frac{2k}{n}}{\l(1+n^{-1/3}W\l(2k\r)\r)^2}$.

%  We also have that 
 \begin{align*}
 % \label{eq: even jump bound}
      & W\left(k\right)-W\left(k-1\right)=2\sqrt{\gamma}n^{1/3}-\frac{t(1-\delta)\left(1+\sqrt{\gamma} \right)^{4/3}}{\left(\sqrt{\gamma}\right)^{1/3}n^{1/3}}+\frac{2k-1}{n^{2/3}}-\sqrt{2\left(1-\frac kn \right)}\frac{\zeta_k'}{\sqrt{\beta}n^{1/6}}\\
      &-\sqrt{2\left(\gamma-\frac{k-1}{n} \right)}\frac{\zeta_k}{\sqrt{\beta}n^{1/6}}-\left(L_{k,k+1}+L_{k-1,k}\right)n^{1/3}-\frac{W(k-1)^2}{n^{1/3}L_{k,k-1}+W(k-1)}.
      \end{align*}
      We plug in the value of $L_{k,k+1}=\frac{1}{\beta n}\l(\beta\sqrt{(m-k)(n-k)}+\sqrt{\beta (m-k)}\frac{\zeta_k'}{\sqrt{2}}+\frac{\zeta_{k+1}}{\sqrt{2}}\sqrt{\beta (n-k)} \r)$ and $L_{k-1,k}$ we get
      % &\textcolor{blue}{=2\sqrt{\gamma}n^{1/3}-\frac{t(1-\delta)\left(1+\sqrt{\gamma} \right)^{4/3}}{\left(\sqrt{\gamma}\right)^{1/3}n^{1/3}}+\frac{2k-1}{n^{2/3}}-\sqrt{2\left(1-\frac kn \right)}\frac{\zeta_k'}{\sqrt{\beta}n^{1/6}}-\sqrt{2\left(\gamma-\frac{k-1}{n} \right)}\frac{\zeta_k}{\sqrt{\beta}n^{1/6}}}\\
      % &\textcolor{blue}{-n^{1/3}\sqrt{\left(1-\frac kn \right)\left(\gamma-\frac kn \right)}-\sqrt{\left(1-\frac kn \right)}\frac{\zeta_{k+1}}{\sqrt{2\beta}n^{1/6}}-\sqrt{\left(\gamma-\frac kn \right)}\frac{\zeta_k'}{\sqrt{2\beta}n^{1/6}}-n^{1/3}\sqrt{\left(1-\frac{k-1}{n} \right)\left(\gamma -\frac{k-1}{n}\right)}}\\
      % &\textcolor{blue}{-\sqrt{\left(1-\frac{k-1}{n} \right)}\frac{\zeta_{k}}{\sqrt{2\beta}n^{1/6}}-\sqrt{\left(\gamma-\frac {k-1}{n} \right)}\frac{\zeta_{k-1}'}{\sqrt{2\beta}n^{1/6}}-\frac{W(k-1)^2}{n^{1/3}L_{k-1,k}+W(k-1)}}\\
      \begin{align*}
      &W(k)-W(k-1)=-\frac{t(1-\delta)\left(1+\sqrt{\gamma} \right)^{4/3}}{\left(\sqrt{\gamma}\right)^{1/3}n^{1/3}}+\frac{2k-1}{n^{2/3}}-\sqrt{2\left(1-\frac kn \right)}\frac{\zeta_k'}{\sqrt{\beta}n^{1/6}}\\&-\sqrt{2\left(\gamma-\frac{k-1}{n} \right)}\frac{\zeta_k}{\sqrt{\beta}n^{1/6}}-\sqrt{\left(1-\frac kn \right)}\frac{\zeta_{k+1}}{\sqrt{2\beta}n^{1/6}}-\sqrt{\left(\gamma-\frac kn \right)}\frac{\zeta_k'}{\sqrt{2\beta}n^{1/6}}\\&-\sqrt{\left(1-\frac{k-1}{n} \right)}\frac{\zeta_{k}}{\sqrt{2\beta}n^{1/6}}-\sqrt{\left(\gamma-\frac {k-1}{n} \right)}\frac{\zeta_{k-1}'}{\sqrt{2\beta}n^{1/6}}+\frac{(1+\gamma)(2k-1)}{2\sqrt{\gamma}n^{2/3}}+O\left(\frac {k^2}{n^{5/3}} \right)\\&-\frac{W(k-1)^2}{n^{1/3}L_{k-1,k}+W(k-1)}.
      %\frac{W\left(k-1\right)^2}{n^{1/3}\sqrt{\gamma}+W\left(k-1\right)}-\frac{\sqrt{2\beta\left(\gamma n-k+1\right)}\zeta_k}{n^{2/3}\beta}{\left(1+\frac{\left(1-\frac{k-1}{n}\right)}{\sqrt{\gamma}+w\left(k-1\right)}\right)}\\
    %&-\frac{\sqrt{2\beta\left(n-k\right)}{\zeta'}_k}{n^{2/3}\beta}-\frac{\sqrt{2\beta\left(n-k+1\right)}{\zeta'}_{k-1}}{n^{2/3}\beta}{\left(\frac{\gamma-\frac{k-1}{n}}{\sqrt{\gamma}+w\left(k-1\right)}\right)}+n^{1/3}R(m,n,k)\nonumber
 \end{align*}
%where
%\begin{align*}
%    R(m,n,k,)=\frac{-\left(k-1\right)^2+(m+n)\left(k-1)\right)}{n^2\left(\sqrt{\gamma}+w\left(k-1\right)\right)}
%\end{align*}
We analyse a different but closely related process $\widehat{W}\left(k\right)$ with $\widehat{W}\left(1\right)=5$ and for all $k\leq p$ define
\begin{align*}
    \widehat{W}\left(k\right)-\widehat{W}\left(k-1\right)&=\frac{-t(1-\delta)\left(1+\sqrt{\gamma} \right)^{4/3}}{\left(\sqrt{\gamma}\right)^{1/3}n^{1/3}}+\frac{\left(1+\sqrt{\gamma}\right)^2}{2\sqrt{\gamma}}\frac{2k-1}{{n^{2/3}}}-\frac{M}{n^{1/3}} \nonumber\\
    &-\frac{\sqrt{2}\zeta_k'}{\sqrt{\beta}n^{1/6}}-\frac{\sqrt{2\gamma}\zeta_k}{\sqrt{\beta}n^{1/6}}-\frac{\zeta_{k+1}}{\sqrt{2}\sqrt{\beta}n^{1/6}}-\frac{\sqrt{\gamma}\zeta_k'}{\sqrt{2\beta}n^{1/6}}-\frac{\zeta_k}{\sqrt{2\beta}n^{1/6}}-\frac{\sqrt{\gamma}\zeta_{k-1}'}{\sqrt{2\beta}n^{1/6}},
    %&-\left(1+\frac{1}{\sqrt{\gamma}} \right)\frac{\sqrt{2\gamma}}{\sqrt{\beta} n^{1/6}}{\zeta}_{k}-{\frac{\sqrt{2}}{\sqrt{\beta} n^{1/6}}}{\zeta'}_{k}-{\frac{\sqrt{2\gamma}}{\sqrt{\beta} n^{1/6}}}{\zeta'}_{k-1}
\end{align*}

 where we choose $M$ is chosen so that on the event $\ce_1$
 \[
 \frac{49}{n^{1/3}L_{k-1,k}+7} \leq \frac{M}{n^{1/3}}.
 \]
 Define the event $\mathcal{A}_1:=\left\{1\leq \widehat{W}\left(k\right)\leq 6\mbox{ for all }k\leq p\right\}$. We look at the behaviour of $W\left(k\right)$ on the event $\mathcal{A}_1 \cap \ce_1$. Note that, if $0 < W\left(k\right)<13/2,$ we have that $\Delta\left(W\left(k\right)\right)>\Delta\left(\widehat{W}\left(k\right)\right)$. And there exists $N$ such that if $W\left(k\right)>N,$ we have that $\Delta\left(W\left(k\right)\right)<\Delta\left(\widehat{W}\left(k\right)\right)$. Now we have $W(1)=\widehat{W}(1)=5.$ So, $W(2) \geq \widehat{W}(2).$ Now, if $W(2) < 13/2,$ we have $W(3) \geq \widehat{W}(3).$ But, if $W(2) > 13/2,$ we apply the following argument. If possible let $W(3) < \widehat{W}(3)$ in this case. So, $W(3)-W(2) \leq -\frac 12.$ Two cases are possible. If $13/2 \leq W(2) \leq N$ then we see that $W(3)-W(2)$ is of order $\frac{1}{n^{\delta/6}}$ and we get $W(3) \geq \widehat{W}(3)$. Now, if $W(2) \geq N,$ we have 
 \[
 W(2)-W(1) < \widehat{W}(2)-\widehat{W}(1),
 \]
 which is not possible. Continuing this way, we see that as long as $W(k) \leq N,$ we have $W(k+1) \geq \widehat{W}(k+1).$ Let $k_0$ is the first time $W(k_0) > N$. Then 
 \[
 W(k_0)-W(k_0-1) \leq \widehat{W}(k_0)-\widehat{W}(k_0-1) \leq 6.
 \]
 By the above argument it is clear that $W(k)$ will always lie in a bounded set. Hence, $\Delta \left(W(k) \right)$ will always be of order $\frac{1}{n^{\delta/6}}$
 %Then $N-W(1)<W(2)-W(1) \leq \widehat{W}(2)-\widehat{W}(1)$. 
 %Now, as $N>7$ and on the event $\mathcal{A}_1$ and on an event denoted as $E_1$, with probability at least $1-n\exp\left(-n^{\frac{1}{3}-\frac{\delta}{3}}/2\right), \Delta\left(\widehat{W}(k) \right)$ is of order $\frac{1}{n^{\delta/6}}$ we get a contradiction. So, $W(2) \leq N.$ But for the range $1\leq W\left(k\right)\leq 2N$, and on the event $E_1$, we have that $\Delta\left(W\left(k\right)\right)$ is of
 %\Delta\left(\widehat{W}\left(k\right)\right)
 %we get that on $\mathcal{A}_1 \cap E_1, W(3) \geq \widehat{W}(3)$. We apply the same argument. 
 %If $W(3)< 13/2,$ we have $W(4) \geq \widehat{W}(4) $. If $13/2 <W(3)< N$ then also we have that $W(4) \geq \widehat{W}(4)$. Now, if $W(3)>N, W(4)-W(3) < \widehat{W}(4)-\widehat{W}(3)$
 %and $\Delta\left(\widehat{W} \right) \leq \frac{1}{n^{\delta/6}}$ event denoted as $E_1$, with probability at least $1-n\exp\left(-n^{\frac{1}{3}-\frac{\delta}{3}}/2\right)$. 
. So, on the event $\mathcal{A}_1 \cap \mathcal{E}_1,$ 
 %if $W(k-1) > 13/2$ then $\Delta \left(W(k) \right) \leq 1/n^{\delta/6}$ .Since $W\left(1\right)=\widehat{W}\left(1\right)=5$, on the event $\mathcal{A}_1\cap E_1$ 
 we have that $1\leq \widehat{W}\left(k\right)\leq W\left(k\right)\leq 2N$.
 % Also define the event $E_2$ such that $|\widehat{Z}_k|\leq n^{\frac{1}{6}-\delta}$ for all $k \leq s$. We have $\P\l(E_2\r)\geq 1-s\exp\l(-n^{\frac{1}{3}-2\delta}/4\r)$. We look at the sequence $W\l(k\r)$ on the event $\mathcal{A}_1\cap E_1\cap E_2$.

% Note that, if $W\l(k-1\r)<7-\frac{1}{2}$ then, due to \eqref{eq: even jump bound}, it can be checked that $W\l(k\r)<7$ (\textcolor{red}{as $\alpha<n^{\frac{2}{3}-\delta}$, the jumps are of order $o\l(1\r)$}). This results in  $W\l(k\r)-W\l(k-1\r)>\widehat{W}\l(k\r) -\widehat{W}\l(k-1\r)$.  And if $W\l(k\r)>15/2,$ we have that $\Delta\l(W\l(k\r)\r)<\Delta\l(\widehat{W}\l(k\r)\r)$. Also for the range $1\leq W\l(k\r)\leq 15$, we have that $\Delta\l(W\l(k\r)\r),\Delta\l(\widehat{W}\l(k\r)\r)\sim 1/n^{1/6}$. Also as we are using same source of randomness in $W\l(k\r)$ and $\widehat{W}\l(k\r)$, on the event $\mathcal{A}_1\cap E_1\cap E_2$, inductively (as $W\l(1\r)=\widehat{W}\l(1\r)=5$) one can see that $1\leq \widehat{W}\l(k\r)\leq W\l(k\r)\leq 15$ for $k\leq s$. 

% Since $W\l(1\r)=\widehat{W}\l(1\r)=5$, by inductively on the event $\mathcal{A}_1\cap E_1\cap E_2$ we have that $1\leq \widehat{W}\l(k\r)\leq W\l(k\r)\leq 15$
 
% Repeating the exact same argument as given in the case of Hermite ensemble, gives us that, for the range any $0<\delta<1/2$ and  $ t_n^3< c_{\beta}\log\l(n^{\delta}\r),$ where $0<c_{\beta}<\infty$ and $t_n\rightarrow\infty$, we have that 
% \begin{align*}
%     \mathbb{P}\l(\lambda
% _{max}\l(W_{n,n+c,\beta}\r)\leq 4n-2^{4/3}t_nn^{1/3} \r)\geq \exp\l(-\l(1+o\l(1\r)\r)\frac{\beta}{24}t_n^3\r).
% \end{align*}

\textbf{$\P\left(\mathcal{A}_1\right):$} we have that for $1\leq k\leq p-1$
\begin{align*}
    &\widehat{W}\left(k+1\right)-5=-\left(\frac{t(1-\delta)\left(1+\sqrt{\gamma} \right)^{4/3}}{\left(\sqrt{\gamma}\right)^{1/3}}+M\right)\frac{k}{n^{1/3}}+\frac{\left(1+\sqrt{\gamma} \right)^2}{2\sqrt{\gamma}}\frac{\left(k+1\right)^2}{n^{2/3}}\\
    &-\left(1+\sqrt{\gamma} \right)\frac{\sqrt{2}}{\sqrt{\beta}n^{1/6}}\sum\limits_{j=2}^{k+1}{\zeta}_j-\left(1+\sqrt{\gamma} \right)\frac{\sqrt{2}}{\sqrt{\beta}n^{1/6}}\sum\limits_{j=1}^{k}{\zeta_j'}-\frac{\left(\zeta_{k+2}-\zeta_2 \right)\frac{1}{\sqrt{2}}+\left(\sqrt{2}+\frac{\sqrt{\gamma}}{\sqrt{2}}\right)\left( \zeta_{k+1}'-\zeta_1'\right) }{\sqrt{\beta}n^{1/6}}.
\end{align*}
As the step sizes are not independent, we compare the random walk $\widehat{W}\left(k\right)$ with $\widehat{U}\left(k\right)$ defined by $\widehat{U}\left(1\right)=5$ and for $1\leq k\leq p-1$
\begin{align*}
    &\widehat{U}\left(k+1\right)-5=-\left(\frac{t(1-\delta)\left(1+\sqrt{\gamma} \right)^{4/3}}{\left(\sqrt{\gamma}\right)^{1/3}}+M\right)\frac{k}{n^{1/3}}+\frac{\left(1+\sqrt{\gamma} \right)^2}{2 \sqrt{\gamma}}\frac{\left(k\right)^2}{n^{2/3}}\\
    &-\left( 1+\sqrt{\gamma}\right)\frac{\sqrt{2}}{\sqrt{\beta}n^{1/6}}\sum\limits_{j=2}^{k+1}{\zeta}_j-\left(1+\sqrt{\gamma} \right)\frac{\sqrt{2}}{\sqrt{\beta}n^{1/6}}\sum\limits_{j=1}^{k}{\zeta_j'}.
\end{align*}
We can embed the sums of Gaussians naturally in the Brownian motion $B\left[0,p/n^{1/3}\right]$ as done before. By comparing the difference between $\widehat{U}\left(k\right)$ and $\widehat{W}\left(k\right)$, we can check that the event $\mathcal{A}_1$ occurs if the following event occurs for $0\leq s\leq p/n^{1/3}$:
\begin{align*}
    -3\leq -\left(\frac{t(1-\delta)\left(1+\sqrt{\gamma} \right)^{4/3}}{\left( \sqrt{\gamma}\right)^{1/3}}+M\right)s+\frac{\left( 1+\sqrt{\gamma}\right)^2}{2 \sqrt{\gamma}}(s)^2-\frac{2 \left(1+\sqrt{\gamma} \right)}{\sqrt{\beta}}B\left(s\right)\leq 0.5.
\end{align*}
This is same as 
\begin{align*}
    -\frac{0.5\sqrt{\beta}}{\alpha}\leq B\left(s\right)-g\left(s\right)\leq \frac{3\sqrt{\beta}}{\alpha} \quad \mbox{ for } 0\leq s\leq p/n^{1/3},
\end{align*}

where $g\left(s\right)=\frac{\sqrt{\beta}}{\alpha}\left(s\left(\frac{t(1-\delta)\left(1+\sqrt{\gamma}\right)^{4/3}}{\left(\sqrt{\gamma}\right)^{1/3}}+M\right)-\frac{\left(1+\sqrt{\gamma}\right)^2}{2\sqrt{\gamma}}s^2\right)$ and $\alpha=2\left(1+\sqrt{\gamma} \right).$
We calculate 
\begin{align*}
&\int_{0}^{\frac{p}{n^{1/3}}} \left(g'(s)\right)^2 ds
\end{align*}
and by the Cameron-Martin formula, for sufficiently large $t$ depending on $M$ and $\varepsilon$ and $\delta$
\begin{align*}
&\P\left(\mathcal{A}_1\right)\geq  e^{-\left(\frac{\beta}{24}+\frac \varepsilon2\right)t^3}.
\end{align*}
Hence we have 
\[
\P\left(u(1)<u(2)<\dots<u(p+1)\right) \geq  \P\left (\mathcal{A}_1 \cap \mathcal{E}_1\right) \geq e^{-\left(\frac{\beta}{24}+\varepsilon\right)t^3}-ne^{-\frac{n^{1/3}-2\delta}{2}}.
\]
Once we get the lower bound for the probability of the event $\{u(1) < u(2) < \dots < u(p+1) \},$ to apply Proposition \ref{prop: Tridiag resultwith1} we will consider the lower $(n-p) \times (n-p)$ sub-matrix of $L.$ First recall that
\[
\wih{L}_{n,m,\beta,p}=L'+L'',
\]
where $L'$ has Gaussian entries and $L=\frac{1}{\beta n}L'$. Hence, $L_{p,p}=\wih{C_p}^2+\wih{D_p}^2,$ where\\ $\wih{C_p}^2=\frac{1}{\beta n}\l(\beta(n-p)+\sqrt{2\beta(n-p)}\zeta_p' \r)$ and $\wih{D_p}^2=\frac{1}{\beta n}\l(\beta(m+1-p)+\sqrt{2\beta(m+1-p)}\zeta_p \r)$. We consider a slightly modified version of the matrix $L$ whose $(p,p+1)$ and $(p+1,p)$ entry is replaced by $\wih{C}_pD_{p+1}$. We will show for this matrix the desired lower bound holds. Then by applying Gershgorin's circle theorem and Weyl's inequality as we have done before, we conclude the lower bound for $L$. By abuse of notation we will call this modified matrix also $L$. We consider $L_{\text{low}_{n-p}}$ and to apply Proposition \ref{prop: Tridiag resultwith1} we add $L_{p+1,p}=\frac{\wih{C_p}D_{p+1}}{\beta n}$ to the first entry $\l(\frac{C_{p+1}^2+D_{p+1}^2}{\beta n}\r)$ 
% This entry is $\frac{C_{p+1}^2+D_{p+1}^2}{\beta n}.$ To this entry we now add $L_{p+1,p}=\frac{\wih{C_p}D_{p+1}}{\beta n}$
(we recall here that $C_i$ and $D_i$ are the $\chi$ random variables as defined in the definition of Laguerre-$\beta$ ensemble in \eqref{eq: tridiag matrix}). As consistent with the notation of Proposition \ref{prop: Tridiag resultwith1}, denote this lower submatrix with the added entry as $\overline{L}_{low_{n-p}}$. To get an estimate for the event $\left \{\lambda_{\max}\l(\overline{L}_{low_{n-p}} \r)< \lambda \right \}$, our goal is to use \cite[Theorem 1.1]{BV12} about spiked Laguerre random matrices on $\overline{L}_{low_{n-p}}$. Note that the first entry of $\overline{L}_{low_{n-p}}$ is $\frac{D_{p+1}^2}{\beta n} \l(1+\frac{\wih{C_{p}}}{D_{p+1}} \r)+\frac{C_{p+1}^2}{\beta n}$. To apply \cite[Theorem 1.1]{BV12} we consider the following matrix. Let $\widetilde{L}_{low_{n-p}}$ be the same matrix as $\overline{L}_{low_{n-p}}$ except the $(1,1)$ entry is replaced by $\frac{D_{p+1}^2}{\beta n}\left(1+\sqrt{\frac{n-p}{m-p}} \right)+\frac{C_{p+1}^2}{\beta n}$. We see that 
\[
\left\{\frac{\wih{C_p}}{D_{p+1}} \leq \sqrt{\frac{n-p}{m-p}}\right\} \cap \left \{\lambda_{\max} \left(\widetilde{L}_{low_{n-p}} \right) < \lambda \right\} \subset \left \{\lambda_{\max} \left( \overline{L}_{low_{n-p}}\right) <\lambda \right\}. 
\]
We see that $\wih{C}_p $ and $D_{p+1}$ are independent and 
\[
\P \left(\frac{\wih{C_p}}{\sqrt{\beta n}} \leq \sqrt{\frac{\beta(n-p)}{\beta n}} \right) \geq \frac 12.
\]

\[
\P \left(\frac{D_{p+1}}{\sqrt{\beta n}} \geq \sqrt{\frac{\beta(m-p)}{\beta n}} \right) \geq c,
\]
for some $c>0.$
So, we have that 
\[
\P \left( \frac{\wih{C_p}}{D_{p+1}} \leq \sqrt{\frac{n-p}{m-p}}\right) \geq \frac c2.
\]
Now, we look at the event $\left \{ \lambda_{\max} \left(\widetilde{L}_{low_{n-p}}\right) < \lambda \right\}.$ By \cite[Theorem 1.1]{BV12} (we take $\ell=1+\sqrt{\frac{n-p}{m-p}}$ in \cite[Theorem 1.1]{BV12}) we get that 
\[
\gamma_{m,n,p}\l(\lambda_{\max} \l(n\widetilde{L}_{low_{n-p}} \r)-\l(\sqrt{m-p}+\sqrt{n-p} \r)^2\r) \overset{d}{\rightarrow} \Lambda,
\]
where $\gamma_{m,n,p}=\frac{\l(\sqrt{(m-p)(n-p)} \r)^{1/3}}{\l(\sqrt{m-p}+\sqrt{n-p}\r)^{4/3}}$ and $\Lambda$ is a random variable supported on the whole real line.
Hence, we can choose $M$ large enough so that 
\begin{align*}
&\P \left( \lambda_{\max} \left(n\widetilde{L}_{low_{n-p}} \right) \leq \left( \sqrt{m-p}+\sqrt{n-p}\right)^2+M(n-p)^{1/3}\frac{\left(1+\sqrt{\frac{m-p}{n-p}}\right)^{4/3}}{\left( \sqrt{\frac{m-p}{n-p}}\right)^{1/3}}\right) > 1-\frac c2\\
% & \textcolor{blue}{\implies \P \left(\lambda_{\max} \left(B\right) \leq n\left( \sqrt{\gamma}\sqrt{1-\frac pm}+\sqrt{1-\frac pn}\right)^2+n^{1/3}M\frac{\left(1+\sqrt{\frac{m-p}{n-p}}\right)^{4/3}}{\left( \sqrt{\frac{m-p}{n-p}}\right)^{1/3}} \right)> 1-\frac c2}\\ 
 & \implies \P \left(\lambda_{\max} \left(n \widetilde{L}_{low_{n-p}} \right) \leq n\left( \sqrt{\gamma}\left(1-\frac{p}{2m} \right)+1-\frac {p}{2n}\right)^2+n^{1/3}M\frac{\left(1+\sqrt{\frac{m-p}{n-p}}\right)^{4/3}}{\left( \sqrt{\frac{m-p}{n-p}}\right)^{1/3}} \right)> 1-\frac c2\\
% & \textcolor{blue}{\implies \P \left(\lambda_{\max} \left(B \right) \leq n\left( \sqrt{\gamma}+1\right)^2-p\frac{\left(1+\sqrt{\gamma}\right)^2}{\sqrt{\gamma}}+p^2n \left(\frac{\sqrt{\gamma}}{2m} +\frac{1}{2n}\right)^2+n^{1/3}M\frac{\left(1+\sqrt{\frac{m-p}{n-p}}\right)^{4/3}}{\left( \sqrt{\frac{m-p}{n-p}}\right)^{1/3}} \right)}\\
% &\textcolor{blue}{> 1-\frac c2}\\
\end{align*}
Hence, we get
\begin{align*}
&\P \left(\lambda_{\max} \left(n\widetilde{L}_{low_{n-p}} \right) \leq n\left(1+\sqrt{\gamma}\right)^2-pn\frac{(1+\sqrt{\gamma})^2}{\sqrt{\gamma}} +p^2n \left(\frac{\sqrt{\gamma}}{2m} +\frac{1}{2n}\right)^2+n^{1/3}M\frac{\left(1+\sqrt{\frac{m-p}{n-p}}\right)^{4/3}}{\left( \sqrt{\frac{m-p}{n-p}}\right)^{1/3}} \right)\\
& >1-\frac c2.
%& \implies \P \left(\lambda_{\max} \left(\widetilde{W}\right) 
% \leq n \left(1+\sqrt{\frac{m-p}{n-p}} \right)^2-\frac{t\left(1+\sqrt{\frac{m-p}{n-p}} \right)^{4/3}}{\left(\sqrt{\frac{m-p}{n-p}}\right)^{1/3}}n^{1/3}+\frac{M \left(1+\sqrt{\frac{m-p}{n-p}} \right)^{4/3}n^{1/3}}{\left(\sqrt{\frac{m-p}{n-p}}\right)^{1/3}}\right) > \frac 34.
\end{align*}
%where $a_{n,t,\gamma}=n\left( \sqrt{\gamma}+1\right)^2-tn^{1/3}\frac{\left(1+\sqrt{\gamma}\right)^{4/3}}{\left(\sqrt{\gamma}\right)^{1/3}}.$
So, by plugging in the value of $p$ we get that we can chose $t,n$ large enough depending on $M$ such that 
\[\P \left(\lambda_{\max}\left(n\widetilde{L}_{low_{n-p}}\right) < n \left(1+\sqrt{\gamma} \right)^2-t\left(1-\delta \right)n^{1/3}\frac{\left(1+\sqrt{\gamma} \right)^{4/3}}{\left(\sqrt{\gamma} \right)^{1/3}} \right) > 1-\frac c2. 
\] 
\begin{comment}Hence, 
\begin{align*}
&\P \left( \left \{\frac{\sqrt{\beta(n-p)+\sqrt{2\beta(n-p)\zeta_p'}}}{Z_{p+1}} \leq \sqrt{\frac{m-p}{n-p}}\right\}\\
&\cap \left \{\lambda_{\max}\left(A'_{low}\right) < n \left(1+\sqrt{\gamma} \right)^2-t\left(1-\delta \right)n^{1/3}\frac{\left(1+\sqrt{\gamma} \right)^{4/3}}{\left(\sqrt{\gamma} \right)^{1/3}} \right\}\right)>0
\end{align*}
\end{comment}
Let the above probability be $c$ for some $c>0.$ Once we have this, we need to combine the two events $\{ u(1) < u(2) < \dots <u(p+1) \}$ and the event that $\{\lambda_{\max}\l( \overline{L}_{low_{n-p}}\r) < \lambda \}.$ Note that as both the events involve $\wih{C}_p,$ they are not independent. So, we do the following. We define the following events. Let $\widetilde{L}_{top}$ be the upper $p \times p$ submatrix of $L$ with the $(p,p)$ entry is replaced by $\frac{\wih{D}_{p}^2}{\beta n}+\frac{(m-p)}{\beta n}$ and $\widetilde{u}(i)$ are the real numbers defined by the similar recursion relations as we did before defined by this matrix. Note that $u(i)=\widetilde{u}(i)$ for all $1 \leq i \leq p$ and $u(p+1)>\widetilde{u}(p+1).$ Let us now define the following events.
\begin{itemize}
    \item$\mathcal{A}:=\left \{u(1)<u(2)<\cdots u(p) < u(p+1)\right\}$.
    \item $\widetilde{\mathcal{A}}:=\left \{\widetilde{u}(1) < \widetilde{u}(2)< \cdots <\widetilde{u}(p+1) \right\}$.
    \item $\mathcal{C}:=\left\{ \frac{\wih{C}_p}{\sqrt{\beta n}} \leq \sqrt{\frac{\beta(n-p)}{\beta n}}\right\}.$
    \item $\mathcal{D}:=\left \{\frac{D_{p+1}}{\sqrt{\beta n}} \geq \sqrt{\frac{\beta(m-p)}{\beta n}}\right\}.$
    \item $\widetilde{\mathcal{M}}:=\left \{\lambda_{\max}\left(\widetilde{L}_{low_{n-p}}\right) < n \left(1+\sqrt{\gamma} \right)^2-tn^{1/3}\left(1-\delta \right)\frac{\left(1+\sqrt{\gamma} \right)^{4/3}}{\left(\sqrt{\gamma} \right)^{1/3}} \right\}.$
    \item $\overline{\mathcal{M}}:=\left \{\lambda_{\max}\left(\overline{L}_{low_{n-p}}\right) < n \left(1+\sqrt{\gamma} \right)^2-tn^{1/3}\left(1-\delta \right)\frac{\left(1+\sqrt{\gamma} \right)^{4/3}}{\left(\sqrt{\gamma} \right)^{1/3}} \right\}.$
\end{itemize}
We observe the following. 
\[
\widetilde{\mathcal{A}} \cap \mathcal{C} \cap \mathcal{D} \cap \widetilde{\mathcal{M}} \subset \mathcal{A} \cap \overline{\mathcal{M}}.
\]
Now, we see that $\widetilde{\mathcal{A}}$ is independent of $\mathcal{C}, \mathcal{D}$ and $\widetilde{\mathcal{M}}$. This is because, $\widetilde{\mathcal{A}}$ does not depend on $\wih{C}_p$ or $D_{p+1}$, as on the event $\widetilde{\ca} \cap \cc$ we have replaced the $\wih{C}_p$ entry by a larger constant . But now using this independence we get 
\[
\P \left(\mathcal{A} \cap \overline \cm \right) \geq \P \left(\mathcal{C} \cap \mathcal{D} \cap \wt \cm \right) \P \left( \wt \ca\right) \geq c\P(\wt {\mathcal{A}}).
\]
Now, observe that $\widetilde{L}_{top}$ is same as the upper submatrix of $L$ except the $(p,p)$ entry. So, all the recursion calculations remain valid for this matrix as well. Hence, we can conclude that 
\[
\P(\widetilde{\mathcal{A}}) \geq e^{-\left(\frac{\beta}{24}+\varepsilon\right)t^3}.
\]
Combining all the above and by the Proposition \ref{prop: Tridiag resultwith1} we get that 
\[
\P \left(\lambda_{\max}\left(L'\right) \leq n\left(1+\sqrt{\gamma} \right)^2-\frac{n^{1/3}t(1-\delta)\left(1+\sqrt{\gamma} \right)^{4/3}}{\left(\sqrt{\gamma} \right)^{1/3}}\right) \geq e^{-\left(\frac{\beta}{24}+\frac \varepsilon2 \right)t^3}.
\]
Now, choosing $\delta$ appropriately depending on $\varepsilon$, applying Gershgorin's circle theorem and Weyl's inequality, Proposition \ref{eq: estimate of off diagonals of the difference} to compare $\lambda_{\max}\left(L_{n,m,\beta} \right)$ and $\lambda_{\max}\left(\widehat{L}_{n,m,\beta,p} \right)$ as before will give us the desired lower bound of $\lambda_{\max}\l(L_{n,m,\beta} \r)$. 

\section{Lim sup for the LILs}
\label{sec: limsup}
\begin{comment}
\subsection{Lim Sup Upper Bound (Point to Point)}
\begin{theorem}
\[
\limsup_{n \rightarrow \infty}\frac{T_n-4n}{2^{4/3}n^{1/3}(\log \log n)^{2/3}} \leq \left (\frac34 \right)^{2/3} \text{ a.s. }
\]
\end{theorem}
\end{comment}
We will prove the law of iterated logarithm results in the following sections. Before starting with the proofs we collect some notations that will be used throughout the sections.

\paragraph{\textbf{Notations:}} It is sometimes convenient to work with rotated axes in $\Z^2, x+y=0$ and $x-y=0.$ They are known as \textit{space axis} and \textit{time axis } respectively. For $u \in \Z^2, $ we will use $\phi(u)$ and $\psi(u)$ to denote the time and space coordinates respectively. In particular, if $u=(u_1,u_2)$
\[
\phi(u)=u_1+u_2,\qquad  \psi(u)=u_1-u_2.
\]
Recall that  $\bo{r}$ denotes the vertex $(r,r)$ and $\cl_r$ denotes the line $x+y=r$.  $T_n^{\ptop{}},T_n^{\ptohalfspace{}}$ denote the passage time between $\bo{0}$ and $\bo{n}$ in the exponential LPP and the half space model respectively. $T_n^{\ptoline{}}$(resp.\ $T_n^{\linetop{}}$) will denote the passage time between $\bo{0}$ and $\cl_{2n}$ (resp.\ $\cl_0$ to $\bo n$) in the exponential LPP. $\Gamma_n^{\ptoline{}},\Gamma_n^{\linetop{}}, \Gamma_n^{\ptoline{}}, \Gamma_n^{\ptohalfspace{}}$ are the corresponding point-to-point, line-to-point, point-to-line, point-to-point in half space geodesics respectively. If $\Gamma$ is a geodesic, let $\Gamma(r)$  denote the random vertex where $\Gamma$ intersects $\cl_r.$ For $u,v \in \Z^2$ with $u \leq v, T_{u,v}$ without any superscript will always denote the point-to-point passage time in the exponential LPP model (i.e., when we have i.i.d. Exp(1) random variables on each vertices of $\Z^2$). Similarly, for $u \in \Z^2,$ with $\phi(u) \leq r$ (resp.\ $\phi(u) \geq r$) $T_{u, \cl_r}$ (resp.\ $T_{\cl_r,u}$) without any superscript will denote the line-to-point passage time between $u$ and $\cl_r$ (resp.\ point-to-line passage time between $\cl_r$ and $u$) in the exponential LPP. For $u, v \in \Z^2,$ with $u \leq v, T^{\ptohalfspace{}}_{u,v}$ will denote the last passage time between $u$ and $v$ in the half space model.
\begin{remark}
    Throughout the proofs of this section, in the definition of $T_n^{\ptop{}},T^{\ptoline{}}_{n}, T_n^{\ptohalfspace{}}$ we will usually remove initial vertex and keep the final vertex. In $T^{\linetop{}}_{n}$ we will usually exclude the final vertex. Note that Theorem \ref{t: passage time tail estimates} holds for these variants and  proving Theorem \ref{t: LILs} for these versions will suffice. We will mention explicitly in the corresponding proofs which version of last passage time we are using.
\end{remark}
\subsection{Upper Bounds} We will prove the upper bounds for $\limsup$ for the point-to-line and line-to-point cases in Theorem \ref{t: LILs}. Upper bounds for the $\limsup$ for point-to-point case and the half space case will follow from the $\limsup$ upper bound for the point-to-line case and ordering \eqref{eq: ordering}. Also, note that $\limsup$ upper bound for the point-to-point case was already proved in \cite[Theorem 1]{L18}.\\
%\begin{proof}
 %   This follows from Theorem \ref{Lim Sup Upper Bound (Point to Line)} below and the fact that $T_n \leq \wt{T}_n$.
%\end{proof}
%\begin{remark}
%    The above upper bound was already proved in \cite[Theorem 1]{L18}.
%\end{remark}
%\subsection{Lim Sup Upper Bound (Line to Point)}
%\begin{theorem}
%\label{Lim_sup_upper_bound_line_to_point}
%\begin{equation}
%\label{Lim_sup_upper_bound_line_to_point_equation}
 %   \limsup_{n \rightarrow \infty}\frac{\wih{T}_n-4n}{2^{4/3}n^{1/3}(\log \log n)^{2/3}} \leq \left (\frac34 \right)^{2/3} \text{ a.s. }
    %\end{equation}
%\end{theorem}
\begin{comment}
\begin{remark}
    Note that throughout the proofs of this section, in $T^{\linetop{}}_{n}$ and $T^{\ptoline{}}_{n}$ we will exclude the final vertex and first vertex respectively. Note that proving Theorem \ref{t: ptol lil} for these versions will suffice. 
\end{remark}
\end{comment}
To prove the $\limsup$ upper bounds for point-to-line and line-to-point case we require the following lemma about point-to-line and line-to-point passage times.
\begin{lemma}
\label{maximal_inequality for line to point}
Let $T_n^*$ denote $T_n^{\ptoline{}}$ or $T_n^{\linetop{}}$. There exists $c>0$ such that for any natural numbers $k$ and $\ell$ with $k \leq \ell$ and for any real numbers $t$ we have
    \[
    \P \left( \max_{k \leq n \leq \ell} (T^*_{n}-4n) \geq t \right) \leq \frac 1c \P \left( T^*_{\ell} -4 \ell \geq t\right).
    \]
\end{lemma}
\begin{proof} In this proof, for $T_{n}^{\linetop{}}$ and $ T_{u,v}$ we exclude the final vertex but keep the initial vertex and for $T_{n}^{\ptoline{}}$ we will remove the initial vertex and keep the final vertex. The proof follows the same idea as the proof of \cite[Lemma 3]{L18}. Precisely, we do the following. For both the passage times, for $k \leq n \leq \ell,$ we define the following events for line-to-point and point-to-line case.
    \[
    {\ca}^*_n:=\bigcap_{i=k}^{n-1}\Big \{T^*_{i}-4i < t \Big \} \bigcap \Big \{ T^*_{n}-4n \geq t \Big \}. 
    \]
    Clearly, we have for both the cases
    \[
     \Big \{ \max_{k \leq n \leq \ell} (T^*_{n}-4n) \geq t \Big \}=\bigsqcup_{n=k}^{\ell} \ca_n^*.
    \]
    We further define the following events for $k \leq n \leq \ell$.
\begin{itemize}
    \item $\cb_n^{\linetop{}}:=\Big \{ T_{\bo{n},\bo{\ell}}-4(\ell-n) \geq 0\Big \}$.
    \item $\cb_n^{\ptoline{}}:= \left \{ T_{\Gamma_n^{\ptoline{}}(2n), \cl_{2\ell}}-4(\ell-n) \geq 0\right\}.$
\end{itemize}
We use the common notation $\cb_n^*$ for $\cb_n^{\linetop{}}$(line-to-point case) and $\cb_n^{\ptoline{}}$ (point-to-line case) . Then clearly for both the cases we have (for $\ca_n^*$ corresponding $T_n^*$ and $\cb_n^*$ are taken)
\[
\bigsqcup_{n=k}^{\ell} \left( \ca_n^* \cap \cb_n^* \right) \subset \Big \{ T_\ell^* -4 \ell \geq t \Big \}.
\]
Now, for both the cases we observe that by the definition of the passage times $\ca_n^*$ and $\cb_n^*$ are independent.
\begin{comment}\paragraph{\textit{Proof for the line-to-point case:}}
Clearly, for each $n$ we have $\ca_n$ and $\cb_n^{\linetop{}}$ are independent (as the final vertices are not included in the definition of the passage times). Further,
\[
\bigsqcup_{n=k}^{\ell} \left( \ca_n \cap \cb_n^{\linetop{}} \right) \subset \Big \{ T_\ell^{\linetop{}} -4 \ell \geq t \Big \}.
\]
\end{comment}
Hence, we have 
\[
% \min_{k \leq n \leq \ell} \P \left(T_{\bo{n},\bo{\ell}}-4(\ell-n) \geq 0 \right) \sum_{n=k}^{\ell}\P(\ca_n^*) \leq \P \left(T_\ell^* -4 \ell \geq t \right).
\min_{k \leq n \leq \ell} \P \left(\cb_n^* \right) \sum_{n=k}^{\ell}\P(\ca_n^*) \leq \P \left(T_\ell^* -4 \ell \geq t \right).
\]

Also, as we have removed the initial vertex from the point-to-line passage time definition and using the fact that $T_{v, \cl_{2n}}$ has same distribution as $T^{\ptoline{}}_{n}$ for all $v \in \cl_0,$ we have 
\[
\P \left(\cb_n^{\ptoline{}} \right)=\P \left(T^{\ptoline{}}_{\ell-n}-4(\ell-n) \geq 0 \right).
\]
Now, using the convergence in distribution of $\frac{T_n^{\ptop{}}-4n}{n^{1/3}}, \frac{T_n^{\ptoline{}}-4n}{n^{1/3}} $ we can say that there exists $c>0$ such that for all $k, \ell$, 
\[\min_{k \leq n \leq \ell} \P \left( \cb_n^*\right) \geq c.
\]
This completes the proof for both cases.
    \end{proof}

\begin{proof}[Proof of $\limsup$ Upper Bounds] First we prove upper bound in line-to-point case. The proof follows similar ideas as in the proof of \cite[Theorem 1]{L18} upper bound. Recall that in this case while considering line-to-point passage time we exclude the final vertex and keep the initial vertex. Let us consider $n_k=[\rho^k]$ for some $\rho>1$ (to be chosen later).
For $k \geq 1$ and for some $\delta>0$ We define the following events. 
\[
\cc_k^{\delta}:=\Big\{\max_{n_{k-1}<n \leq n_k}\frac{T_n^{\linetop{}}-4n}{g_+(n)} \geq \left (\frac34 \right)^{2/3}+\delta \Big \}.
\]
We will show that for all $\delta>0,$
\begin{equation}
\label{series_converges}
\sum_{k=1}^{\infty}\P(\cc^{\delta}_k) < \infty.
\end{equation}
Then by Borel-Cantelli lemma the required upper bound follows.
We have 
\[
\P(\cc^{\delta}_k) \leq \P \left ( \max_{n_{k-1}<n \leq n_k}T^{\linetop{}}_n-4n \geq \left(\left (\frac34 \right)^{2/3}+\delta \right)(g_+(n_{k-1}))\right).
\]
Now, using Lemma \ref{maximal_inequality for line to point}, we have the following.  
\begin{equation}
\label{upper_bound_to_apply_BC}
 \P(\cc_k^{\delta})\leq \frac 1c \P\left( T_{n_k}^{\linetop{}}-4n_k \geq \left(\left (\frac34 \right)^{2/3}+\delta\right)(g_+(n_{k-1}))\right),
\end{equation}
for some $c>0.$
\begin{comment}
Now, if we chose $s=M2^{4/3}(n_k-n)^{1/3}$ for some large constant $M$, then we have that the denominator have a uniform lower bound by Theorem \ref{t: passage time tail estimates}(ii).
\end{comment}
\begin{comment}
Hence, we have, there exists $C>0$ such that 
\begin{equation}
\label{upper_bound_to_apply_BC}
\P(\cc_k^{\delta}) \leq \frac 1c \P\left( T_{n_k}^{\linetop{}}-4n_k \geq \left(\left (\frac34 \right)^{2/3}+\delta\right)(g_+(n_{k-1}))\right).
\end{equation}
\end{comment}
Now, we chose $\rho$ close enough to 1 so that there exists $0< \delta'< \delta$ such that for all $k$ large enough we have 
\[
\left (\left (\frac34 \right)^{2/3}+\delta) \right)g_+(n_{k-1}) \geq \left (\left (\frac34 \right)^{2/3}+\delta' \right ) g_+(n_k).
\]
So, by \eqref{upper_bound_to_apply_BC} and Theorem \ref{t: passage time tail estimates}, (i) for any $\vep>0$,
\[
\sum_{k=1}^{\infty}\P(\cc_k^{\delta}) \leq \frac 1c \sum_{k=1}^{\infty}\P \left( T_{n_k}^{\linetop{}}-4n_k \geq \left(\left (\frac34 \right)^{2/3}+\delta'\right)g_+(n_k)\right) \leq \frac 1c\sum_{k=1}^{\infty}e^{-\left(\frac43-\vep\right) \left( \left(\frac34 \right)^{2/3}+\delta'\right)^{3/2} \log \log n_k}.
\]
Note that the right hand side defines a convergent series if we chose $\vep$ small enough. Hence, we conclude \eqref{series_converges}.\\
Proof of the $\limsup$ upper bound for the point-to-line case is exactly same as proof of line-to-point case. We just apply Lemma \ref{maximal_inequality for line to point} for the point-to-line case. Rest of the argument is essentially same. 
\end{proof}
\subsection{Lower Bounds} We will prove only the $\limsup$ lower bound for the half space case. Rest of the $\limsup$ lower bounds will follow from the half space case and ordering \eqref{eq: ordering}.  Note that assuming the constant $\left( \frac 43 \right)$ in the exponent in the lower bound for the right tail, the $\limsup$ lower bound for the point-to-point case was proved in \cite[Theorem 1]{L18}. In this article we also prove the $\left( \frac 43 \right)$ exponent in Theorem \ref{t: passage time tail estimates}, (i).
    \begin{proof}[Proof of the $\limsup$ Lower Bounds] Here in the definition of passage time in the half space model we exclude the initial vertex and keep the final vertex. The proof follows same idea as in Theorem \cite[Theorem 1]{L18}. We fix $\vep>0$ and let $n_k:=\rho^k$, for some $\rho>1$ large enough (to be chosen later depending on $\vep$). Note that, it is enough to prove for any $\vep>0,$
    \[
    \P \left( \frac{T^{\ptohalfspace{}}_{n_k}-4n_k}{g_+(n_k)} \geq \left( \frac 34 \right)^{2/3}-\vep \text{ i.o. } \right)=1.
    \]
    We do it in the following way. First we observe that for any $\delta>0$ there exists sufficiently large $k$ such that by Theorem \ref{t: passage time tail estimates}, (ii)
    \[
    \sum_{k=1}^{\infty}\P\left( \frac{T^{\ptohalfspace{}}_{n_{k-1}}-4n_{k-1}}{g_+(n_{k-1})} \leq - \frac \vep2 \right) \leq \sum_{k=1}^{\infty} e^{-\left(\frac {1}{24}-\delta\right) \left( \frac \vep2 \right)^3 (\log \log n_{k-1})^2}.
    \]
    Clearly, the right hand side series converges. Therefore by Borel-Cantelli lemma we have a.s. for sufficiently large $k,$
    \[
    \frac{T^{\ptohalfspace{}}_{n_{k-1}}-4n_{k-1}}{g_+(n_{k-1})} \geq - \frac \vep2. 
    \]
    %We define 
    %\[
    %T_{n_{k-1},n_k}^{HS}:=\max_{\gamma}\ell(\gamma),
    %\]
    %where the maximum is taken over all up-right paths starting from $\bo{n_{k-1}}$ and ending at $\bo{n_k}$ and they stay above the diagonal.
    Observe that 
    \begin{equation}
    \label{eq: subadditivity for half space}
    T^{\ptohalfspace{}}_{n_k} \geq T^{\ptohalfspace{}}_{n_{k-1}}+T_{\bo{n_{k-1}},\bo{n_k}}^{\ptohalfspace{}}.
    \end{equation}
    Now, consider the following sum. By Theorem \ref{t: passage time tail estimates}, (i) we have 
    \begin{equation}
    \label{eq: BC sum for HS lower bound}
    \sum_{k=1}^{\infty}\P \left( \frac{T_{\bo{n_{k-1}},\bo{n_k}}^{\ptohalfspace{}}-4(n_k-n_{k-1})}{g_+(n_k)} \geq \left( \frac 34\right)^{2/3}-\frac \vep2\right) \geq \sum_{k=1}^{\infty}e^{-\frac 43\left(1+\frac{\vep'}{2}\right)\left(\left( \frac 34\right)^{2/3}-\frac \vep2 \right)^{3/2}\frac{\sqrt{n_k}}{\sqrt{n_{k}-n_{k-1}}} \log \log n_k}.
    \end{equation}
    Now, we choose $\rho>1$ large enough so that
    \[
    \frac{\sqrt{n_k}}{\sqrt{({n_k-n_{k-1})}}}< 1+\frac{\vep'}{2},
    \]
    where we can choose $\vep'$ so that
    \[
    \frac 43 \left(1+\frac{\vep'}{2} \right)^2 < \frac{1}{\left(\left( \frac 34\right)^{2/3} -\frac \vep2\right)^{2/3}}.
    \]
    Note that in \eqref{eq: BC sum for HS lower bound}, the events in the left side are independent (this is because we are considering the passage time without the initial vertex) and the sum in the right hand side diverge by the choice of $\rho$ and $\vep'$. Hence, using Borel-Cantelli lemma we have 
    \[
    \P \left( \frac{T_{\bo{n_{k-1}},\bo{n_k}}^{\ptohalfspace{}}-4(n_k-n_{k-1})}{g_+(n_k)} \geq \left( \frac 34\right)^{2/3}-\frac \vep2 \text{ i.o. }\right)=1.
    \]
    Finally, we get two full probability set $\Omega_1$ and $\Omega_2$ as follows.
    \[
    \Omega_1:=\Big\{ \frac{T^{\ptohalfspace{}}_{n_{k-1}}-4n_{k-1}}{g_+(n_{k-1})} \leq - \frac \vep2 \text{ i.o. }\Big \}^c,
    \]
    \[
    \Omega_2:= \Big\{ \frac{T_{\bo{n_{k-1}},\bo{n_k}}^{\ptohalfspace{}}-4(n_k-n_{k-1})}{g_+(n_k)} \geq \left( \frac 34\right)^{2/3}-\frac \vep2 \text{ i.o. } \Big\}.
    \]
    Observe that on $\Omega_1 \cap \Omega_2$ by \eqref{eq: subadditivity for half space}, we have infinitely often,
    \[
    T^{\ptohalfspace{}}_{n_k}\geq 4n_k+\left( \left(\frac 34 \right)^{2/3}-\vep \right)g_+(n_k).
    \]
       This completes the proof.
    \end{proof}
    
\section{Lim inf for the LILs}
\label{sec: liminf}
\subsection{Lower Bounds} We will prove the $\liminf$ lower bounds in Theorems \ref{t: LILs} in this section. Note that in \cite[Theorem 2]{L18} a lower bound for the $\liminf$ for the point-to-point case was proved, but the constant obtained there is not optimal. We will use Lemma \ref{lemma: infimum over a line for half space} instead of the union bound argument used in \cite[Theorem 2]{L18} to obtain the optimal lower bound in all the four cases.
\paragraph{\textbf{Proof of liminf Lower Bounds:}}
\begin{comment}
We fix $\rho>1$ (to be chosen later) and consider $n_k=[\rho^k]$ for all $k \geq 0.$ We fix any $\delta>0$.
\end{comment}
In this proof, $T_n^*$ will denote any of the passage times $T_n^{\ptop{}},T_n^{\linetop{}}, T_n^{\ptoline{}}$,\newline $T_n^{\ptohalfspace{}}$. In  the three cases other than the line-to-point case, we exclude the weight of initial vertex in the definitions of respective passage times. In the definition of the line-to-point passage time, we remove the weight of the final vertex. We fix $\delta>0$. We choose $\varepsilon>0$ small enough (to be specified later). 
%such that 
%\[
%      \frac{\varepsilon}{1-\varepsilon}< \frac{\left((12)^{1/3}+\frac \delta2 \right)^3}{12}-1.
%\]
We define the sequence $n_k:=[e^{k^{1-\varepsilon}}]$. Consider the following event. 
\begin{displaymath}
    \ci^{\delta}_k:=\Big\{\min_{n_{k-1} < n \leq n_k}\frac{T_n^*-4n}{g_-(n)} \leq -(6\beta)^{1/3}-\delta\Big\},
\end{displaymath}
where $\beta =1,1,2,4$ when $T_n^*=T_n^{\linetop{}},T_n^{\ptoline{}},T_n^{\ptop{}},T_n^{\ptohalfspace{}}$ respectively.
Then by Borel-Cantelli lemma it is enough to show for all $\delta>0$
\begin{equation}
\label{Borel Cantelli for lim inf point to point}
    \sum_{k=0}^{\infty}\mathbb{P}(\ci^{\delta}_k) < \infty.
\end{equation}
It is easy to see that $\ci^{\delta}_k\subseteq \ca_k^* \cup \cb_k^*$ 
%We have 
%begin{align*}
%\P(\ci_k^{\delta}) \leq %\P \left( \frac{T_{n_{k-1}}^*-4n_{k-1}}{g_-(n_{k-1})} \leq -%(6\beta)^{1/3}-\frac \delta2\right)+\P(\cb_k^*)
%\P\left(\ca_k^* \right)+\P(\cb_k^*).
%\P \left( \inf_{n_{k-1}< n \leq n_k} \frac{T_{\bo{n_{k-1}}, \bo{n}}-4(n-n_{k-1})}{g_-(n_{k-1})} \leq -\frac \delta2\right),
%\end{align*}
where in all the four cases 
\[
\ca_k^*:=\left \{ \frac{T_{n_{k-1}}^*-4n_{k-1}}{g_-(n_{k-1})} \leq -(6\beta)^{1/3}-\frac \delta2\right\}.
\]
and the events $\cb_k^*$ are defined as follows in each of the four cases. Let $v_k:=\Gamma_{n_{k-1}}(2n_{k-1}).$
\begin{itemize}
    \item For point-to-point case: $\cb_k^{\ptop{}}:=\left \{ \inf\limits_{n_{k-1}< n \leq n_k} \frac{T_{\bo{n_{k-1}}, \bo{n}}-4(n-n_{k-1})}{g_-(n_{k-1})} \leq -\frac \delta2\right\}.$
    \item For the line-to-point case: $\cb_k^{\linetop{}}:=\left \{\inf\limits_{n_{k-1}< n \leq n_k} \frac{T_{\bo{n_{k-1}}, \bo{n}}-4(n-n_{k-1})}{g_-(n_{k-1})} \leq -\frac \delta2\right\}.$
    \item For the point-to-line case: $\cb_k^{\ptoline{}}:=\left \{\inf\limits_{n_{k-1}< n \leq n_k} \frac{T_{v_k, \cl_{2n}}-4(n-n_{k-1})}{g_-(n_{k-1})} \leq -\frac \delta2\right\}.$
    \item For the half space case: $\cb_{k}^{\ptohalfspace{}}:=\left \{\inf\limits_{n_{k-1}< n \leq n_k} \frac{T^{\ptohalfspace{}}_{\bo{n_{k-1}}, \bo{n}}-4(n-n_{k-1})}{g_-(n_{k-1})} \leq -\frac \delta2 \right\}.$
\end{itemize}
 Note that although $v_k$ is random, 
 \[
 \P \left( \cb_k^{\ptoline{}}\right)=\P \left(\inf_{n_{k-1}< n \leq n_k} \frac{T_{\bo{n_{k-1}},\cl_{2n}}-4(n-n_{k-1})}{g_-(n_{k-1})} \leq -\frac \delta2 \right).
 \]
For the events $\ca_k^*$, by Theorem \ref{t: passage time tail estimates}, (ii), we have for sufficiently large $k$,  
\begin{align}\label{eq: lefttaillbineq1}
    % \P \left( \frac{T_{n_{k-1}}^{\ptop{}}-4n_{k-1}}{g_-(n_{k-1})} \leq -(12)^{1/3}-\frac \delta2\right) \leq e^{-\frac{1}{12}\left(1-\vep\right)\left( (12)^{1/3}+\frac\delta2 \right)^3 \log \log n_{k-1}},\\
    %{\sum\limits_{k=1}^{\infty}
    %\P \left( \frac{T_{n_{k-1}}^*-4n_{k-1}}{g_-(n_{k-1})} \leq -(6\beta)^{1/3}-\frac \delta2\right) \leq 
    %\sum\limits_{k=1}^{\infty}
    \P(\ca_k^*)\le e^{-\frac{1}{6 \beta}\left(1-\vep\right)\left( (6 \beta)^{1/3}+\frac\delta2 \right)^3 \log \log n_{k-1}}. 
\end{align}
Now, we can chose $\vep$ small enough so that when summed over $k$, the resulting series on the right hand series converges in all the four cases.
%where 
%\[
%\delta'=\frac{1}{12}-\frac{1}{\left( (12)^{1/3}+\frac\delta2 \right)^3}.
%\]
\begin{comment}
 Note that, 
 \[
 \P \left( \cb_k^{\ptoline{}}\right)=\P \left(\inf_{n_{k-1}< n \leq n_k} \frac{T_{\bo{n_{k-1}},\cl_{2n}}-4(n-n_{k-1})}{g_-(n_{k-1})} \leq -\frac \delta2 \right)
 \]
 \end{comment}
 For the second summand, by the ordering between passage times as in \eqref{eq: ordering}, we have in all the  four cases
 \[
 \P \left(\cb_k^{*} \right) \leq \P \left(\inf_{n_{k-1}< n \leq n_k} \frac{T_{\bo{n_{k-1}}, \bo{n}}^{\ptohalfspace{}}-4(n-n_{k-1})}{g_-(n_{k-1})} \leq -\frac \delta2 \right).
 \]
 For the right hand side we have the following lemma which we will prove in Section \ref{s: LPP estimates}.
 \begin{lemma}
\label{lemma: infimum over a line for half space} There exists $C,c>0$ such that for all sufficiently large $n$ and $\theta>0$ sufficiently large we have
\[
    \P \left(\inf_{1 \leq m \leq n}\{T_m^{\ptohalfspace{}}-4m\} \leq -\theta n^{1/3} \right) \leq Ce^{-c \theta}.
\]
\end{lemma}
It is natural to expect that the exponent can be improved to $\theta^3$, and indeed such an estimate can be proved for the point-to-point passage times. See Remark~\ref{rem:ptptcubeexponent}. 

 \begin{comment}

 For the second summand we deal with each of the four cases as follows. For $\cb_k^{\ptop{}}$ we have 
 \begin{align*}
     &
     %\P \left( \inf_{n_{k-1}< n \leq n_k} \frac{T_{\bo{n_{k-1}}, \bo{n}}-4(n-n_{k-1})}{g_-(n_{k-1})} \leq -\frac \delta2 \right) \leq \\ 
     \P \left( \cb_k^{\ptop{}}\right) \leq \P \left(  \inf_{n_{k-1}< n \leq n_k} \bigg \{T_{\bo{n_{k-1}}, \bo{n}}-4(n-n_{k-1})\bigg \} \leq -\frac \delta22^{4/3}(n_k-n_{k-1})^{1/3}\frac{n_{k-1}^{1/3} (\log \log n_{k-1})^{1/3}}{(n_k-n_{k-1})^{1/3}} \right). 
 \end{align*}
    
    We use the following lemma to estimate this probability. We will prove it in Section \ref{s: LPP estimates}.
    \begin{lemma}
 \label{infimum over a line}
     There exist $C,c>0$ such that for all sufficiently large $n$ and $\theta>0$ sufficiently large we have 
     \[
     \P \left( \inf_{1 \leq m \leq n} \{ T_m^{\ptop{}}-4m\} \leq -\theta n^{1/3} \right) \leq Ce^{-c\theta^3}.
     \]
 \end{lemma}
 \end{comment}
    By the above lemma we have in all four cases
    \begin{equation}
    \label{convergent series for the second term}
    {\sum\limits_{k=1}^{\infty}}
    %\P \left( \inf_{n_{k-1}< n \leq n_k} \frac{T_{\bo{n_{k-1}}, \bo{n}}-4(n-n_{k-1})}{g_-(n_{k-1})} \leq -\frac \delta2 \right) 
    \P \left( \cb_k^*\right)\leq {\sum\limits_{k=1}^{\infty}}Ce^{-c\frac{\delta}{2}\left(\frac{n_{k-1}}{n_k-n_{k-1}}\right)^{1/3}}.
    \end{equation}
    Further, we have for sufficiently large $k$ there exists a constant $c>0$ such that 
    \begin{equation}
    \label{eq: ratio}
    \frac{e^{(k-1)^{1-\varepsilon}}}{e^{k^{1-\varepsilon}}-e^{(k-1)^{1-\varepsilon}}} \geq ck^{\varepsilon}.
    \end{equation}
    By the above inequality, {the right hand side of \eqref{convergent series for the second term} converges.
Therefore, we see that in all the four cases 
\[
\sum_{k=1}^{\infty}\P(\ci_k^{\delta}) \leq %\P \left( \frac{T_{n_{k-1}}^*-4n_{k-1}}{g_-(n_{k-1})} \leq -(6\beta)^{1/3}-\frac \delta2\right)+\P(\cb_k^*)
\sum_{k=1}^{\infty}\P\left(\ca_k^* \right)+\sum_{k=1}^{\infty}\P(\cb_k^*) < \infty.
\]
This completes the proof for the $\liminf$ lower bounds.
\qed
\subsection{Upper Bounds} In this subsection we will prove upper bounds for the $\liminf$s. As mentioned before note that an upper bound for the $\liminf$ for the point-to-point case was proved in \cite[Proposition 2.2]{BGHK21}. But the constant obtained was not optimal. Here we obtain the correct constant.\\
\paragraph{\textbf{Proof of liminf Upper Bound in the Point-to-Point Case:}} The proof follows same idea as in the proof of \cite[Proposition 2.2]{BGHK21}. The difference will be that unlike the point-to-line passage time estimates used in the proof of \cite[Proposition 2.2]{BGHK21}, we will use point-to-point passage time estimates using a transversal fluctuation argument. This will prove the optimal constant. In this proof in the definition of the point-to-point passage time initial vertex will be removed. We first fix $\varepsilon>0$ small. We define $n_k:=[(k!)^{(1-\varepsilon)^3}]$. We will show that there exists a positive probability set with probability $> \gamma$ for some $\gamma>0$, such that on this event for all $m$ sufficiently large there exists $\frac m2 \leq k \leq m$ such that 
    \[
    \frac{T_{n_k}^{\ptop{}}-4n_k}{g_-(n_k)} \leq -(12)^{1/3}.
    \]
    \begin{figure}[t!]
    \includegraphics[width=13 cm]{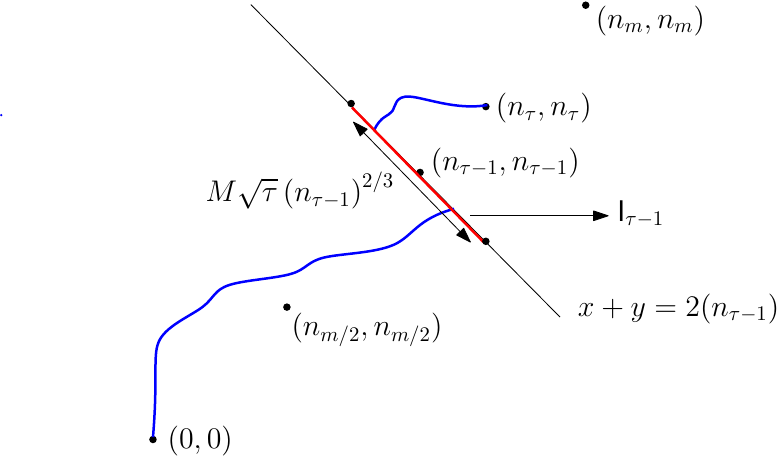}
    \caption{To prove $\liminf$ upper bound for point-to-point passage time we consider the sequence $(n_k)_{k \geq 1}$ as defined in the proof. $\mathsf{TF}$ is the event that all the geodesics $\Gamma_k$ from $\bo{0}$ to $\bo{n_k}$ does not go outside of the interval $\mathsf{I_{k-1}}$. We use \cite[Proposition 2.1]{BBB23} to ensure that we can choose $M$ sufficiently large to make $\P\left (\mathsf{TF}^c \right)$ sufficiently small. Further, $\cm$ is the event that between $\frac m2$ and $m$ there is a $\tau$ such that the passage time from $I_{\tau-1}$ to $\bo{n_\tau}$ is smaller than $4(n_\tau-n_{\tau-1})-(12)^{1/3}g_-(n_\tau)$. The above figure shows a geodesic between $\mathsf{I}_{\tau-1}$ to $\bo{n_\tau}$. Using independence and Lemma \ref{typical interval to point passage time} we get that $\P(\cm^c)$ can be made sufficiently small for large $m$. Finally, $\cn_k$ is the event that the passage time from $\bo{0}$ to $\cl_{2n_{k-1}}$ has passage time smaller than $4n_{k-1}$. The point-to-line geodesic from $\bo{0}$ to $\cl_{2n_{k-1}}$ is shown in the figure. Using weak convergence for each fixed $k, \cn_k$ has positive probability. Finally, on the event $\mathsf{TF} \cap \cm \cap \cn_\tau $ we get $T_{\bo{0},n_{\tau}} \leq 4n_{n_\tau}-(12)^{1/3}g_-(n_\tau)$. We use conditioning on $\tau$ to find a positive lower bound for the probability $\mathsf{TF} \cap \cm \cap \cn_\tau $.  } 
    \label{fig: liminf upper bound point to point} 
\end{figure}
By Kolmogorov 0-1 law this will prove the theorem. We fix $M$ large enough (to be chosen later). First we consider the following event.
\[
\mathsf{TF}:=\left \{ \text{ for all } k \geq 1,  \left |\psi \left (\Gamma_{n_{k}}^{\ptop{}}(n_{k-1}) \right) \right| \leq M\sqrt{k}(n_{k-1})^{2/3} \right \}.
\]
Using \cite[Proposition 2.1]{BBB23} we have there exists $a>0$ such that 
\begin{equation}
\label{estimate for tranversal fluctuation event}
\P (\mathsf{TF}^c) \leq \sum_{k=1}^{\infty}e^{-aM^3k^{3/2}}.
\end{equation}
Now, we consider the following events. For $k \geq 1,$ let $\mathsf{I}_{k-1}$ denote the interval $\{v \in \cl_{2n_{k-1}}: |\psi(v)| \leq M \sqrt{k}(n_{k-1})^{2/3} \}$ (see Figure \ref{fig: liminf upper bound point to point}). For $k \geq 1$ we define 
\[
\cm_k:= \left \{ T_{\mathsf{I_{k-1}}, \bo{n_k}}-4(n_k-n_{k-1}) \leq -(12)^{1/3}g_-(n_k)\right \},
\]
where 
\[
T_{\mathsf{I}_{k-1}, \bo{n_k}}:=\max_{u \in \mathsf{I}_{k-1}}T_{u, \bo{n_k}}.
\]
Further we consider the events $\cn_k$ defined by (see Figure \ref{fig: liminf upper bound point to point}) 
\[
\cn_k:=\{T_{\bo{0}, \mathsf{I}_{k-1}} \leq 4n_{k-1} \},
\] 
where similarly, 
\[
T_{\bo{0}, \mathsf{I}_{k-1}}:=\max_{u \in \mathsf{I}_{k-1}}T_{\bo{0},u}.
\]
Consider the event $\cm:=\bigcup_{k=\frac m2}^m \cm_k.$ We fix a small $\gamma$ (to be chosen later). We want to show that for sufficiently large $m, \P(\cm) \geq 1-\frac \gamma 4.$ Clearly, $\cm_k$ are independent events. First we fix $M$. Using \eqref{estimate for tranversal fluctuation event} we can chose $M$ so that 
\[
    \P ( \mathsf{TF}) \geq 1 -\frac \gamma4.
\]
To estimate the probability of the events $\cm_k,$ we apply the following lemma. The lemma says that for small enough (at the scale of $n^{2/3}$) interval, the interval-to-point passage time tail has the same exponent (upto $\vep$ error) as the point-to-point passage time. We will prove the following lemma in Section \ref{s: LPP estimates}.\\ 
First we fix some notations. We consider the parallelogram $U_\delta$ whose opposite pair of sides lie on $\cl_0$ and \ $\cl_{2n}$ with midpoints $\bo{0}$ and $\bo{n}$ respectively and each side have length $\delta n^{2/3}$. Further, $U_\delta \cap \cl_0$ (resp.\ $U_\delta \cap \cl_{2n}$) are denoted by $U_\delta^0$ (resp.\ $U_\delta^n$). 
\begin{lemma}
\label{typical interval to point passage time}
    For any $\vep>0$ small enough, there exists $\delta>0, C>0$ such that for $n$ and $x$ sufficiently large and $x \ll n^{1/9}$
    \[
    \P \left( \sup_{u \in U_\delta^0, v \in U_\delta^n} T_{u,v} \leq 4n -
    x 2^{4/3}n^{1/3}\right) \geq e^{-(\frac{1}{12}+\vep)x^3}.
    \]
\end{lemma}
We chose $\delta$ corresponding to $ \frac{\varepsilon}{12}$ as in Lemma \ref{typical interval to point passage time}. Now, we chose $k$ large enough so that 
\[
    M \sqrt{k}(n_{k-1})^{2/3} \leq \delta (n_k-n_{k-1})^{2/3}.
\]
\begin{comment}
To estimate the probability of the events $\cm_k,$ we apply the following lemma. The lemma says that for small enough (at the scale of $n^{2/3}$) interval, the interval to point passage time tail has the same exponent (upto $\vep$ error) as the point to point passage time. We will prove the following lemma in Section \ref{s: LPP estimates}.
\end{comment}
Then using Lemma \ref{typical interval to point passage time} we have 
\[
\P ( \cm_k ) \geq e^{-(1+\vep)\frac{ n_k \log \log n_k}{n_k-n_{k-1}}}.
\]
Further, for sufficiently large $k,$ we have 
\[
\frac{n_k}{n_{k}-n_{{k-1}}}<1+\varepsilon.
\]
So, for sufficiently large $k$ we have 
\[
\P (\cm_k) \geq e^{-(1+\varepsilon)^2 \log \log n_k}.
\]
Now, we observe that we have
\[
\lim_{k \rightarrow \infty} \frac{\log \log k!}{\log k}=1.
\]
Hence, we get for sufficiently large $k$
\[
\log \log k! <\frac{\log k}{1-\varepsilon}.
\]
Combining the above we get 
\[
\P (\cm_k) \geq e^{-(1-\varepsilon^2)^2 \log k}.
\]
Now, as $\cm_k$ are independent we have 
\[
\P \left (\bigcap_{k=\frac m2}^m \cm_k^c \right) \leq \prod_{k=\frac m2}^m\left (1-\frac{1}{k^{(1-\varepsilon^2)^2}} \right ).
\]
Clearly, the right hand side goes to zero as $m$ goes to infinity. So, we can choose $m$ large enough so that 
\[
\P \left( \bigcup_{k=\frac m2}^m \cm_k\right) \geq 1-\frac{\gamma}{4}.
\]
%Similar to the proof of \cite[Proposition 2.2]{BGHK21} 
We define the following random variable.
\[
\tau:=\max \left \{ \frac m2 \leq k \leq m: \cm_k \text{ occurs } \right \},
\]
where we define $\tau$ to be $-\infty$ if the above set is empty. We have on $\cm, \tau$ is finite. Further, we define 
\[
\cn_\tau:=\{T_{\bo{0}, \mathsf{I}_{\tau-1}} \leq 4 n_{\tau-1}\}.
\]
For $\frac m2 \leq k \leq m,$ let $\cg_k:=\sigma \{ \tau_v: \phi(v) > n_{k-1}\}.$
We wish to calculate $\P \left ( \mathsf{TF} \cap \cm \cap \cn_{\tau} \right).$ Note that for all $\frac m2 \leq k \leq m$, %the events $\{\tau \leq k\}$ and $ \{\tau \geq k \} \in \cg_k.$ 
$\{\tau = k\} \in \cg_k.$ Now, on the event $\{\tau=k\},$ we have 
\[
\P (\cn_\tau \vert \tau=k)=\P( \cn_k \vert \tau=k)=\P( \cn_k),
\]
as $\cn_k$ is independent of $\cg_k.$ Also, note that 
\[
\cn_k \supset \{ T_{n_{k-1}}^{\ptoline{}} < 4n_{k-1}\}.
\]
Now, using the weak convergence of point-to-line passage time, we conclude there exists $c>0$ such that for all sufficiently large $k$
\[
\P( \cn_k) \geq c.
\]
We fix $\gamma:=c.$ Now, consider the event $\mathsf{TF} \cap \cm \cap \cn_{\tau}.$ We have for all $\frac m2 \leq k \leq m,$
\[
\P(\mathsf{TF} \cap \cm \cap \cn_{\tau}\vert \tau=k) \geq \P(\mathsf{TF} \cap \cm \vert \tau=k)+\P(\cn_k)-1\geq \P(\mathsf{TF} \cap \cm \vert \tau=k)+\gamma-1.
\]
Combining all the above we get 
\begin{align*}
\P(\mathsf{TF} \cap \cm \cap \cn_{\tau})&=\sum_{k=\frac m2}^m\P \left(\mathsf{TF} \cap \cm \vert \tau=k \right)\P(\tau=k)\\ 
&\geq \sum_{k=\frac m2}^m\P(\mathsf{TF} \cap \cm \vert \tau=k )\P ( \tau=k)+\left(\gamma-1\right)(1-\P\left(\tau=-\infty \right)) \geq \frac \gamma4,
\end{align*}
where the last inequality comes from the fact that $\P\left(\tau=-\infty \right) \leq \P(\cm^c)$ and we can choose $m$ large enough so that $\P(\cm^c)$ is small.
Further, observe that on $\mathsf{TF} \cap \cm \cap \cn_{\tau},$ we have (see Figure \ref{fig: liminf upper bound point to point})
\[
T_{n_{\tau}}^{\ptop{}} \leq T_{\bo{0}, \mathsf{I}_{\tau-1}}+T_{\mathsf{I}_{\tau-1}, \bo{n_\tau}} \leq 4n_{\tau-1}+4(n_{\tau}-n_{\tau-1})-(12)^{1/3}g_-(n_\tau).
\]
So, for $n$ sufficiently large we define the event 
\[
\mathcal{R}_{m}:= \bigcup_{k=\frac m2}^{m}\left \{\frac{T_{n_k}^{\ptop{}}-4n_k}{g_-(n_k)} \leq -(12)^{1/3} \right \}.
\]
Then we have for sufficiently large $m$
\[
\P(\mathcal{R}_m) \geq \frac \gamma4.
\]
Now,
\begin{align*}
&\P \left (\left\{ \liminf_{n \rightarrow \infty}\frac{T_n^{\ptop{}}-4n}{g_-(n)} \leq -(12)^{1/3}\right\} \right)\\
&=\P \left (\bigcap_{j=0}^{\infty} \bigcup_{n=j}^{\infty}\left \{\frac{T_n^{\ptop{}}-4n}{g_-(n)} \leq -(12)^{1/3}\right \} \right)\\
&=\lim_{j \rightarrow \infty}\P \left( \bigcup_{n=j}^{\infty}\left \{\frac{T_n^{\ptop{}}-4n}{g_-(n)} \leq -(12)^{1/3}\right \}\right)\\
%=\lim_{k \rightarrow \infty}\P \left( \bigcup_{n=(\frac k2)!^{(1-\varepsilon)^3}}^{\infty}\Big \{\frac{T_n-4n}{2^{4/3}n^{1/3}(\log \log n)^{1/3}} \leq -(12)^{1/3}\Big \}\right)
&\geq \liminf_{k \rightarrow \infty} \mathcal{R}_k \geq \frac \gamma4.
\end{align*}
This completes the proof. \qed
\begin{figure}[t!]
    \includegraphics[width=13 cm]{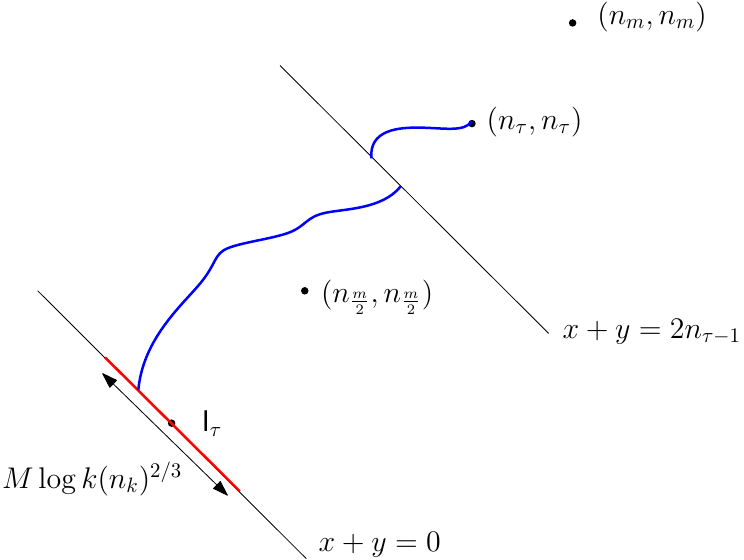}
    \caption{To prove $\liminf$ upper bound for the line-to-point case we first consider the sequence $(n_k)_{k \geq 1}$ as defined in the proof. $\mathsf{TF}$ is the event that all geodesics from $\cl_{0}$ to $\bo{n_k}$ intersect $\cl_0$ in $\mathsf{I}_k$. By \cite[Theorem 2.3]{BBF22} $\P(\mathsf{TF}^c)$ is arbitrarily small. The event $\cm$ ensures that there exists $\frac m2 \leq \tau \leq m$ such that $T_{\cl_{2n_{\tau-1}},\bo{n_\tau}} \leq 4(n_\tau-n_{\tau-1})-(6)^{1/3}g_-(n_\tau).$ By using independence and Theorem \ref{t: passage time tail estimates}, (ii) we can chose $m$ sufficiently large so that $\P(\cm^c)$ is arbitrarily small. Finally the event $\cn_k$ ensures that for all $k$ the passage time between $\mathsf{I}_k$ to $\cl_{2n_{k-1}}$ is not too large from $4n_{k-1}$. By Theorem \cite[Proposition A.5]{BBF22} we see that $\P(\cn_k)$ is small. Now, on $\mathsf{TF}\cap \cm \cap \cn_\tau$ we can choose $m$ large enough so that the line-to-point passage time is less than  $-(6^{1/3}-\vep)g_-(n_\tau).$ Finally a conditioning argument as used in the proof of $\liminf$ upper bound point-to-point case we get the a positive lower bound for the above event. } 
    \label{fig: liminf line to point }
\end{figure}
\paragraph{\textbf{Proof of liminf Upper Bound in the Line-to-Point Case:}}
    In the definition of line-to-point passage time and point-to-point passage time in this proof we will exclude the final vertex. Let $\rho>1$ and $M$ be large enough (to be chosen later). Let $n_k:= \rho^k$ for all $k \geq 1.$ Let $\varepsilon>0$ be any fixed positive number. As before we will show that there exists a set with probability at least $\gamma$ for some $\gamma>0,$ such that on this event for all $m \in \N$ sufficiently large there exists $ \frac m2 \leq k \leq m$ such that 
    \[
    \frac{T_{n_k}^{\linetop{}}-4n_k}{g_-(n_k)} \leq -(6)^{1/3}+\varepsilon.
    \]
    Showing this for all $\varepsilon>0$ and applying Kolmogorov 0-1 law will again conclude the theorem. For $k \geq 1,$ we define the following events. 
    \[
    \mathsf{TF}_k:=\left \{ \left |\psi \left ({\Gamma^{\linetop{}}_{n_k}}(0)\right) \right| \leq M \log kn_k^{2/3} \right\}.
    \]
    By \cite[Theorem 2.3]{BBF22} we have there exists $a>0$ such that 
    \[
    \P \left(\mathsf{TF}_k^c \right) \leq e^{a (M \log k)^2-\frac 43 ( M \log k)^3}.
    \]
    We define 
    \[
    \mathsf{TF}:= \bigcap_{k=1}^{\infty} \mathsf{TF}_k.
    \]
    Then we have 
    \[
    \P \left(\mathsf{TF} \right) \geq 1- \sum_{k=1}^{\infty}e^{a (M \log k)^2-\frac 43 ( M \log k)^3}.
    \]
    Hence, there exists $C>0$ (independent of $M$) such that 
    \begin{equation}
    \label{line to point transversal fluctuation happens with large probability}
    \P \left(\mathsf{TF} \right) \geq 1-Ce^{-\frac 43 M^3}.
    \end{equation}
    Now, for $k \geq 1,$ we define the following events. 
    \[
    \cm_k:=\left \{ T_{\cl_{2n_{k-1}},\bo{n_k}}-4(n_k-n_{k-1}) \leq -\left(6^{1/3}-\frac{\varepsilon}{2}\right)g_-(n_k)\right\},
    \]
    where
    \[
    T_{\cl_{2n_{k-1}},\bo{n_k}}:=\max_{u \in \cl_{2n_{k-1}}}T_{u, \bo{n_k}}.
    \]
    Now, we have by Theorem \ref{t: passage time tail estimates}, (ii)
    \[
    \P( \cm_k) \geq e^{-\frac 16(1+\vep') \left (6^{1/3}-\frac{\vep}{2} \right)^3\frac{n_k \log \log n_k}{n_k-n_{k-1}} }.
    \]
    Let 
    \[
    \vep':= 1- \sqrt{\frac{(6^{1/3}-\frac{\vep}{2})^3}{6}}.
    \]
    Now, we chose $\rho>1$ large enough so that 
    \[
    \frac{\rho}{\rho-1} < 1+\vep'.
    \]
    We get 
    \[
    \P( \cm_k) \geq e^{-(1-\vep'^2)^2\log \log n_k}.
    \]
    Let 
    \[
    \cm:=\bigcup_{k=\frac m2}^m \cm_k.
    \]
    As $\cm_k$ are independent we have 
    \[
    \P \left( \cm^c \right) \leq \prod_{k= \frac m2}^m \left( 1- \frac{e^{-(1-\vep'^2)^2\log \log \rho}}{k^{(1- \vep'^2)^2}}\right).
    \]
    Clearly, as $m$ goes to infinity, right hand side goes to 0. So, we can choose $m$ large enough so that 
    \[
    \P \left( \cm^c \right) \leq Ce^{-\frac 43 M^3},
    \]
    where $C$ is as in \eqref{line to point transversal fluctuation happens with large probability}.
    Let us define the following interval (see Figure \ref{fig: liminf line to point }). 
    \[
    \mathsf{I}_k:=\big \{ v \in \cl_0: |\psi(v)| \leq M n_k^{2/3}\log k  \}.
    \]
    For $k \geq 1,$ we define 
    \[
    \cn_k:=\big \{ T_{\mathsf{I}_k, \cl_{2n_{k-1}}} \leq 4n_{k-1} +M 2^{4/3}(\log k)^{\frac{1}{6}}n_{k}^{1/3}\big\},
    \]
    where
    \[
    T_{\mathsf{I}_k}, \cl_{2n_{k-1}}:=\max_{\substack{u \in \mathsf{I}_k, \\ v \in \cl_{2n_{k-1}}}}T_{u,v}.
    \]
    We wish to estimate $\P \left(\cn_k \right)$. For $1 \leq i \leq M\rho^{2/3} \log k$ we divide $\msf{I}_k$ into intervals of length $2n_{k-1}^{2/3}$ and denote them by $\msf{I}_k^i.$ 
    We define 
    \[ \cn_k^i:=\big\{  T_{\mathsf{I}_k^i, \cl_{2n_{k-1}}} \leq 4n_{k-1} +M 2^{4/3}(\log k)^{\frac{1}{6}}n_{k}^{1/3}\big \}.
    \]
    \[
    \left(\cn_k\right)^c=\bigcup_{i=1}^{M\rho^{2/3} \log k }\left(\cn_k^i\right)^c.
    \]
    For each $i$ as above we have by \cite[Proposition A.5]{BBF22} there exists $C>0$ such that 
    \[
    \P \left( \left( \cn_k^i\right)^c\right) \leq C (\log k)^{1/6}e^{-\frac 43 M^{3/2} \rho^{1/2}(\log k)^{1/4}}.
    \]
    Then by a union bound we have 
    \[
    \P \left( \left( \cn_k\right)^c\right) \leq C M\left(\log k\right)\rho^{2/3}  (\log k)^{1/6}e^{-\frac 43 M^{3/2} \rho^{1/2}(\log k)^{1/4}}.
    \]
    Hence, we have for sufficiently large $k,$
    \[
    \P \left( \left( \cn_k\right)^c\right) \leq \frac{Me^{-\frac 43 M^{3/2}}}{2}.
    \]
    So, we have for sufficiently large $k$
    \[
    \P \left ( \cn_k \right) \geq 1-\frac{Me^{-\frac 43 M^{3/2}}}{2}.
    \]
    Same as before, we now define the following random variable.
    \[
    \tau:=\max \left \{ \frac m2 \leq k \leq m: \cm_k \text{ occurs }\right \},
    \]
where we define $\tau$ to be $-\infty$ if the above set is empty. We have on $\cm, \tau$ is finite. Further, we define

\[
\cn_{\tau}:=\big \{ T_{\mathsf{I}_{\tau}, \cl_{2n_{\tau-1}}} \leq 4n_{\tau-1} +M 2^{4/3}(\log \tau)^{\frac{1}{6}}n_{\tau} ^{1/3}\big\}.
\]
We now, consider the event $\msf{TF} \cap \cm  \cap \cn_{\tau}$. Observe that on this event, $\tau$ is finite and we have 
\[
T^{\linetop{}}_{n_{\tau}} \leq T_{\mathsf{I}_{\tau}, \cl_{2n_{\tau}-1}}+T_{2\cl_{n_{\tau}-1}, \bo{n_{\tau}}} \leq 4n_{\tau} +M2^{4/3} (\log \tau)^{\frac{1}{6}}n_{\tau} ^{1/3}-\left (6^{1/3}-\frac \vep2 \right)g_-(n_\tau).
\]
Finally, observe that depending on $M,$ there exists $m_0$ such that for all $m \geq m_0,$ we have 
\[
T_{n_{\tau}} \leq 4n_{\tau}- \left (6^{1/3}-\varepsilon\right)g_-(n_\tau).
\]
Rest of the proof is similar once we show that for sufficiently large $k,$
\[
\P \left( \msf{TF} \cap \cm \cap \cn_\tau\right) \geq 1-2Ce^{-\frac 43 M^3}-\frac{Me^{-4/3M^{3/2}}}{2} \left(1-Ce^{-\frac 43M^3} \right).
\]
But this follows by same argument using independence and conditioning on $\tau$ as done in the previous case. This completes the proof. \qed
\paragraph{\textbf{Proof of Upper Bound of the liminf for Point-to-Line Case:}}
\begin{comment}
\textcolor{blue}{
        The proof follows similar idea as the point to point case. We choose $n_k$ as in the point to point case. We will again show there exists a positive probability set with probability $> \gamma$ for some $\gamma>0$, such that on this event for all $n$ sufficiently large there exists $\frac m2 \leq k \leq m$ such that
        }
    \textcolor{blue}{\[
    \frac{T_{n_k}^{\ptoline{}}-4n_k}{g_-(n_k)} \leq -(6)^{1/3}.
    \]
    }
    \textcolor{blue}{
    Then again Kolmogorov 0-1 law will conclude the theorem. We consider the following event. 
    \[
\mathsf{TF}:=\left \{ \text{ for all } k \geq 1,  \left |\psi\left(\Gamma^{\ptoline{}}_{n_{k}}(n_{k-1})\right) \right| \leq M\sqrt{k}(n_{k-1})^{2/3} \right \}.
\]
}
To estimate $\P(\mathsf{TF})$ we apply the following lemma about local transversal fluctuation of point to line geodesic. Note that this is analogous to \cite[Proposition 2.1]{BBB23}, which we used in the point to point case. The proof will be done in Section \ref{s: LPP estimates}. 
\end{comment}
The proof is exactly similar to the point-to-point case using the following two lemmas. The first lemma is about local transversal fluctuation of point-to-line geodesic, whose point-to-point version was proved in \cite[Proposition 2.1]{BBB23}.
\begin{lemma}
    \label{local transversal fluctuation for point to line geodesic}
     For all $0 <r \le 2n$ and $1< u \ll n^{1/14}$ sufficiently large there exists $C,c_1,c_2>0$ such that 
    \[
    \P\left (\left |\psi\left(\Gamma^{\ptoline{}}_{n}(r)\right) \right| \geq u r^{2/3} \right) \le Ce^{c_1u^2-c_2 u^3}.
    \]
\end{lemma}

\begin{comment}\textcolor{blue}{By Lemma \ref{local transversal fluctuation for point to line geodesic} we have there exists $C,c_1,c_2>0$ such that \begin{equation}
\label{estimate for tranversal fluctuation event for point to line}
\P (\mathsf{TF}^c) \leq \sum_{k=1}^{\infty}Ce^{c_1M^2k-c_2M^3k^{3/2}}.
\end{equation}
We again consider the following events. Recall $\mathsf{TYP}_{k-1}$ as defined in case of point to point. We define
\[
{\cm_k}:=\left \{ T_{\mathsf{I}_{k-1}, \cl_{n_k}}-4(n_k-n_{k-1}) \leq-\left(6\right)^{1/3} g_-(n_k) \right \},
\]
where 
\[
T_{\mathsf{I}_{k-1}, \cl_{n_k}}:=\max_{u \in \mathsf{I}_{k-1}}T_{u, \cl_{n_k}}.
\]
We consider the event ${\cm}:=\bigcup_{k=\frac m2}^m \cm_k.$ Recall $\cn_k$ and $\gamma$ as defined in the proof of point to point case. Same as before we choose $M$ large enough so that 
\[
\P( \mathsf{TF}) \geq 1- \frac \gamma4.
\]
}
\end{comment}
The second lemma says that for small enough interval (at the scale of $ n^{2/3}),$ the interval-to-line passage time tail has same estimate as the point-to-line passage time. Note that this is analogous to Lemma \ref{typical interval to point passage time}. Recall the notations as defined before Lemma \ref{typical interval to point passage time}.
\begin{lemma}
        \label{Interval to full line lemma}
        For any $\vep>0$ small enough there exists $\delta>0,C>0$ such that for $n$ and $x$ sufficiently large with $x \ll n^{1/9}$
        \[
        \P \left( \sup_{u \in U_\delta^0, v \in \cl_{2n} } T_{u,v} \leq 4n -
    x 2^{4/3}n^{1/3}\right) \geq Ce^{-(\frac{1}{6}+\vep)x^3}.
        \]
        \end{lemma}
We will prove both the lemmas in Section \ref{s: LPP estimates}. \qed
        \begin{comment}
        \textcolor{blue}{
Corresponding to $\frac \varepsilon6$ we choose $\delta$ as in Lemma \ref{Interval to full line lemma}. We choose $k$ large enough so that 
\[
M \sqrt{k}(n_{k-1})^{2/3} \leq \delta (n_k-n_{k-1})^{2/3}.
\]
Hence, by using Lemma \ref{Interval to full line lemma} we have 
\[
\P ({\cm_k}) \geq Ce^{-(\frac{1}{6}+\frac \varepsilon6)6\frac{n_k \log \log n_k}{(n_k-n_{k-1})}}.
\]
}
\textcolor{blue}{
By same argument as before we get for sufficiently large $k,$
\[
\P ({\cm_k}) \geq Ce^{-(1-\varepsilon^2)^2 \log k}.
\]
Using independence of ${\cm_k}$'s we can choose $n$ large enough so that 
\[
\P ({\cm}) \geq 1-\frac \gamma4.
\]
Rest of the proof is similar. We define 
\[
{\tau}:=\max \left\{ \frac m2 \leq k \leq m: {\cm_k} \text{ occurs }\right\},
\]
where we define ${\tau}$ to be $-\infty$ if the above set is empty. We have on ${\cm}, {\tau}$ is finite. Further, we define 
\[
{\cn_{{\tau}}}:=\left \{T_{\bo{0}, \mathsf{I}_{{\tau}-1}} \leq 4 n_{{\tau}-1}\right \}.
\]
We again use conditioning and independence to conclude 
\[
\P \left( {\mathsf{TF}} \cap {\cm} \cap {\cn_{{\tau}}}\right) \geq \frac \gamma4.
\]
and use this to conclude the required lower bound. To avoid repetition we omit the details here. \qed
}
\end{comment}
\paragraph{\textbf{Proof of liminf Upper Bound for the Half Space Case:}}
        We will again not write the proof in details, since the proof is exactly same as the point-to-point case using the following two lemmas which we will prove in Section \ref{s: LPP estimates}.\\
         Let $U_\delta$ denote the parallelogram whose sides are on $\cl_0$ (resp.\ $\cl_{2n}$) with endpoints $\bo{0}$ and \newline $(-\delta n^{2/3},\delta n^{2/3})$ (resp.\ $\bo{n}$ and $(n-\delta n^{2/3},n+\delta n^{2/3})$). $U_\delta^0$ and $U_\delta^n$ are defined as before. We have 
    \begin{lemma}
    \label{lemma: interval to line in half space}
    For any $\vep>0$ small enough there exists $\delta>0,C>0$ such that for $n$ and $x$ sufficiently large with $x \ll n^{1/9}$
     \[
        \P \left( \sup_{u \in U_\delta^0, v \in U_\delta^n } T^{\ptohalfspace{}}_{u,v} \leq 4n -
    x 2^{4/3}n^{1/3}\right) \geq Ce^{-(\frac{1}{24}+\vep)x^3}.
\]
\end{lemma}
\begin{lemma}
 \label{lemma: local transversal fluctuation for half space}
 For all $0 \leq r \leq 2n,$ we have there exists $C,c>0$ such that for all $n \geq 1, \ell$ sufficiently large 
 \[
 \P \left( \psi \left ( \Gamma_n^{\ptohalfspace{}}(r) \right ) \leq -\ell r^{2/3}\right) \leq Ce^{-c \ell}.
 \]
\end{lemma}
\begin{remark}
    In the above lemma also we expect $e^{-\ell^3}$ upper bound like the previous cases. But as we will see, since we are using the constrained passage time instead of the passage time in the half space model for the estimates, we get the weaker upper bound. Although, this is sufficient for our purpose.
\end{remark}

\section{Geometric estimates in LPP}
\label{s: LPP estimates}
% \textcolor{red}{collect here all requires estimates about geodesics and passage times.}
In this section we prove the LPP estimates which we have used to prove the LIL results. We start with proving Lemma \ref{lemma: infimum over a line for half space}.

\begin{comment}
\begin{lemma}
 \label{infimum over a line}
     There exist $C,c>0$ such that for all sufficiently large $n$ and $\theta>0$ sufficiently large we have 
     \[
     \P \left( \inf_{1 \leq m \leq n} \{ T_m^{\ptop{}}-4m\} \leq -\theta n^{1/3} \right) \leq Ce^{-c\theta^3}.
     \]
 \end{lemma}
 \end{comment}
 \begin{proof}[Proof of Lemma \ref{lemma: infimum over a line for half space}]
     The proof is an application of \cite[Theorem 4.2 (iii)]{BGZ21} and a union bound. 
     Consider the rectangle $U_n$ whose vertices of the longer sides are $\bo{0},\bo{n}$ and the shorter sides have length $n^{2/3}$ and the rectangle lies above the line $x=y$. We will use the coupling between the half space model and the exponential LPP model. Consider the usual exponential LPP mpdel. We define 
\[
T_n^{U_n}:=\max \{ \ell(\gamma): \text{$\gamma$ is an up-right path from $\bo{0}$ to $\bo{n}$ and } \gamma \subset U_n\}.
\]
Note that 
\[
\P \left(\inf_{1 \leq m \leq n}\l(T_m^{\ptohalfspace{}}-4m\r) \leq -\theta n^{1/3} \right) \leq \P \left ( \inf_{1 \leq m \leq n}\l(T_m^{U_m}-4m\r) \leq -\theta n^{1/3} \right).
\]
Now, by \cite[Theorem 4.2(iii)]{BGZ21} we have for sufficiently large $n(n \geq n_0),$
\begin{equation}
\label{eq: half space estimate for well seperated points}
\P \left( \inf_{\frac n2 \leq m \leq n} \l( T_m^{U_m}-\E T_m^{\ptop{}}\r) \leq -\theta n^{1/3} \right) \leq Ce^{-c\theta}.
\end{equation}
\begin{comment}From \cite[Theorem 4.2]{BGZ21} we get that there exists $n_0$ and $C,c>0$ such that for all $n \geq n_0$ we have
     \begin{equation}
     \label{small lim inf has small probability for well seprated points}
     \P \left( \inf_{\frac n2 \leq m \leq n} \{ T_m^{\ptop{}}-\E T_m^{\ptop{}}\} \leq -\theta n^{1/3} \right) \leq Ce^{-c\theta^3}.
     \end{equation}
     \end{comment}
\begin{comment}For $k \geq 1$ and $\frac{n}{2^k} \geq n_1$ we define the following events.
\[
\cj_k:=\Big \{\inf_{\frac {n}{2^{k}} \leq m \leq \frac{n}{2^{k-1}}} \{ T_m-4m\} \leq -\theta n^{1/3} \Big \}.
\]
\end{comment}
Also we have (see \cite[Theorem 2]{LR10}) there exists $C_1>0$ such that for all $n$
\begin{comment}for each $\delta >0$, there exists a positive constant $C_1$ (depending only on $\delta$) such that for all $m,n$ sufficiently large with $\delta < \frac{m}{n} < \delta^{-1}$ we have
\end{comment}
\begin{equation}
\label{expected_passage_time_estimate}
      \left|\mathbb{E}T_{n}^{\ptop{}}-4n\right| \leq C_1n^{1/3}.
\end{equation}
\begin{comment}
$(\sqrt{m}+\sqrt{n})^2$ will be called the \textit{time constant} in the direction $(m,n)$. The above inequality also holds for passage times when defined including the initial vertex.
\end{comment}
Using \eqref{expected_passage_time_estimate} we have, for $\theta>C_1$
\begin{equation}
\label{relation between 4n and expected passage time}
\big \{ T_n^{\ptohalfspace{}}-4n \leq -\theta n^{1/3} \big \} \subset \big \{ T_n^{\ptohalfspace{}}-\E T_n^{\ptop{}} \leq -(\theta-C_1) n^{1/3} \big \}.
\end{equation}
For $k \geq 1$ and $\frac{n}{2^k} \geq n_0$ we define the following events.
\[
\cj_k:=\Big \{\inf_{\frac {n}{2^{k}} \leq m \leq \frac{n}{2^{k-1}}} \{ T_m^{\ptohalfspace{}}-4m\} \leq -\theta n^{1/3} \Big \}.
\]
Then using \eqref{eq: half space estimate for well seperated points} and \eqref{relation between 4n and expected passage time} we have for $k$ as above there exists $c>0$ such that 
\[
\P( \cj_k ) \leq Ce^{-c(\theta-C_1)2^{\frac{k-1}{3}}}.
\]
So, for $\theta>2C_1$ we have for some $c>0,$
\[
\P( \cj_k ) \leq Ce^{-c\theta2^{\frac{k-1}{3}}}.
\]
Further, we note that when $m < n_0$ for $\theta$ and $n$ sufficiently large 
\[
T_m^{\ptohalfspace{}}-4m > -4n_0 > -\theta n^{1/3}.
\]
Hence, combining the above 
\[
\bigg \{ \inf_{1 \leq m \leq n} \{ T_m^{\ptohalfspace{}}-4m\} \leq -\theta n^{1/3} \bigg \} \subset \bigg \{ \inf_{1 \leq m < n_0} \{ T_m^{\ptohalfspace{}}-4m\} \leq -\theta n^{1/3} \bigg \} \cup \bigcup_{k=1}^{\log_2 \frac{n}{n_0}} \cj_k.
\]
Finally we have 
\[
\P \left( \inf_{1 \leq m \leq n} \{ T_m^{\ptohalfspace{}}-4m\} \leq -\theta n^{1/3} \right) \leq \sum_{k=1}^{\infty}Ce^{-c \theta 2^{\frac{k-1}{3}}} \leq Ce^{-c \theta}.
\]
This completes the proof. 
 \end{proof}
 \begin{comment}
 \begin{lemma}
\label{lemma: infimum over a line for half space}There exists $C,c>0$ such that for all sufficiently large $n$ and $\theta>0$ sufficiently large we have
\[
    \P \left(\inf_{1 \leq m \leq n}\{T_m^{\ptohalfspace{}}-4m\} \leq -\theta n^{1/3} \right) \leq Ce^{-c \theta}.
\]
\end{lemma}
\end{comment}
\begin{comment}
\begin{proof}[Proof of Lemma \ref{lemma: infimum over a line for half space}] The proof is similar to the previous lemma except for the fact that we need constrained passage time. Consider the rectangle $U_n$ whose vertices of the longer sides are $\bo{0},\bo{n}$ and the shorter sides have length $n^{2/3}$ and the rectangle lies above the line $x=y$. We define 
\[
T_n^{U_n}:=\max \{ \ell(\gamma): \text{$\gamma$ is an up-right path from $\bo{0}$ to $\bo{n}$ and } \gamma \subset U_n\}.
\]
Note that 
\[
\left \{ \inf_{1 \leq m \leq n}\l(T_m^{\ptohalfspace{}}-4m\r) \leq -\theta n^{1/3} \right \} \subset \left \{ \inf_{1 \leq m \leq n}\l(T_m^{U_m}-4m\r) \leq -\theta n^{1/3} \right \}.
\]
Now, by \cite[Theorem 4.2]{BGZ21} we have for sufficiently large $n,$
\[
\P \left( \inf_{\frac n2 \leq m \leq n} \l( T_m^{U_m}-\E T_m^{\ptop{}}\r) \leq -\theta n^{1/3} \right) \leq Ce^{-c\theta}.
\]
Now, applying union bound as in the previous lemma we obtain the required upper bound.   
\end{proof}
\end{comment}
\begin{remark}\label{rem:ptptcubeexponent}
The above lemma is true for point-to-point passage time as well. The proof will be exactly same, except we will use \cite[Theorem 4.2(i)]{BGZ21}. Also, in this case we will get $Ce^{-c \theta^3}$ as upper bound. We get two different powers for point-to-point and half space case because, in the half space case we are using the constrained passage time instead of the passage time in the half space model. Although, we expect similar $e^{-c \theta^3}$ in Lemma \ref{lemma: infimum over a line for half space} also, we could not prove it without the passage time estimates in general directions for the half space model. However, the bound in Lemma \ref{lemma: infimum over a line for half space} is enough for our purpose to prove the lower bound for $\liminf$.
\end{remark}
\begin{comment}First we fix some notations. We consider the parallelogram $U_\delta$ whose opposite pair of sides lie on $\cl_0$ and \ $\cl_n$ with midpoints $\bo{0}$ and $\bo{n}$ respectively and each side have length $\delta n^{2/3}$. Further, $U_\delta \cap \cl_0$ (resp.\ $U_\delta \cap \cl_n$) are denoted by $U_\delta^0$ (resp.\ $U_\delta^n$). 
\begin{lemma}
\label{typical interval to point passage time}
    For any $\vep>0$ small enough, there exists $\delta>0, C>0$ such that for $n$ and $x$ sufficiently large and $x \ll n^{1/9}$
    \[
    \P \left( \sup_{u \in U_\delta^0, v \in U_\delta^n} T_{u,v} \leq 4n -
    x 2^{4/3}n^{1/3}\right) \geq Ce^{-(\frac{1}{12}+\vep)x^3}.
    \]
\end{lemma}
\end{comment}
\begin{figure}[t!]
    \includegraphics[width=15 cm]{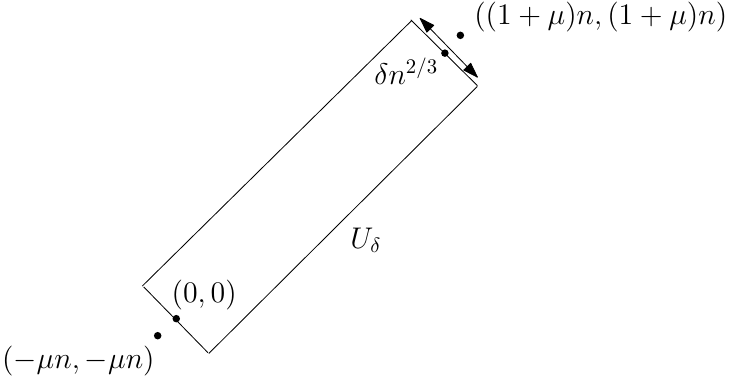}
    \caption{To prove Lemma \ref{typical interval to point passage time} we consider two points $-\bo{\mu n}$ and $\bo{(1+\mu)n}$ for small enough $\mu$. We consider three events $\cs_1,\cs_2,\ct.$ $\cs_1$ (resp.\ $\cs_2$) is the event that the passage time from $-\bo{\mu n}$ (resp.\ $\bo{(1+\mu)n}$) to $U_\delta^0$ (resp.\ $U_\delta^n$) is not too small. $\ct$ is the event that the passage time between $-\bo{\mu n}$ and $\bo{(1+\mu)n}$ is small. By super-additivity on $\cs_1 \cap \cs_2 \cap \ct,$ the passage time between $U_\delta^0$ and $U_\delta^n$ is small. We estimate the probabilities of the events $\cs_1,\cs_2$ using \cite[Theorem 4.2]{BGZ21}. Finally we apply Theorem \ref{t: passage time tail estimates}, (ii) to find the desired lower bound.}
    \label{interval to interval figure}
\end{figure}
\begin{proof}[Proof of Lemma \ref{typical interval to point passage time}]
    The proof is a simplified version of the proofs of \cite[Lemma 4.1]{BGZ21}, \cite[Proposition 1.12]{BBB23}. Depending on $\vep$ we will choose $\mu$ small enough. Now, we define $\delta:=\mu^{2/3}$. Consider the lines $\cl_{-2 \mu n}$ and $\cl_{2 (\mu+1)n}$ and consider the vertices $\bo{-\mu n}$ and $\bo{(\mu+1)n}$ (see Figure \ref{interval to interval figure}).\\ 
    We have (see \cite[Theorem 2]{LR10}) for each $\delta >0$, there exists a positive constant $C_1$ (depending only on $\delta$) such that for all $m,n$ sufficiently large with $\delta < \frac{m}{n} < \delta^{-1}$ we have
\begin{equation}
\label{expected_passage_time_estimate general direction}
      |\mathbb{E}T_{\bo{0},(m,n)}-(\sqrt{m}+\sqrt{n})^2| \leq C_1n^{1/3}.
\end{equation}
\begin{comment}
$(\sqrt{m}+\sqrt{n})^2$ will be called the \textit{time constant} in the direction $(m,n)$. The above inequality also holds for passage times when defined including the initial vertex.\\We define the following events. 
From \eqref{expected_passage_time_estimate} and \eqref{expected_passage_time_estimate general direction} and an calculus argument it follows that there exists $c_1>0$ and $C_2>0$ such that for all $u \in U_\delta^0, v \in U_\delta^n$ we have 
\begin{align*}\E T_{\bo{-\mu n,u}}+\E T_{u,v}+\E T_{v, \bo{\mu(n+1)}} \geq \E T_{\bo{-\mu n}, \bo{\mu(n+1)}}-c_1(\mu n)^{1/3}-C_2(\mu n)^{1/3}.
\end{align*}
We note that for all $u \in U_\delta^0, v \in U_\delta^n$\\
$T_{\bo{-\mu n},\bo{(\mu+1)n}}-4(1+2 \mu)n \geq T_{\bo{-\mu n}, u}+T_{u,v}+T_{v, \bo{(\mu+1)n}}-\E T_{\bo{-\mu n}, \bo{\mu(n+1)}}-C_1((1+2 \mu)n)^{1/3}\\ \geq T_{\bo{-\mu n}, u}+T_{u,v}+T_{v, \bo{(\mu+1)n}}-\E T_{\bo{-\mu n,u}}-\E T_{u,v}-\E T_{v, \bo{\mu(n+1)}}+c_1(\mu n)^{1/3}+C_2(\mu n)^{1/3}-C_1((1+2 \mu)n)^{1/3}$
\end{comment}
We note the following. By \eqref{expected_passage_time_estimate general direction} and a Taylor series argument we have there exists $C_2>0$ such that 
\begin{equation}
\label{M_1 expectation bound}
   \inf_{u \in U_{\delta}^0} \E T_{\bo{-\mu n}, u} \geq 4\mu n-C_2( \mu n)^{1/3}.
\end{equation}
Similarly, we have 
\begin{equation}
\label{M_2 expectation bound}
\inf_{u \in U_\delta^n} \E T_{u, \bo{(\mu+1)n}} \geq 4 \mu n-C_2 (\mu n)^{1/3}.
\end{equation}
We define the following events.
\begin{itemize}
    \item $\cs_1:=\big \{ \inf_{u \in U_\delta^0} T_{\bo{-\mu n}, u}-4 \mu n \geq -
    \frac {\vep}{4} 2^{4/3}x n^{1/3} \big \}$.
    \item $\cs_2:=\big \{ \inf_{u \in U_\delta^n} T_{u, \bo{(\mu+1)n}} -4 \mu n
    \geq -
    \frac {\vep}{4} 2^{4/3}x n^{1/3} \big \}.$
    \item $\ct:= \big \{ T_{\bo{-\mu n},\bo{(\mu+1)n}} \leq 4((1+2 \mu)n)-(1+\frac \vep2)2^{4/3}xn^{1/3}\big \}.$
    \end{itemize}
    Clearly, 
    \[
    \cs_1 \cap \cs_2 \cap \ct \subset \Big \{ \sup_{u \in U_\delta^0, v \in U_\delta^n} T_{u,v} \leq 4n -
    2^{4/3}x n^{1/3} \Big \}.
    \] 
    Further note that by \eqref{M_1 expectation bound} we have 
    \[
    \cs_1^c \subset \Big \{ \inf_{u \in U_\delta^0} T_{\bo{-\mu n}, u}-\E T_{\bo{-\mu n},u} \leq C_2(\mu n)^{1/3}-\frac {\vep}{4} 2^{4/3}x n^{1/3} \Big \}.
    \]
    Using \cite[Theorem 4.2]{BGZ21} we have for $n$ sufficiently large there exists $c>0$ such that 
    \[
    \P ( \cs_1^c) \leq e^{-\frac{c \vep^3x^3}{\mu}}.
    \]
    Now, we can chose $\mu$ small enough (depending on $\vep$) so that for all $x$ large enough the right hand side is upper bounded by $\frac{e^{-(\frac{1}{12}+\vep)x^3}}{4}.$ Similarly, we have 
    \[
    \P (\cs_2^c) \leq \frac{e^{-(\frac{1}{12}+\vep)x^3}}{4}.
    \]
    Further, we have for sufficiently large $n$ and $x \ll n^{1/9}$ sufficiently large we have by Theorem \ref{t: passage time tail estimates}, (ii)
    \[
    \P( \ct) \geq e^{-\frac{1}{12}(1+\vep)\frac{(1+\frac \vep2)^3}{1+2 \mu }x^3}.
    \]
    We also choose $\mu$ sufficiently small (depending on $\vep$) so that 
    \[
    \frac{(1+\vep)^4}{1+2 \mu} \leq 1+12\vep.
    \]
    Combining all the above we have 
    \[
    \P \left( \sup_{u \in U_\delta^0, v \in U_\delta^n} T_{u,v} \leq 4n -
    2^{4/3}x n^{1/3}\right) \geq \frac{e^{-(\frac{1}{12}+\vep)x^3}}{2}.
    \]\end{proof}
    \begin{comment}
    We have a similar lemma for point to line passage time as well. Recall the setup of Lemma \ref{typical interval to point passage time}. We have the following lemma.
    
    \begin{lemma}
        \label{Interval to full line lemma}
        For any $\vep>0$ small enough there exists $\delta>0,C>0$ such that for $n$ and $x$ sufficiently large with $x \ll n^{1/9}$
        \[
        \P \left( \sup_{u \in U_\delta^0, v \in \cl_{2n} } T_{u,v} \leq 4n -
    x 2^{4/3}n^{1/3}\right) \geq Ce^{-(\frac{1}{6}+\vep)x^3}.
        \]
    \end{lemma}
    \end{comment}
    \begin{proof}[Proof of Lemma \ref{Interval to full line lemma}]
        Proof is similar to the proof of Lemma \ref{typical interval to point passage time}. We again chose $\mu$ small enough depending on $\vep$ and set $\delta:=\mu^{2/3}.$ We consider $\cl_{-2 \mu n}$ and the vertex $\bo{-\mu n}.$
        \begin{itemize}
    \item $\cs:=\big \{ \inf_{u \in U_\delta^0} T_{\bo{-\mu n}, u}-4 \mu n \geq -
    \frac {\vep}{2} 2^{4/3}x n^{1/3} \big \}$.
    \item $\ct:= \big \{ T_{\bo{-\mu n},\cl_{2n}} \leq 4((1+\mu)n)-(1+\frac \vep2)2^{4/3}xn^{1/3}\big \}.$
    \end{itemize}
    Clearly,
    \[
    \cs \cap \ct \subset \big \{ \sup_{u \in U_\delta^0, v \in \cl_{2n} } T_{u,v} \leq 4n -
    x 2^{4/3}n^{1/3} \big \}.
    \]
    Rest of the proof is exactly similar to Lemma \ref{typical interval to point passage time}. We use \cite[Theorem 4.2]{BGZ21} and Theorem \ref{t: passage time tail estimates}, (ii) respectively to estimate $\P(\cs \cap \ct).$ This completes the proof.
    \begin{comment}
    
    Recall that we already showed
    \[
     \P ( \cs^c) \leq e^{-\frac{c \vep^3x^3}{\mu}},
    \]
    for some $c>0$.
    We choose $\mu$ small enough (depending on $\vep$) so that for all $x$ large enough 
    \[
    \P (\cs^c) \leq \frac{e^{-(\frac{1}{6}+\frac \vep2)x^3}}{2}.
    \]
    Further, we have for sufficiently large $n$ and $x \ll n^{1/9}$ sufficiently large by Theorem \ref{t: passage time tail estimates}(iv)
    \[
    \P( \ct) \geq e^{-\frac 16 \frac{(1+\frac \vep2)^4}{1+\mu} x^3}.
    \]
     We again choose $\mu$ small enough (depending on $\vep$) such that 
    \[
    \frac{(1+\frac \vep2)^4}{1+\mu} \leq 1+6 \vep.
    \]
    Combining all the above we get 
    \[
    \P \left( \sup_{u \in U_\delta^0, v \in \cl_{2n} } T_{u,v} \leq 4n -
    x 2^{4/3}n^{1/3}  \right) \geq \frac{e^{-(\frac 16+\vep)x^3}}{2}.
    \]
    \end{comment}
    \end{proof} 
    \begin{comment}Let $U_\delta$ denote the parallelogram whose sides are on $\cl_0$ (resp.\ $\cl_{2n}$) with endpoints $\bo{0}$ and $(-\delta n^{2,3},\delta n^{2/3})$ (resp.\ $\bo{n}$ and $(n-\delta n^{2/3},n+\delta n^{2/3})$). $U_\delta^0$ and $U_\delta^n$ are defined as before. We have 
    \label{lemma: interval to line in half space}
    For any $\vep>0$ small enough there exists $\delta>0,C>0$ such that for $n$ and $x$ sufficiently large with $x \ll n^{1/9}$
     \[
        \P \left( \sup_{u \in U_\delta^0, v \in U_\delta^n } T^{\ptohalfspace{}}_{u,v} \leq 4n -
    x 2^{4/3}n^{1/3}\right) \geq Ce^{-(\frac{1}{24}+\vep)x^3}
\]
\end{comment}
\begin{proof}[Proof of Lemma \ref{lemma: interval to line in half space}]
    Recall the notations defined before Lemma \ref{lemma: interval to line in half space}.
    \begin{itemize}
    \item $\cs_1:=\big \{ \inf_{u \in U_\delta^0} T^{\ptohalfspace{}}_{\bo{-\mu n}, u}-4 \mu n \geq -
    \frac {\vep}{4} 2^{4/3}x n^{1/3} \big \}$.
    \item $\cs_2:=\big \{ \inf_{u \in U_\delta^n} T^{\ptohalfspace{}}_{u, \bo{(\mu+1)n}} -4 \mu n
    \geq -
    \frac {\vep}{4} 2^{4/3}x n^{1/3} \big \}.$
    \item $\ct:= \big \{ T^{\ptohalfspace{}}_{\bo{-\mu n},\bo{(\mu+1)n}} \leq 4((1+2 \mu)n)-(1+\frac \vep2)2^{4/3}xn^{1/3}\big \}.$
    \end{itemize}
    For the first event we observe that 
    \[
    (\cs_1)^c \subset \Big \{\inf_{u \in U_\delta^0}T^{U_\delta}_{-\bo{\mu n},u} \leq 4 \mu n-\frac{\vep}{2}2^{4/3}xn^{1/3} \Big \}.
    \]
    We find an upper bound for $\P\left( \left(\cs_1\right)^c\right)$ using \cite[Theorem 4.2 (iii)]{BGZ21}. We get for sufficiently large $n$
    \[
    \P \left(\cs_1^c \right) \leq Ce^{-\frac{cx}{\mu^{1/3}}}.
    \]
   
    Same bound holds for $\P \left(\cs_2^c \right)$. We bound $\P\left(\ct\right)$ using Theorem \ref{t: passage time tail estimates}, (ii). Rest of the arguments are similar to those in proof of Lemma \ref{typical interval to point passage time} and Lemma \ref{Interval to full line lemma} and we omit the details.
\end{proof}
\begin{comment}
Next we prove local transversal fluctuation of point to line geodesic. The following lemma is a generalisation of \cite[Theorem 2.3]{BBF22} and analogue of \cite[Proposition 2.1]{BBB23} in case of point to line geodesics.
\end{comment}
\begin{comment}
 \begin{lemma}
    \label{local transversal fluctuation for point to line geodesic}
     For all $0 <T \le 2n$ and $1< u \ll n^{1/14}$ sufficiently large there exists $C,c_1,c_2>0$ such that 
    \[
    \P\left (\left |\psi\left(\Gamma^{\ptoline{}}_{2n}(T)\right) \right| \geq u T^{2/3} \right) \le Ce^{c_1u^2-c_2 u^3}.
    \].
\end{lemma}
\end{comment}
\begin{figure}[t!]
    \includegraphics[width=13 cm]{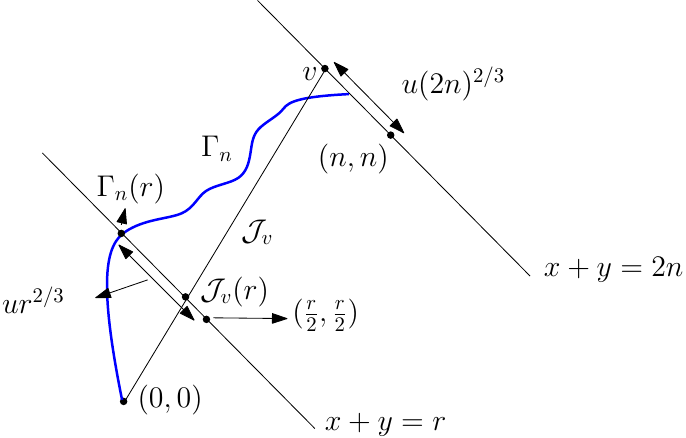}
    \caption{In the figure, the $\Gamma_{n}$(coloured in blue) is the point-to-line geodesic. To prove Lemma \ref{local transversal fluctuation for point to line geodesic}, we estimate the probability $\mathbb{P}\l(\psi\left (\Gamma_{n}(r)\r)\geq ur^{2/3}\right)$. First restrict to the event when $\Gamma_{n}$ does not have transversal fluctuation more than $u(2n)^{2/3}$ on $\cl_{2n}$. This is the event $\mathsf{TF}_{n}$. \cite[Theorem 2.3]{BBF22} implies that complement of this event has exponentially small (in $u$) probability.  Now, on the event $\mathsf{TF}_{n},$ if the geodesic $\Gamma_{n}$ has large transversal fluctuation on the line $\cl_r,$ from $x=y$ then the point-to-point geodesic $\Gamma_{\bo{0},v}$ will have large transversal fluctuation on the line $\cl_r$ from the line $\cj_v$ (this is a consequence of ordering of geodesics). \cite[Proposition 2.1]{BBB23} implies exponentially small upper bound for probability of this event.} 
    \label{fig: ptoline local transversal fluctuation} 
\end{figure}
\begin{proof}[Proof of Lemma \ref{local transversal fluctuation for point to line geodesic}]

    The proof will follow from \cite[Theorem 2.3]{BBF22} and \cite[Proposition 2.1]{BBB23}.  We will show for $1 < u \ll n^{1/14},\\ 
    \P\left(\psi\left (\Gamma^{\ptoline{}}_{n}(r)\right) \ge ur^{2/3}\right) \le Ce^{c_1u^2-c_2 u^3}$. The other case will follow by symmetry. Further, we can assume $r<n$. The cases when $r\geq n$ follows from \cite[Theorem 2.3]{BBF22}. We define the following event. 
    \[
    \mathsf{TF}_{n}:=\left \{0 \leq \psi \left(\Gamma^{\ptoline{}}_{n}(2n)\right) \leq u(2n)^{2/3}\right \}.
    \]
    We have  
    \[ 
    \left\{ \psi  \left(\Gamma^{\ptoline{}}_{n}(r)\right)\ge ur^{2/3} \right\} \subset (\mathsf{TF}_{n})^c \cup \left (\mathsf{TF}_{n} \cap \left \{ \psi\left (\Gamma^{\ptoline{}}_{n}(r)\right)\ge ur^{2/3} \right \} \right)  .
    \]
    Hence, by \cite[Theorem 2.3]{BBF22} we have for some constant $c>0$ and $1<u \ll n^{1/14}$
    \begin{equation}
    \label{estimate at n scale}
    \P\left (\psi\left (\Gamma^{\ptoline{}}_{n}(r) \right) \geq u r^{2/3} \right) \le e^{cu^2-\frac 43 u^3}+ \P \left( \mathsf{TF}_{n} \cap \left \{\psi \left (\Gamma^{\ptoline{}}_{n}(r)\right) \geq ur^{2/3} \right\} \right).
    \end{equation}
    We consider the second probability on the right hand side. We consider the vertex $v=(n+u (2n)^{2/3},n-u (2n)^{2/3}) \in \cl_{2n}$ (see Figure \ref{fig: ptoline local transversal fluctuation}). Consider $\Gamma_{v}$, the point-to-point geodesics from $\bo{0}$ to $v$. Now we use planarity of geodesics. Planarity is a property satisfied by the geodesics by which the geodesics that start from ordered vertices (in the space direction) stay ordered (in the space direction) throughout their journeys. This follows from the almost sure uniqueness of geodesics and from the fact that they do not form loops.  
    \[
    \mathsf{TF}_{n} \cap \left \{\psi \left (\Gamma^{\ptoline{}}_{n}(r)\right) \geq ur^{2/3}\right \} \subset \left \{ \psi \left (\Gamma_v(r) \right) \geq ur^{2/3} \right \}. 
    \]
    Further, consider the straight line $\cj_v$ joining $\bo{0}$ and $v$. Let $\cj_v(r)$ denote the intersection point of $\cj_v$ and $\cl_r$ (see Figure \ref{fig: ptoline local transversal fluctuation}). We have 
    \[
     \frac{\psi\left (\cj_v(r) \right)}{\frac u2 (2n)^{2/3}}=\frac{r}{2n}.
    \]
    Now, we observe if $\psi(\Gamma_v(r)) \geq u r^{2/3}$ we have 
    \begin{equation}
    \label{local transversal fluctuation of point to point geodesic}
    \psi \left (\Gamma_v(r) \right)-\psi\left (\cj_v(r) \right) \geq ur^{2/3}-\frac{r}{2n}u (2n)^{2/3} \geq \left (1-\frac{1}{2^{1/3}} \right )ur^{2/3}.
    \end{equation}
    Now, we apply \cite[Proposition 2.1]{BBB23} to $\Gamma_v$. Hence, using \eqref{local transversal fluctuation of point to point geodesic} and \cite[Proposition 2.1]{BBB23} there exists $C,c>0$ such that 
    \begin{equation}
    \label{estimate for point to point}
    \P \left( \psi \left (\Gamma_v(r) \right) \geq ur^{2/3}\right) \leq Ce^{-cu^3}.
    \end{equation}
    \eqref{estimate at n scale} and \eqref{estimate for point to point} prove the lemma.\end{proof}
    \begin{figure}[t!]
    \includegraphics[width=13 cm]{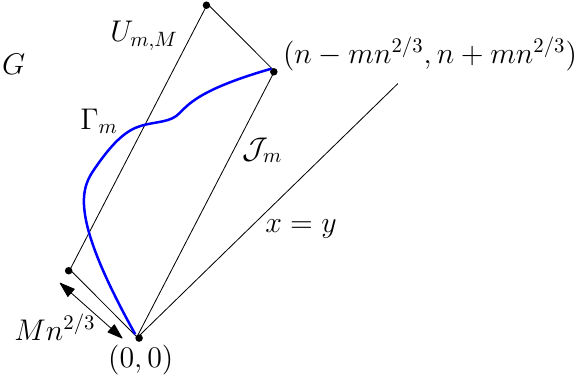}
    \caption{We consider the usual exponential LPP model. In the above figure $\Gamma_m$ is the maximising path between $\bo{0}$ and $\bo{n}_m$, where the maximum is taken over all up-right paths those don't intersect the region $\{(u,v) \in \Z^2: u>v\}.$ So, $\Gamma$ has the same distribution as the geodesic between $\bo{0}$ and $\bo{n}_m$ in the half space model. We want to estimate the probability of the event $\mathsf{LargeTF }(m,M,n)$, which is the event as shown in the figure. The above event can happen two ways. Either there is a path from $\bo{0}$ to $\bo{n}_m$ intersecting $G$ and has passage time not too small (this is the event $\ca)$. We use \cite[Proposition C.8]{BGZ21} to estimate the probability of $\ca$. On the complement of $\ca$, if $\Gamma_m$ intersects $G,$ then the half space passage time (hence also the restricted passage time $T_{\bo{0},\bo{n}_m}^{U_{m,M}}$) between $\bo{0}$ and $\bo{n}_m$ is small. We use \cite[Theorem 4.2(iii)]{BGZ21} to estimate this probability.} 
    \label{fig: halfspace transversal fluctuation} 
\end{figure}
To prove Lemma \ref{lemma: local transversal fluctuation for half space} first we prove the transversal fluctuation estimate in this case. For $m \geq 0 $, let $\bo{n}_{m}$ denote the vertex $(n-mn^{2/3},n+mn^{2/3})$ and let $\Gamma_m^{\ptohalfspace{}}$ denote the geodesic from $\bo{0}$ to $\bo{n}_m$ in the half space model. Let $\cj_m$ denote the straight line joining $\bo{0}$ and $\bo{n}_m$. Let $U_{m,M}$ denote the parallelogram whose opposite sides are on $\cl_0$ and $\cl_{2n}$ with endpoints $\bo{0}, \bo{n}_m, (-Mn^{2/3},Mn^{2/3}), \bo{n}_{m+M}$ respectively and these sides have lengths $2Mn^{2/3}$. So, $\cj_m$ is a side of the parallelogram (see Figure \ref{fig: halfspace transversal fluctuation}). Hence, the parallelogram does not intersect the region $\{(u,v) \in \Z: u>v\}$.
%\[
%\text{ LargeTF($m,M,n$)}:= \{\exists 0 \leq T \leq 2n \text{ such that }
%\psi\left(\Gamma_m^{\ptohalfspace{}}(T)\right) < \psi(v) \forall v \in U_{m,M}\}.\]
Let $\mathsf{LargeTF}(m,M,n)$ is the event that the geodesic $\Gamma^{\ptohalfspace{}}_m$ in the half space model intersect the region that lies above the opposite side of $\cj_m$ in $U_{m,M}$ (we will denote this region by $G$. See Figure \ref{fig: halfspace transversal fluctuation}).
\begin{lemma}
    \label{lemma: transversal fluctuation in half space}
    For $0< \gamma <1,$ and for all $0 <m \leq \gamma n^{1/3},$ we have there exists $C,c>0$ (depending only on $\gamma$) such that for all $M,n$ large enough (depending only on $\gamma$)
    \[
    \P \left( \mathsf{LargeTF }(m,M,n) \right) \leq Ce^{-c M}.
    \]
\end{lemma}
\begin{comment}
\begin{figure}[t!]
    \includegraphics[width=15 cm]{Half space transversal fluctuation.pdf}
    \caption{We consider the usual exponential LPP model. In the above figure $\Gamma_m$ is the maximising path between $\bo{0}$ and $\bo{n}_m$, where the maximum is taken over all up-right paths those don't intersect the region $\{(u,v) \in \Z^2: u>v\}.$ So, $\Gamma$ has the same distribution as the geodesic between $\bo{0}$ and $\bo{n}_m$ in the half space model. We want to estimate the probability of the event $\mathsf{LargeTF }(m,M,n)$, which is the event as shown in the figure. The above event can happen two ways. Either there is a path from $\bo{0}$ to $\bo{n}_m$ intersecting $G$ and has passage time not too small (this is the event $\ca)$. We use \cite[Proposition C.8]{BGZ21} to estimate the probability of $\ca$. On the complement of $\ca$, if $\Gamma_m$ intersects $G,$ then the half space passage time (hence also the restricted passage time $T_{\bo{0},\bo{n}_m}^{U_{m,M}}$) between $\bo{0}$ and $\bo{n}_m$ is small. We use \cite[Theorem 4.2(iii)]{BGZ21} to estimate this probability. } 
    \label{fig: halfspace transversal fluctuation} 
\end{figure}
\end{comment}
\begin{proof}
We will use the natural coupling between exponential LPP and the half space model. Consider the usual exponential LPP model. Let $\Gamma_m$ is the maximising path between $\bo{0}$ and $\bo{n}_m$, where the maximum is taken over all up-right paths those don't intersect the region $\{(u,v) \in \Z^2: u>v\}.$ So, $\Gamma_m$ has the same distribution as $\Gamma_m^{\ptohalfspace{}}$. Now the proof is a direct application of \cite[Proposition C.8 and Theorem 4.2 (iii)]{BGZ21}. As in \cite[Proposition C.8]{BGZ21} we define the following event.
    \[
    \ca:=\{ \exists \gamma: \bo{0} \rightarrow \bo{n}_m \text{ with } \gamma \cap U_{m,M} \neq \emptyset \text{ and }\ell(\gamma) \geq \E T_{\bo{0}, \bo{n_m}}-c_1 M^2n^{1/3}\},
    \]
    where $c_1$ is as in \cite[Proposition C.8]{BGZ21}. We have 
    \[
     \mathsf{LargeTF}(m,M,n) \subset \ca \cup \{ T_{\bo{0}, \bo{n}_m}^{U_{m,M}} \leq \E T_{\bo{0}, \bo{n}_m}-c_1 M^2 n^{1/3}\},
    \]
    where $T_{\bo{0}, \bo{n}_m}^{U_{m,M}}$ denotes the maximum passage time over all paths from $\bo{0}$ to $\bo{n}_m$ those are constrained to be in $U_{m,M}$.
    Now, applying \cite[Proposition C.8 and Theorem 4.2 (iii)]{BGZ21} we get
    \[
    \P \left( \mathsf{LargeTF}(m,M,n) \right) \leq e^{-c M^3}+Ce^{-c M},
    \]
    for some $C,c>0$ (depending on $\gamma$). This completes the proof of Lemma.\end{proof}
\begin{comment}
\begin{lemma}
 \label{lemma: local transversal fluctuation for half space}
 For all $0 \leq T \leq n,$ we have there exists $C,c>0$ such that for all $n \geq 1, \ell$ sufficiently large 
 \[
 \P \left( \psi \left ( \Gamma_n^{\ptohalfspace{}}(T) \right ) \leq -\ell T^{2/3}\right) \leq Ce^{-c \ell.}
 \]
\end{lemma} 
\end{comment}
\begin{proof}[Proof of Lemma \ref{lemma: local transversal fluctuation for half space}]
The Lemma follows from Lemma \ref{lemma: transversal fluctuation in half space} by applying a similar argument as \cite[Proposition 2.1]{BBB23}. As the proof is exactly similar we omit the details here.
\end{proof}

 \bibliography{intersection,main,main2}
 \bibliographystyle{plain}

\end{document}